\documentclass[Times1COL]{WileyNJDv5}

\usepackage{amssymb}
\usepackage{makeidx}
\usepackage{amsmath,amsfonts}
\usepackage{graphicx}
\usepackage{cite}
\usepackage{bm}

\newtheorem{example}{Example}

\articletype{Article Type}%
\received{6 February 2025}
\revised{7 June 2025}
\accepted{2 July 2025}
\journal{International Journal for Numerical Methods in Engineering}
\volume{126}
\copyyear{2025}
\startpage{1}
\begin{document}

\title{Product Of Exponentials (POE) Splines on Lie-Groups: Limitations, Extensions, and Application to $SO(3)$ and $SE(3)$}

\author[1]{Andreas M\"uller}

\authormark{M\"ULLER}
\titlemark{Product Of Exponentials (POE) Splines on Lie-Groups: Limitations,
Extensions, and Application to $SO(3)$ and $SE(3)$}

\address[1]{\orgdiv{Johannes Kepler University}, \orgname{Institute of
Robotics}, \orgaddress{\country{Austria}}}

\corres{Corresponding author Andreas M\"uller, \email{a.mueller@jku.at}}

\presentaddress{Johannes Kepler University, Altenberger Str. 69, 4040 Linz}

\abstract[Abstract]{Existing methods for constructing splines and B\'{e}zier curves on a Lie
group $G$ involve repeated products of exponentials deduced from local
geodesics, w.r.t. a Riemannian metric, or rely on general polynomials.
Moreover, each of these local curves is supposed to start at the identity of 
$G$. Both assumptions may not reflect the actual curve to be interpolated.
This paper pursues a different approach to construct splines on $G$. Local
curves are expressed as solutions of the Poisson equation on $G$. Therewith,
the local interpolations satisfies the boundary conditions while respecting
the geometry of $G$. A $k$th-order approximation of the solutions gives rise
to a $k$th-order product of exponential (POE) spline. Algorithms for
constructing 3rd- and 4th-order splines are derived from closed form
expressions for the approximate solutions. Additionally, spline algorithms
are introduced that allow prescribing a vector field the curve must follow
at the interpolation points. It is shown that the established algorithms,
where $k$th-order POE-splines are constructed by concatenating local curves
starting at the identity, cannot exactly reconstruct a $k$th-order motion.
To tackle this issue, the formulations are extended by allowing for local
curves between arbitrary points, rather than curves emanating from the
identity. This gives rise to a global $k$th-order spline with arbitrary
initial conditions. Several examples are presented, in particular the shape
reconstruction of slender rods modeled as geometrically non-linear Cosserat
rods.
}

\keywords{Splines, B{\'{e}}zier curves, De Casteljau
algorithm\color{black}, Lie groups, Poisson equation, approximations,
product of exponentials, $SE(3)$, $SO(3)$, rigid body motions, Cosserat
continua, Magnus expansion}

\jnlcitation{\cname{\author{M\"uller A}}.
\ctitle{Product Of Exponentials (POE) Splines on Lie-Groups: Limitations, Extensions, and Application to $SO(3)$ and $SE(3)$.} \cjournal{\it International Journal for Numerical Methods in Engineering} \cvol{2025}.}

\maketitle

\renewcommand\thefootnote{} \footnotetext{\textbf{Abbreviations:} POE,
product of exponentials.}

\renewcommand\thefootnote{\fnsymbol{footnote}} \setcounter{footnote}{1}

\section{Introduction}

%
Most systems with practical relevance evolve on Lie groups. Interpolation on
manifolds, and Lie groups in particular, are hence relevant in various
contexts, such as 3D animation \cite{Shoemake1985,Shoemake1987}, motion
planning of robots \cite{Huang2020,Sarker2020,Legnani2021,Tagliavini2023}
and unmanned aerial vehicles (UAV) \cite{Lovegrove2013,Dhullipalla2019},
interpolation of motion within numerical integration schemes for multibody
systems \cite{BauchauEppleHeo2008,HanBauchau2018,Ren2025}, and the geometric
modeling of Cosserat rods \cite{RuckerJonesWebster2010,JMR2024}.

Shoemake \cite{Shoemake1985} was the first who introduced an algorithm for
constructing B\'{e}zier curves of unit quaternions, i.e. curves on the
compact Lie group $Sp\left( 1\right) \cong SU\left( 2\right) $, and thus on $%
SO(3)$, for the purpose of animating spatial rotations. Key element is the
spherical linear interpolation (SLERP), which is the representation of a
rotation between two given attitudes, $g_{0},g_{1}\in SO(3)$, by the
exponential $\exp \left( \tau \boldsymbol{\xi }\right) $ on $SO\left(
3\right) $, with constant $\boldsymbol{\xi }=\log (g_{0}^{-1}g_{1})\in
so\left( 3\right) $ and path parameter $\tau \in \left[ 0,1\right] $. The so
defined curve corresponds to 1-parameter subgroup of rotations about a
constant axis. Moreover, in $SO\left( 3\right) $, as a compact Lie group,
the 1-parameter subgroups are geodesics. This suggests constructing B\'{e}%
zier curves by adopting the classical De Casteljau algorithm on Euclidean
vector spaces to compact Lie groups, where geodesics on the group replace
straight lines in Euclidean space. A De Casteljau algorithm for constructing
B\'{e}zier curves on $SO\left( 3\right) $ was introduced in \cite%
{ParkRavani-JMD1995} (this is summarized in appendix B). An algorithm to
construct 5th-order B\'{e}zier curves for quaternion interpolation, i.e. on $%
SU\left( 2\right) $, was presented in \cite{TanXingFanHong2018}. The
description of intermediated curves, which are subdivided in course of the
De Casteljau algorithm, by an exponential of the form above is valid on any
Lie group. This motivated introducing a De Casteljau algorithm on the
non-compact Lie group $SE\left( 3\right) $ of (rigid body) screw motions in 
\cite{GeRavani-JMD1994}, where screws are represented as dual quaternions.
The intermediate curves (screw motions) described by the exponential are
geodesics w.r.t. a left-invariant metric on $SE\left( 3\right) $. Working in
kinematic image space, a cubic interpolation of rigid body motions was
introduced in \cite{GeRavani-JMD1994-2}. The B\'{e}zier curves on $SE\left(
3\right) $ were also used to construct ruled surfaces from screw motions 
\cite{SprottRavani2002,Ding2002}. The construction of B\'{e}zier curves on
symmetric subspaces of $SO(3)$ was addressed in \cite{Wu-ARK2016}. Clearly,
the choice of metric, and its invariance properties, is crucial for
constructing meaningful B\'{e}zier curves. A thorough mathematical treatment
of De Casteljau algorithms on Riemannian spaces can be found in \cite%
{CrouchKunLeite1999}, and on B\'{e}zier curves and B-splines in \cite%
{Camarinha2001,Popiel2006,Huber2017}.

Instead of using the exponential map to describe curves in $SE\left(
3\right) $ in terms of canonical coordinates, in \cite{Selig2010} the Cayley
transformation \cite{RSPA2021} was employed. This gives rise to rational
curves. In \cite{RavaniMeghdari2004} rational Frenet-Serret motions in $%
SE\left( 3\right) $ were introduced by using rational curvature, and the
possibility of constructing rotation-minimizing motions was discussed. An
algorithm for constructing rational frame motions that minimize rotations
was presented in \cite{BartovnJuettlerWang2010}. The groups $SO\left(
3\right) $ and $SE\left( 3\right) $ can be modeled as algebraic varieties,
e.g. using Study's model \cite{HustySchroecker2010}, and as algebraic
groups, which shall provide a framework for developing algorithms to
construct rational curves. An exhaustive treatment of rational curves on $%
SO\left( 3\right) $ and $SE\left( 3\right) $ regarded as algebraic groups
can be found in \cite{LiYe2024}. A remarkable result presented in \cite%
{LiYe2024} is the generalization of Kempe's Universality theorem, which
states that any algebraic curve in the plane can be generated by a
mechanical linkage comprising rotational joints only, \cite%
{LiSchichoSchroecker2018} to homogeneous spaces.

Another approach to the interpolation on Lie groups is adopting the concept
of splines defined by piecewise polynomials. In \cite%
{ParkRavani1997,KangPark1998}, a 2-point interpolation on $SO\left( 3\right) 
$ was introduced, where the interpolation between two attitudes is $\exp
\left( p\left( \tau \right) \boldsymbol{\xi }\right) $, with a 3rd-order
polynomial $p\left( \tau \right) $ in $\tau $. Concatenating such 2-point
interpolations yields a cubic spline expressed as product of the
exponentials of the individual segments. Conceptually the same approach was
proposed in \cite{KimNam1995}. Again, such spline algorithms rely on a
distance metric, and an interpolation between two points should be
geodesics. A crucial point is the invariance of the particular Riemannian
metric when the Lie group is non-compact. This issue was addressed in \cite%
{CrouchLeite1991,CrouchLeite1995} with the principle of variations. For the
special case of interpolation on $SE\left( 3\right) $, the choice of metric
was investigated in \cite%
{ZefranKumar1996,ZefranKumar1998,ZefranKumarCroke-TRO1998,ZefranKumarCroke-IJRR1999}%
. Conditions for minimum acceleration and minimum energy curves in $SE\left(
3\right) $ were derived using a left-invariant metric, e.g. kinetic energy
metric. More general discussions on splines in curved spaces can be found in 
\cite{Noakes1989}, and for general symmetric spaces in \cite%
{BogfjellmoModinVerdier2018}. Curves with stationary acceleration on $%
SE\left( 3\right) $ were investigated in \cite{Selig-IMA2007}. It is to be
noticed that, due to lack of bi-invariant metric on $SE\left( 3\right) $,
stationary acceleration does not necessarily imply minimum acceleration.
Furthermore, the screw motion with canonical coordinates obtained with the
logarithm (as above) on $SE\left( 3\right) $, as used in the screw SLERP 
\cite{Kavan2006}, is not a geodesic \cite{DuffyARK1991}. In conclusion,
common to all interpolation methods, is that a certain form of the canonical
(exponential) coordinates is deduced from metric properties.

Instead of explicitly relying on metric properties, an alternative approach
is to construct interpolation curves as solutions of the Poisson equation on
a Lie group $G$, i.e. to reconstruct the curve in the group from a given
vector field in the corresponding Lie algebra $\mathfrak{g}$. The so
constructed interpolation naturally accounts for the geometry of $G$, which
is encoded in the Poisson equation, but additionally captures the particular
motion. In this paper, an approach to interpolation on Lie groups is
introduced making use of higher-order solutions of the Poisson equation
reported in \cite{ZAMM2010}. It does not rely on geodesics, i.e. on metric
properties, as the De Casteljau algorithms in \cite%
{Shoemake1985,GeRavani-JMD1994,ParkRavani-JMD1995,CrouchKunLeite1999}, nor
does it assume specific polynomial expressions, as the spline algorithms in 
\cite{ParkRavani1997,KangPark1998}. A somewhat related publication in this
direction is \cite{KryslEndres2005}, where approximate solutions were used
to derive a Newmark integration scheme on $SO\left( 3\right) $. The approach
presented in this paper shares concepts applied to the integration on Lie
groups, which can be viewed as interpolation problems when a flow on $G$ is
to be reconstructed \cite{Marthinsen1999}. Interpolation on $SE\left(
3\right) $ is an important topic for finite element formulations and Newmark
generalized $\alpha $ integration schemes, in general \cite%
{HanBauchau2018,Ren2025}. It should be mentioned that often the
interpolation derived from the logarithm on $SE\left( 3\right) $ is
incorrectly regarded as geodesic, e.g. in the finite element formulation in 
\cite{Ren2025}.%
%

The general interpolation problem, addressed in this paper, is to find a
curve $\bar{h}\left( t\right) ,t\in \left[ 0,T\right] $ in a
finite-dimensional connected Lie group $G$ through a given sequence $%
h_{0},h_{1}\ldots ,h_{n}\in G$, where $\bar{h}\left( t\right) $ has required
continuity and its derivatives satisfies certain boundary conditions. A more
specific problem is the interpolation when additionally samples of a vector
field, $\mathbf{v}_{0},\mathbf{v}_{1},\ldots ,\mathbf{v}_{n}\in \mathfrak{g}$%
, are given that must be met by the curve. In this paper, spline
interpolations are considered that are defined piecewise as%
\begin{equation}
\bar{h}_{i}\left( t\right) =h_{0}h_{1}\cdots h_{i-1}\exp \boldsymbol{\xi }%
_{i}^{\left[ k\right] }\left( t\right) ,\ t\in \left[ t_{i-1},t_{i}\right]
\label{hi}
\end{equation}%
where $\bar{h}_{i}\left( t_{i}\right) =h_{i}$, and $\boldsymbol{\xi }_{i}^{%
\left[ k\right] }\left( t\right) $ is a curve in the Lie algebra $\mathfrak{g%
}$ of $G$ defined as $k$th-order polynomial in $t\in \left[ t_{i-1},t_{i}%
\right] $. 
%
Such a curve will be referred to as a $k$th-order curve in $G$. 
%
Since (\ref{hi}) can be written as $\bar{h}_{i}\left( t\right) =h_{0}\exp 
\boldsymbol{\xi }_{1}^{\left[ k\right] }\left( t_{1}\right) \cdots \exp 
\boldsymbol{\xi }_{i-1}^{\left[ k\right] }\left( t_{i-1}\right) \exp 
\boldsymbol{\xi }_{i}^{\left[ k\right] }\left( t\right) $, it is referred to
as \emph{product of exponentials (POE) spline}. The special case of cubic
splines was reported in \cite{ParkRavani1997,KangPark1998}, where $%
\boldsymbol{\xi }_{i}^{\left[ 3\right] }\left( t\right) $ are 3rd-order
polynomials in $t$, with $\boldsymbol{\xi }_{i}^{\left[ 3\right] }\left(
t_{i-1}\right) =\mathbf{0}$. Further, globally defined $k$th-order splines
of the form%
\begin{equation}
\bar{h}\left( t\right) =\bar{h}_{0}\exp \boldsymbol{\xi }^{\left[ k\right]
}\left( t\right) ,\ t\in \left[ 0,T\right]  \label{hGlob1}
\end{equation}%
are introduced, with non-zero initial value $\boldsymbol{\xi }^{\left[ k%
\right] }\left( 0\right) \neq \mathbf{0}$, and thus $\bar{h}\left( 0\right)
\neq \bar{h}_{0}$. This enables interpolations of general curves with
parameterized initial value. The contributions of this paper are summarized
as follows:%
%

\begin{itemize}
\item The coordinate functions $\boldsymbol{\xi }_{i}^{\left[ k\right]
}\left( t\right) $ in (\ref{hi}) are derived consistently as $k$th-order
approximations of the exact solution of the Poisson equation on $G$ (section
2).

\item Closed form expressions for the 3rd- and 4th-order 2-point
interpolation with given initial and boundary values on the left trivialized
derivative $\bar{\mathbf{v}}\left( t\right) =\bar{h}\left( t\right) ^{-1}%
\bar{h}^{\prime }\left( t\right) \in \mathfrak{g}$ are presented (section
3). The local order of accuracy of the $k$th-oder interpolation is analyzed.

\item An algorithm for constructing 3rd- and 4th-order POE-splines through $%
h_{0},h_{1},\ldots ,h_{n}\in G$ is introduced (section 4.1).

\item A 3rd- and 4th-order POE-spline algorithm is presented that admits
prescribing values of the vector field $\bar{\mathbf{v}}\left( t\right) =%
\bar{h}\left( t\right) ^{-1}\bar{h}^{\prime }\left( t\right) \in \mathfrak{g}
$ at interpolation points (section 4.2). This enables prescribing terminal
conditions on $\bar{h}^{\prime }$, which is not possible in existing spline
formulations.

\item A higher-order 2-point interpolation that admits non-zero initial
values of the canonical coordinates (section 6).

\item An algorithm for constructing globally parameterized $k$th-order
splines of the form (\ref{hGlob1}) is proposed (section 7). Such splines can
exactly reconstruct a $k$th-order curve in $G$, i.e. a curve $h\left(
t\right) =h_{0}\exp \xi \left( t\right) $, with $\xi \left( t\right) $ of
degree $k$ in $t$. More importantly, they allow for non-zero initial values
of the coordinate function $\boldsymbol{\xi }^{\left[ k\right] }\left(
t\right) $ in (\ref{hGlob1}). The global spline is thus symmetric in the
sense that it admits non-zero terminal as well as initial values.

\item The 3rd- and 4th-order splines are applied to the problem of
reconstructing the shape of a Cosserat rod described as curve in $SE\left(
3\right) $, which is particularly important in the area of continuum and
soft robotics (examples 1 and 5).

\item A De Casteljau algorithm is summarized in appendix B, and results are
compared with those of the POE splines.
\end{itemize}

%

\section{Reduction and Reconstruction on Lie Groups}

\subsection{Poisson Equation}

Given a curve $g\left( \tau \right) \in G$ with (normalized) path parameter $%
\tau \in \left[ 0,1\right] $, left translating the vector field $g^{\prime
}\in T_{g}G$ to the identity $e\in G$ yields the left-invariant vector field 
$\mathbf{v}=g^{-1}g^{\prime }\in \mathfrak{g}$. This is known as (kinematic)
reduction \cite{MarsdenBook1995,BlochBook2003}. Prominent examples in rigid
body kinematics and continuum mechanics are rotation group $G=SO\left(
3\right) $ and the group of rigid body motions $G=SE\left( 3\right) $. In
these cases, $\mathbf{v}\in so\left( 3\right) $ is the body-fixed angular
velocity, and $\mathbf{v}\in se\left( 3\right) $ is the body-fixed twist,
respectively \cite{MurrayBook,SeligBook}. The reconstruction problem is to
determine $g\left( t\right) $ from a given vector field $\mathbf{v}\left(
t\right) $, i.e. to solve the \emph{reconstruction equation}%
\begin{equation}
g^{\prime }=g\mathbf{v}.  \label{Poisson}
\end{equation}%
As this equation (in the right-invariant form) was originally derived by
Poisson for spatial rotations, it is referred to as (generalized) Poisson
equation. It was also reported by Darboux \cite{Darboux1887}, which is why
it is also called Poisson-Darboux equation \cite{ConduracheAASAIAA2017}.

The solution of (\ref{Poisson}) can be expressed as%
\begin{equation}
g\left( \tau \right) =g_{0}\exp \boldsymbol{\xi }\left( \tau \right) ,\ \tau
\in \left[ 0,1\right]  \label{exp1}
\end{equation}%
with $\boldsymbol{\xi }\left( \tau \right) \in \mathfrak{g}$. The so defined 
$g$ satisfies (\ref{Poisson}) if and only if $\boldsymbol{\xi }$ satisfies%
\begin{equation}
\boldsymbol{\xi }^{\prime }=\mathbf{dexp}_{-\boldsymbol{\xi }}^{-1}\mathbf{v}
\label{locRec}
\end{equation}%
where $\mathbf{dexp}_{\boldsymbol{\xi }}$ is the right-trivialized
differential of the exp map \cite{RSPA2021}. The solution of the ODE system (%
\ref{Poisson}) on $G$ is thus reduced to solving the ODE (\ref{locRec}) on $%
\mathfrak{g}$. Equation (\ref{locRec}) is referred to as \emph{local
reconstruction equation}. The solution (\ref{locRec}) was reported in the
seminal papers \cite{Hausdorff1906,Iserles1984} and in \cite{Varadarajan1984}
in terms of a series expansion of the right-trivialized differential. In the
following, the series expansion of the inverse differential is used, which
was originally reported in \cite{Hausdorff1906},%
\begin{equation}
\mathbf{dexp}_{\boldsymbol{\xi }}^{-1}=\sum_{i\geq 0}\frac{B_{i}}{i!}\mathrm{%
ad}_{\boldsymbol{\xi }}^{i}  \label{dexpInv}
\end{equation}%
where $B_{i}$ are the Bernoulli numbers. The significance of the (local)
reconstruction problem in relation to the interpolation problem lies in the
fact that the solution of (\ref{Poisson}) respectively (\ref{locRec}), given
specific boundary values, serves as an interpolant. Moreover, since this
interpolant satisfies the reconstruction equation, it guarantees consistency
with the motion $g(t)$ to be interpolated. This is addressed in the
following.

\begin{remark}
The most prominent application of the local reconstruction equations is the
family of Munthe-Kaas integration schemes on Lie groups \cite%
{MuntheKaas-BIT1998,MuntheKaas1999,IserlesMuntheKaasNrsettZanna2000}, where
the original ODE system (\ref{Poisson}) on $G$ is transcribed to the ODE (%
\ref{locRec}) on $\mathfrak{g}$, and solved numerically by Runge-Kutta
methods (but any other integration can be applied). This is motivated by the
fact that (\ref{Poisson}) must satisfy non-linear invariants whereas (\ref%
{locRec}) must satisfy linear invariants only. The latter are encoded in the
vector representation of $\mathfrak{g}$, and thus automatically preserved
during the integration step.
\end{remark}

\subsection{A $k$th-Order Approximate Solution of the Poisson Equation}

Under the standard assumption $\boldsymbol{\xi }\left( 0\right) =0$, it was
shown in \cite{ZAMM2010} that the canonical coordinates $\boldsymbol{\xi }%
\left( \tau \right) $ in (\ref{exp1}), satisfying (\ref{locRec}), admit a
series expansion of the form%
\begin{equation}
\boldsymbol{\xi }\left( \tau \right) =\sum_{k\geq 0}\tfrac{\tau ^{k}}{k!}%
\boldsymbol{\Phi }_{k}^{0}  \label{xi}
\end{equation}%
where the coefficients satisfy the recursive relation%
\begin{equation}
\boldsymbol{\Phi }_{k}^{0}=\mathbf{v}_{0}^{\left( k-1\right) }-\left(
k-1\right) !\sum_{j=1}^{k-1}\tfrac{1}{\left( j-1\right) !}\mathbf{F}_{kj}%
\boldsymbol{\Phi }_{j}^{0}  \label{Phik3}
\end{equation}%
in terms of $\mathbf{v}_{0}$ and its derivatives, with%
\begin{equation}
\mathbf{F}_{kj}=\sum_{l=1}^{k-j}\tfrac{\left( -1\right) ^{l}}{\left(
l+1\right) !}%
\hspace{-2ex}%
\sum_{\substack{ i_{1},\ldots ,i_{l}\in \mathbb{N}_{0}  \\ i_{1}+\ldots
+i_{l}+j=k}}%
\hspace{-2ex}%
\mathrm{ad}_{\mathbf{Y}_{i_{1}}}\ldots \mathrm{ad}_{\mathbf{Y}_{i_{l}}}
\label{Fk}
\end{equation}%
where $\mathbf{Y}_{i}=\frac{1}{i!}\boldsymbol{\Phi }_{i}^{0}$ and $%
\boldsymbol{\Phi }_{0}^{0}=\mathbf{0}$. The coefficients up to order 5 are 
\begin{eqnarray}
\boldsymbol{\Phi }_{1}^{0} &=&\mathbf{v}_{0},\ \ \boldsymbol{\Phi }_{2}^{0}=%
\mathbf{v}_{0}^{\prime },\ \boldsymbol{\Phi }_{3}^{0}=\mathbf{v}_{0}^{\prime
\prime }+\frac{1}{2}\left[ \mathbf{v}_{0},\mathbf{v}_{0}^{\prime }\right] ,\ 
\boldsymbol{\Phi }_{4}^{0}=\mathbf{v}_{0}^{\prime \prime \prime }+\left[ 
\mathbf{v}_{0},\mathbf{v}_{0}^{\prime \prime }\right]  \label{PhiVal} \\
\boldsymbol{\Phi }_{5}^{0} &=&\mathbf{v}_{0}^{\left( 4\right) }+\frac{3}{2}%
\left[ \mathbf{v}_{0},\mathbf{v}_{0}^{\prime \prime \prime }\right] +\left[ 
\mathbf{v}_{0}^{\prime },\mathbf{v}_{0}^{\prime \prime }\right] +\frac{1}{2}%
\left[ \mathbf{v}_{0}^{\prime },\left[ \mathbf{v}_{0}^{\prime },\mathbf{v}%
_{0}\right] \right] -\frac{1}{6}\left[ \mathbf{v}_{0},\left[ \mathbf{v}_{0},%
\mathbf{v}_{0}^{\prime \prime }\right] \right] -\frac{1}{6}\left[ \mathbf{v}%
_{0},\left[ \mathbf{v}_{0},\left[ \mathbf{v}_{0},\mathbf{v}_{0}^{\prime }%
\right] \right] \right]  \notag
\end{eqnarray}%
where $\left[ \boldsymbol{\xi },\boldsymbol{\eta }\right] $ denotes the Lie
bracket of $\boldsymbol{\xi },\boldsymbol{\eta }\in \mathfrak{g}$.
Throughout the paper, the Lie bracket is also expressed as as $\left[ 
\boldsymbol{\xi },\boldsymbol{\eta }\right] =\mathbf{ad}_{\boldsymbol{\xi }}%
\boldsymbol{\eta }$, where as $\mathbf{ad}:\mathfrak{g}\times \mathfrak{g}%
\rightarrow \mathfrak{g}$ is the adjoint action\footnote{%
Throughout the paper, the vector representation of Lie algebra elements will
be used. For an $n$-dimensional Lie algebra, $\mathbf{ad}_{\boldsymbol{\xi }%
} $ is then a $n\times n$ matrix.} on $\mathfrak{g}$. Collecting all terms
with equal power of $\tau $, the expansion (\ref{xi}) can be written as%
\begin{equation}
\boldsymbol{\xi }\left( \tau \right) =\sum_{j=1}^{k}\frac{1}{j!}\tau ^{j}%
\mathbf{v}_{0}^{\left( j-1\right) }+\sum_{r=3}^{k}\tau ^{r}\mathbf{a}_{r}
\label{expansion}
\end{equation}%
with%
\vspace{-2ex}%
\begin{equation}
\mathbf{a}_{r}:=\sum_{l=2}^{r-1}\sum_{\substack{ i_{1},\ldots ,i_{l}\in 
\mathbb{N}_{0}  \\ i_{1}+\ldots +i_{l}+l=r}}%
\hspace{-2ex}%
a_{i_{1}\ldots i_{l}}\mathrm{ad}_{\mathbf{v}_{0}^{\left( i_{1}\right)
}}\ldots \mathrm{ad}_{\mathbf{v}_{0}^{\left( i_{l-1}\right) }}\mathbf{v}%
_{0}^{\left( i_{l}\right) }.  \label{ar}
\end{equation}%
The expression (\ref{expansion}) is a $k$th-order \emph{extrapolation }of
the solution of (\ref{locRec}) in terms of $\mathbf{v}_{0}$ and its first $%
k-1$ derivatives.

The non-zero coefficients in (\ref{ar}) are explicitly, for $r=1,\ldots ,5$,%
\begin{eqnarray}
\mathbf{a}_{3} &=&\frac{1}{12}[\mathbf{v}_{0},\mathbf{v}_{0}^{\left(
1\right) }],\ \mathbf{a}_{4}=\frac{1}{24}[\mathbf{v}_{0},\mathbf{v}%
_{0}^{\left( 2\right) }]  \label{arVal} \\
\mathbf{a}_{5} &=&\frac{1}{80}[\mathbf{v}_{0},\mathbf{v}_{0}^{\left(
3\right) }]+\frac{1}{120}[\mathbf{v}_{0}^{\left( 1\right) },\mathbf{v}%
_{0}^{\left( 2\right) }]+\frac{1}{240}[\mathbf{v}_{0}^{\left( 1\right) },[%
\mathbf{v}_{0}^{\left( 1\right) },\mathbf{v}_{0}]-\frac{1}{720}[\mathbf{v}%
_{0},[\mathbf{v}_{0},\mathbf{v}_{0}^{\left( 2\right) }]-\frac{1}{720}[%
\mathbf{v}_{0},[\mathbf{v}_{0},[\mathbf{v}_{0},\mathbf{v}_{0}^{\left(
1\right) }]]].  \notag
\end{eqnarray}%
The extrapolation (\ref{expansion}) provides a basis for deriving 2-point
interpolation formulae of certain approximation orders. It also enables
deriving $k$th-order spline interpolations in a consistent systematic way,
in contrast to other approaches, e.g. \cite{KangPark1998} where general
cubic extrapolations, were assumed in the derivation.

\begin{remark}
An approximation of the solution of the right Poisson equation, $g^{\prime }=%
\mathbf{v}g$, is obtained analogously as above. The solution is expressed as 
$g\left( \tau \right) =\exp \boldsymbol{\xi }\left( \tau \right) g_{0}$, and
the right-invariant vector field $\mathbf{v}=g^{\prime }g^{-1}$ satisfies
the ODE $\boldsymbol{\xi }^{\prime }=\mathbf{dexp}_{\boldsymbol{\xi }}^{-1}%
\mathbf{v}$ on $\mathfrak{g}$. The approximation is obtained with (\ref%
{Phik3}), but with the alternating sign $\left( -1\right) ^{l}$ in (\ref{Fk}%
) omitted \cite{ZAMM2010}. In case of spatial rotations and rigid body
motions, i.e. for Lie groups $SO\left( 3\right) $ and $SE\left( 3\right) $,
the right-invariant vector field is the spatial representation of angular
velocity and twist, respectively.
\end{remark}

\section{2-Point Interpolations with Zero Initial Value $\protect\xi \left(
0\right) =0$\label{sec2PointZero}}

\begin{definition}
A $k$th-order 2-point boundary value interpolation (in canonical
coordinates), between $g_{0}$ and $g_{1}$ on $G$ is a map%
\begin{equation}
\bar{g}\left( \tau \right) =\bar{g}_{0}\exp \boldsymbol{\xi }^{\left[ k%
\right] }\left( \tau \right) ,\ \tau \in \left[ 0,1\right]  \label{2PointDef}
\end{equation}%
satisfying boundary values $\bar{g}\left( 0\right) =g_{0}$, $\bar{g}\left(
1\right) =g_{1}$, and $\bar{\mathbf{v}}^{\left( s\right) }\left( 0\right) =%
\mathbf{v}_{0}^{\left( s\right) },s=0,\ldots ,n_{0}$ and $\bar{\mathbf{v}}%
^{\left( s\right) }\left( 1\right) =\mathbf{v}_{1}^{\left( s\right)
},s=0,\ldots ,n_{1}$, with $n_{0}+n_{1}=k-1$, where $\bar{\mathbf{v}}:=\bar{g%
}^{-1}\bar{g}^{\prime }\in \mathfrak{g}$ is the left-trivialized derivative
of $\bar{g}$, and $\boldsymbol{\xi }^{\left[ k\right] }\left( \tau \right) $
is a polynomial of degree $k$ in $\tau $.
\end{definition}

It is assumed in the following that $\boldsymbol{\xi }^{\left[ k\right]
}\left( 0\right) =0$, which implies that $\bar{g}_{0}=g\left( 0\right) $.
This is the standard assumption found in the literature. It will be relaxed
in section 6 
to allow for $\boldsymbol{\xi }^{\left[ k\right] }\left( 0\right) \neq 0$. A 
$k$th-order $C^{k-1}$ continuos initial value 2-point interpolation between $%
g_{0}$ and $g_{1}$ is derived in the following. The point of departure is
the closed form approximate solution (\ref{expansion}) of the reconstruction
equation (\ref{exp1}).

\subsection{Initial Value 2-Point Interpolation}

It is assumed that the initial values $\mathbf{v}_{0}^{\left( s\right) }=%
\mathbf{v}^{\left( s\right) }\left( 0\right) ,s=0,\ldots ,k-1$ are given,
along with the boundary values $g_{0}$ and $g_{1}$. At the terminal end, the
canonical coordinates are $\bar{\boldsymbol{\xi }}:=\log \left(
g_{0}^{-1}g_{1}\right) $. Evaluating the $k$th-order truncation of the
expansion (\ref{expansion}) at $\tau =1$, and solving this for the highest
derivative of $\mathbf{v}_{0}$ yields%
\begin{equation}
k!\mathbf{v}_{0}^{\left( k-1\right) }=\bar{\boldsymbol{\xi }}%
-\sum_{j=1}^{k-1}\frac{1}{j!}\mathbf{v}_{0}^{\left( j-1\right)
}+\sum_{r=3}^{k}\mathbf{a}_{r}.
\end{equation}%
This is substituted back into (\ref{expansion}) giving rise to the $k$\emph{%
th-order initial value interpolation formula} for the canonical coordinates
in (\ref{2PointDef}) 
\begin{equation}
\boldsymbol{\xi }^{\left[ k\right] }\left( \tau \right) =\;\tau ^{k}\bar{%
\boldsymbol{\xi }}+\sum_{j=1}^{k-1}\frac{1}{j!}(\tau ^{j}-\tau ^{k})\mathbf{v%
}_{0}^{\left( j-1\right) }+\sum_{r=3}^{k-1}(\tau ^{r}-\tau ^{k})\mathbf{a}%
_{r}  \label{2PointInt}
\end{equation}%
with $\mathbf{a}_{r}$ in (\ref{ar}). In particular, the \emph{3rd-, 4th-,
and 5th-order initial value interpolation} formulae are%
\begin{eqnarray}
\boldsymbol{\xi }^{\left[ 3\right] }\left( \tau \right) &=&\tau ^{3}\bar{%
\boldsymbol{\xi }}+\left( \tau -\tau ^{3}\right) \mathbf{v}_{0}+\frac{1}{2}%
\left( \tau ^{2}-\tau ^{3}\right) \mathbf{v}_{0}^{\prime }  \label{3PP} \\
\boldsymbol{\xi }^{\left[ 4\right] }\left( \tau \right) &=&\tau ^{4}\bar{%
\boldsymbol{\xi }}+\left( \tau -\tau ^{4}\right) \mathbf{v}_{0}+\frac{1}{2}%
\left( \tau ^{2}-\tau ^{4}\right) \mathbf{v}_{0}^{\prime }+\frac{1}{6}\left(
\tau ^{3}-\tau ^{4}\right) (\mathbf{v}_{0}^{\prime \prime }+\frac{1}{2}\left[
\mathbf{v}_{0},\mathbf{v}_{0}^{\prime }\right] )  \label{4PP} \\
\boldsymbol{\xi }^{\left[ 5\right] }\left( \tau \right) &=&\tau ^{5}\bar{%
\boldsymbol{\xi }}+\left( \tau -\tau ^{5}\right) \mathbf{v}_{0}+\frac{1}{2}%
\left( \tau ^{2}-\tau ^{5}\right) \mathbf{v}_{0}^{\prime }+\frac{1}{6}\left(
\tau ^{3}-\tau ^{5}\right) (\mathbf{v}_{0}^{\prime \prime }+\frac{1}{2}\left[
\mathbf{v}_{0},\mathbf{v}_{0}^{\prime }\right] )+\frac{1}{12}\left( \tau
^{4}-\tau ^{5}\right) (\mathbf{v}_{0}^{\prime \prime \prime }+\left[ \mathbf{%
v}_{0},\mathbf{v}_{0}^{\prime \prime }\right] ).
\end{eqnarray}%
The boundary conditions $\boldsymbol{\xi }^{\left[ k\right] }\left( 0\right)
=0,\boldsymbol{\xi }^{\left[ k\right] }\left( 1\right) =\bar{\boldsymbol{\xi 
}}$, and the initial condition on $\bar{\mathbf{v}}_{0}=\left. \bar{g}^{-1}%
\bar{g}^{\prime }\right\vert _{\tau =0}$ and on its $k-1$ derivatives are
satisfied by construction.

\begin{remark}
It was shown in \cite{ParkRavani1997} that, on $SO(3)$, that (\ref{3PP}) is
the minimum energy curve if $\mathbf{v}_{0}$ and $\mathbf{v}_{0}^{\prime }$
are zero. If the initial velocity and its derivatives are zero, all $k$%
th-order interpolations describe the same path (a 1-parameter subgroup of $G$%
) but the trajectories differ ($k$th-order polynomials in $\tau $).
\end{remark}

\begin{lemma}
\label{lemError}The $k$th-order initial value 2-point interpolation formula (%
\ref{2PointInt}) has local error%
\begin{equation}
\tau ^{k}%
\Big%
(\bar{\boldsymbol{\xi }}-\sum_{j=1}^{k}\frac{1}{j!}\mathbf{v}_{0}^{\left(
j-1\right) }-\sum_{r=3}^{k}\mathbf{a}_{r}%
\Big%
)+O(\tau ^{k+1}).  \label{LocError}
\end{equation}
\end{lemma}

\begin{proof}
%
Let $g\left( \tau \right) =g_{0}\exp \boldsymbol{\xi }\left( \tau \right) $
be the curve in $G$ to be interpolated, with $\boldsymbol{\xi }\left(
0\right) =0$, and $g^{\left[ 3\right] }\left( \tau \right) =\bar{g}_{0}\exp 
\boldsymbol{\xi }^{\left[ k\right] }\left( \tau \right) $ be the
interpolant, with $\bar{g}_{0}=g_{0}$. The local error is introduced as $%
\log (g^{-1}g^{\left[ 3\right] })$. The expansion (\ref{expansion}) of $%
g\left( \tau \right) $ with $k\rightarrow \infty $ and the $k$th-order
interpolation (\ref{2PointInt}), the local error becomes%
\begin{eqnarray*}
\log (g^{-1}g^{\left[ 3\right] }) &=&\boldsymbol{\xi }^{\left[ k\right]
}\left( \tau \right) -\boldsymbol{\xi }\left( \tau \right) +\frac{1}{2}[%
\boldsymbol{\xi }^{\left[ k\right] }\left( \tau \right) ,\boldsymbol{\xi }%
\left( \tau \right) ]+\ldots \\
&=&\tau ^{k}%
\Big%
(\bar{\boldsymbol{\xi }}-\sum_{j=1}^{k}\frac{1}{j!}\mathbf{v}_{0}^{\left(
j-1\right) }-\sum_{r=3}^{k}\mathbf{a}_{r}%
\Big%
)-\sum_{j>k}\frac{1}{j!}\tau ^{j}\mathbf{v}_{0}^{\left( j-1\right)
}-\sum_{r>k}\tau ^{r}\mathbf{a}_{r} \\
&&+\frac{1}{2}\tau ^{k}%
\Big%
[\bar{\boldsymbol{\xi }}-\sum_{j=1}^{k-1}\frac{1}{j!}\mathbf{v}_{0}^{\left(
j-1\right) }-\sum_{r=3}^{k-1}\mathbf{a}_{r},\boldsymbol{\xi }\left( \tau
\right) 
\Big%
]+\frac{1}{2}%
\Big%
[\sum_{j=1}^{k-1}\frac{1}{j!}\tau ^{j}\mathbf{v}_{0}^{\left( j-1\right)
}+\sum_{r=3}^{k-1}\tau ^{r}\mathbf{a}_{r},\sum_{j\geq k}\frac{1}{j!}\tau ^{j}%
\mathbf{v}_{0}^{\left( j-1\right) }+\sum_{r\geq k}\tau ^{r}\mathbf{a}_{r}%
\Big%
] \\
&&+\frac{1}{2}%
\Big%
[\sum_{j=1}^{k-1}\frac{1}{j!}\tau ^{j}\mathbf{v}_{0}^{\left( j-1\right)
}+\sum_{r=3}^{k-1}\tau ^{r}\mathbf{a}_{r},\sum_{j=1}^{k-1}\frac{1}{j!}\tau
^{j}\mathbf{v}_{0}^{\left( j-1\right) }+\sum_{r=3}^{k-1}\tau ^{r}\mathbf{a}%
_{r}%
\Big%
] \\
&=&\tau ^{k}%
\Big%
(\bar{\boldsymbol{\xi }}-\sum_{j=1}^{k}\frac{1}{j!}\mathbf{v}_{0}^{\left(
j-1\right) }-\sum_{r=3}^{k}\mathbf{a}_{r}%
\Big%
)+O(\tau ^{k+1}).
\end{eqnarray*}%
The Lie bracket term in the third line vanishes, and the remaining terms are
of order higher than $k$ in $\tau $ leading to the local error (\ref%
{LocError}).%
%
\end{proof}

\begin{lemma}
\label{lemExactRec}The $k$th-order initial value 2-point interpolation
exactly reconstructs a $k$th-order curve defined on a 1-parameter subgroup
of $G$.
\end{lemma}

\begin{proof}
%
A $k$th-order 1-parameter subgroup is defined by $g\left( \tau \right)
=g_{0}\exp (p\left( \tau \right) \boldsymbol{\xi })$, with constant $%
\boldsymbol{\xi }\in \mathfrak{g}$ and a polynomial $p\left( \tau \right)
=c_{0}+c_{1}\tau +\ldots +c_{k}\tau ^{k}$ of degree $k$. The velocity and
derivatives are thus $\mathbf{v}_{0}^{\left( s\right) }=p^{\left( s+1\right)
}\boldsymbol{\xi }$. The polynomial satisfies $p\left( 0\right) =0$ and $%
p\left( 1\right) =1$, thus $c_{0}=0$ and $c_{1}+\ldots +c_{k}=1$. The
derivatives of $p$ at $\tau =0$ are $p^{\left( s\right) }\left( 0\right)
=s!c_{s}$. The 2-point interpolation is then 
\begin{equation*}
\boldsymbol{\xi }^{\left[ k\right] }\left( \tau \right) =\tau ^{k}\bar{%
\boldsymbol{\xi }}+\sum_{j=1}^{k-1}\frac{1}{j!}(\tau ^{j}-\tau ^{k})\mathbf{v%
}_{0}^{\left( j-1\right) }+\sum_{r=3}^{k-1}(\tau ^{r}-\tau ^{k})\mathbf{a}%
_{r}=\tau ^{k}\bar{\boldsymbol{\xi }}+\sum_{j=1}^{k-1}c_{j}(\tau ^{j}-\tau
^{k})\boldsymbol{\xi }=\left( c_{1}\tau +c_{2}\tau ^{2}+\ldots +c_{k}\tau
^{k}\right) \boldsymbol{\xi }=p\left( \tau \right) \boldsymbol{\xi }
\end{equation*}%
hence exactly reproduce the $k$th-order curve.
\end{proof}

%
The following is a direct consequence of the De Casteljau algorithm 9 in
appendix B for constructing B\'{e}zier curves on Lie groups. The proof
follows immediately from the fact the log and exp map cancel (within their
uniqueness domains), and that the Lie group algorithm does reduces to the
classical De Casteljau algorithm on a vector space, which can exactly
reconstruct a $k$th-order curve with correctly chosen control points.

\begin{proposition}
A B\'{e}zier curve of order $k$ constructed with the De Casteljau algorithm
9 (appendix B) on $G$ exactly reconstructs a $k$th-order curve defined on a
1-parameter subgroup of $G$, if the control points (in terms of $\tau $) are
the Bernstein basis.
\end{proposition}

%

\subsection{Boundary Value 2-Point Interpolation}

Now it is assumed that, in addition to some initial values, the terminal
value $\mathbf{v}_{1}$ is given. The general case, when derivatives of $%
\mathbf{v}_{1}$ are prescribed, is not addressed here.

A 2-point interpolation is derived from the initial value interpolation (\ref%
{2PointInt}). To this end, the derivative of the $k$th-order interpolation (%
\ref{2PointInt}) is evaluated at $\tau =1$, which yields, after some
manipulation,%
\begin{align}
\frac{\mathrm{d}^{s}}{\mathrm{d}\tau ^{s}}\boldsymbol{\xi }^{\left[ k\right]
}\left( 1\right) =\;& \frac{k!}{\left( k-s\right) !}\bar{\boldsymbol{\xi }}-%
\frac{s}{\left( k-s\right) !}\mathbf{v}_{0}^{\left( k-2\right)
}+\sum_{j=s}^{k-2}\frac{1}{\left( j-s\right) !}\mathbf{v}_{0}^{\left(
j-1\right) }-\sum_{j=1}^{k-2}\frac{k!}{j!\left( k-s\right) !}\mathbf{v}%
_{0}^{\left( j-1\right) }  \notag \\
& +\sum_{r=\max (3,s)}^{k}\frac{r!}{\left( r-s\right) !}\mathbf{a}%
_{r}-\sum_{r=3}^{k}\frac{k!}{\left( k-s\right) !}\mathbf{a}_{r}  \label{dxi2}
\end{align}%
for $1\leq s\leq k-1$. The derivative of $\boldsymbol{\xi }\left( \tau
\right) $ is also obtained exactly from (\ref{locRec}) as 
\begin{equation}
\frac{\mathrm{d}^{s}}{\mathrm{d}\tau ^{s}}\boldsymbol{\xi }\left( \tau
\right) =\boldsymbol{\Phi }_{s}(\boldsymbol{\xi },\mathbf{v},\mathbf{v}%
^{\prime },\ldots ,\mathbf{v}^{\left( s-1\right) })  \label{dxi1}
\end{equation}%
where%
\begin{equation}
\boldsymbol{\Phi }_{k}(\boldsymbol{\xi }\left( \tau \right) ,\mathbf{v}%
\left( \tau \right) ,\mathbf{v}^{\prime }\left( \tau \right) ,\ldots ,%
\mathbf{v}^{\left( s-1\right) }\left( \tau \right) ):=\sum_{i=1}^{k}%
\boldsymbol{\Phi }_{ki}(\boldsymbol{\xi }\left( \tau \right) )\mathbf{v}%
^{\left( i-1\right) }\left( \tau \right) ,k\geq 1  \label{Phik}
\end{equation}%
\begin{equation}
\mathrm{with\ \ \ \ }\boldsymbol{\Phi }_{ki}(\boldsymbol{\xi }\left( \tau
\right) ):=\binom{k-1}{i-1}\frac{\mathrm{d}^{k-i}}{\mathrm{d}\tau ^{k-i}}%
\mathbf{dexp}_{-\boldsymbol{\xi }\left( \tau \right) }^{-1}.  \label{Phik2}
\end{equation}%
Evaluated at $\tau =1$ yields $\frac{\mathrm{d}^{s}}{\mathrm{d}\tau ^{s}}%
\boldsymbol{\xi }\left( 1\right) =\boldsymbol{\Phi }_{s}(\bar{\boldsymbol{%
\xi }},\mathbf{v}_{1},\mathbf{v}_{1}^{\prime },\ldots ,\mathbf{v}%
_{1}^{\left( s-1\right) })$. The latter depends on $\mathbf{v}_{1}$ and its
derivatives, while (\ref{dxi2}) depends on $\mathbf{v}_{0}$ and its
derivatives. Equating the two at $\tau =1$, and solving (\ref{dxi2}) for the
highest derivative $\mathbf{v}_{0}^{\left( k-2\right) }$ yields%
\begin{equation}
\mathbf{v}_{0}^{\left( k-2\right) }=\frac{k!}{s}\bar{\boldsymbol{\xi }}-%
\frac{\left( k-s\right) !}{s}\boldsymbol{\Phi }_{s}(\bar{\boldsymbol{\xi }},%
\mathbf{v}_{1},\mathbf{v}_{1}^{\prime },\ldots ,\mathbf{v}_{1}^{\left(
s-1\right) })+\sum_{j=s}^{k-2}\frac{\left( k-s\right) !}{\left( j-s\right) !s%
}\mathbf{v}_{0}^{\left( j-1\right) }-\sum_{j=1}^{k-2}\frac{k!}{j!s}\mathbf{v}%
_{0}^{\left( j-1\right) }+\sum_{r=\max (3,s)}^{k}\frac{r!\left( k-s\right) !%
}{\left( r-s\right) !s}\mathbf{a}_{r}-\sum_{r=3}^{k}\frac{k!}{s}\mathbf{a}%
_{r}.  \label{vk}
\end{equation}%
Thus $\mathbf{v}_{0}^{\left( k-2\right) }$ can be expressed in terms of $%
\mathbf{v}_{0},\ldots ,\mathbf{v}_{0}^{\left( k-3\right) },\mathbf{v}%
_{1},\ldots ,\mathbf{v}_{1}^{\left( s-1\right) }$. In the following, only
the case $s=1$ is considered, in which the highest derivative $\mathbf{v}%
_{0}^{\left( k-2\right) }$ is replaced by the terminal value $\mathbf{v}_{1}$%
. This seems to be most relevant for applications.

For the 3rd-order interpolation, i.e. $k=3,s=1$, relation (\ref{vk}) yields%
\begin{equation}
\mathbf{v}_{0}^{\prime }=6\bar{\boldsymbol{\xi }}-2\mathrm{dexp}_{-\bar{%
\boldsymbol{\xi }}}^{-1}\mathbf{v}_{1}-4\mathbf{v}_{0}.
\end{equation}%
Inserting this result back into (\ref{3PP}) yields the \textbf{3rd-order
boundary value interpolation}\emph{\ }%
\begin{equation}
\boldsymbol{\xi }^{\left[ 3\right] }\left( \tau \right) =\left( 3\tau
^{2}-2\tau ^{3}\right) \bar{\boldsymbol{\xi }}+\tau \left( \tau -1\right)
^{2}\mathbf{v}_{0}+\left( \tau ^{3}-\tau ^{2}\right) 
\color[rgb]{0.7,0,0}%
\mathrm{dexp}_{-\bar{\boldsymbol{\xi }}}^{-1}\mathbf{v}_{1}%
\color{black}
\label{3PPB}
\end{equation}%
with prescribed initial and terminal velocity. Relation (\ref{vk}) evaluated
for the 4th-order interpolation ($k=4,s=1$) yields%
\begin{equation}
\mathbf{v}_{0}^{\prime }=24\bar{\boldsymbol{\xi }}-6\mathrm{dexp}_{-\bar{%
\boldsymbol{\xi }}}^{-1}\mathbf{v}_{1}-18\mathbf{v}_{0}-6\mathbf{v}%
_{0}^{\prime }-\frac{1}{2}[\mathbf{v}_{0},\mathbf{v}_{0}^{\prime }]
\end{equation}%
and, along with (\ref{4PP}), leads to the the \textbf{4th-order boundary
value interpolation}%
\begin{equation}
\boldsymbol{\xi }^{\left[ 4\right] }\left( \tau \right) =\left( 4\tau
^{3}-3\tau ^{4}\right) \bar{\boldsymbol{\xi }}+\tau \left( \tau -\tau
^{2}\right) \left( 1+2\tau \right) \mathbf{v}_{0}+\left( \tau ^{4}-\tau
^{3}\right) 
\color[rgb]{0.7,0,0}%
\mathrm{dexp}_{-\bar{\boldsymbol{\xi }}}^{-1}%
\color{black}%
\mathbf{v}_{1}+\frac{1}{2}\tau ^{2}\left( 1-\tau \right) ^{2}\mathbf{v}%
_{0}^{\prime }  \label{4PPB}
\end{equation}%
with prescribed initial velocity and derivative, and terminal velocity. In
the same way the \textbf{5th-order boundary value interpolation} ($k=5,s=1$%
), with additionally prescribed initial value $\mathbf{v}_{0}^{\prime \prime
}$, is obtained as%
\begin{eqnarray}
\boldsymbol{\xi }^{\left[ 5\right] }\left( \tau \right)  &=&\left( 5\tau
^{4}-4\tau ^{5}\right) \bar{\boldsymbol{\xi }}+\left( \tau -4\tau ^{4}+3\tau
^{5}\right) \mathbf{v}_{0}+\frac{1}{2}\left( \tau ^{2}-3\tau ^{4}+2\tau
^{5}\right) \mathbf{v}_{0}^{\prime } \\
&&+\frac{1}{6}\left( \tau ^{3}-2\tau ^{4}+\tau ^{5}\right) (\mathbf{v}%
_{0}^{\prime \prime }+\frac{1}{2}[\mathbf{v}_{0},\mathbf{v}_{0}^{\prime
}])+\left( \tau ^{5}-\tau ^{4}\right) 
\color[rgb]{0.7,0,0}%
\mathrm{dexp}_{-\bar{\boldsymbol{\xi }}}^{-1}%
\color{black}%
\mathbf{v}_{1}  \notag
\end{eqnarray}

\begin{remark}
It was shown in \cite{ZefranKumarCroke-TRO1998} that the 3rd-order
approximation (\ref{3PPB}) is exactly the minimum acceleration curve when
interpolating rigid body motions on the Lie group $SE(3)$. This was also
observed in \cite{ParkRavani1997} for the special case of zero initial and
terminal velocities, $\mathbf{v}_{0}=\mathbf{v}_{1}=\mathbf{0}$. In the
latter case, the minium acceleration curve (\ref{3PPB}) reduces to the
1-parameter motion $\boldsymbol{\xi }^{\left[ 3\right] }\left( \tau \right)
=\left( 3\tau ^{2}-2\tau ^{3}\right) \bar{\boldsymbol{\xi }}$, which is a
geodesic w.r.t. the left-invariant metric on $SE(3)$.
\end{remark}

\begin{example}
\label{exaBeam2Point}The generation of curves in the Lie group $%
SE(3)=SO(3)\ltimes \mathbb{R}^{3}$ of isometric and orientation preserving
transformations in Euclidean 3-space is of particular practical relevance.
Then, $\boldsymbol{\xi }\in {\mathbb{R}}^{6}\cong se\left( 3\right) $ is the
screw coordinate vector that generates a frame transformation $g\in SE\left(
3\right) $ via the exp map (\ref{expSE3}). Besides trajectory planning of
UAVs \cite{Lovegrove2013,Dhullipalla2019}, the shape reconstruction of
flexible rods, modeled as Cosserat continua, is an application of particular
importance for soft robots \cite%
{BlackTillRucker2018,Hussain2021,BriotBoyer2022,JMR2024}. The rod's cross
section performs spatial motions parameterized by the normalized arc length $%
\tau \in \left[ 0,1\right] $. The spatial pose of the rod's cross section is
represented by a rigid frame transformation $h\left( \tau \right) \in SE(3)$%
. A Cosserat rod is thus a one-dimensional continuum, represented as a curve
in $SE(3)$. The deformation field in body-fixed representation is then
defined by the Poisson equation (\ref{Poisson}) as $\mathbf{v}%
=g^{-1}g^{\prime }$ \cite{BriotBoyer2022}. As an example, the following
problem is addressed: Given the pose $g_{0}=g\left( 0\right) $ and $%
g_{1}=g\left( 1\right) $ at the start and terminal end, and the
deformation/strain $\mathbf{v}_{0}=\mathbf{v}\left( 0\right) $ and $\mathbf{v%
}_{1}=\mathbf{v}\left( 1\right) $ at the both ends, compute the shape of the
rod in a static equilibrium, i.e. the curve $g\left( t\right) $ in $SE(3)$
for which the elastic potential is stationary. This is a typical problem for
dual arm robotic manipulation of deformable slender objects \cite%
{AlmaghoutCherubiniKlimchik2024,AlmaghoutKlimchik2024}, and for continuum
robots \cite{Wang2018,XiaoJMR2023}.\newline
The static equilibrium is governed by the Euler-Poincar\'{e} equation \cite%
{BriotBoyer2022}%
\begin{equation}
\boldsymbol{\Lambda }^{\prime }-%
%
\mathbf{ad}_{\mathbf{v}}^{T}%
%
\boldsymbol{\Lambda }=\mathbf{W}  \label{Strain}
\end{equation}%
where $\boldsymbol{\Lambda }\left( \tau \right) \in se^{\ast }\left(
3\right) $ is the stress, and $\mathbf{W}\left( \tau \right) \in se^{\ast
}\left( 3\right) $ represents distributed loads along the beam. Assuming
linear elastic constitutive relations, the stress is related to the strain
as $\boldsymbol{\Lambda }=\mathbf{K}\left( \mathbf{v}-\mathbf{v}^{0}\right) $
with the stiffness matrix $\mathbf{K}$. Therein, $\mathbf{v}^{0}$ is in
correspondence to the geometry of the undeformed beam. For an initially
straight beam, where the $x$-axis of the cross sectional frame is aligned
with the beam center line, it is $\mathbf{v}^{0}=\left[ 0,0,0,1,0,0\right]
^{T}$. The cross section frame is located at the center of the cross section
and aligned with the principal axes. Then the stiffness matrix is $\mathbf{K}%
\left( \tau \right) =\mathrm{diag}\left(
GJ_{x},EI_{yy},EI_{zz},EA,GA,GA\right) $, with Young's modulus $E$, shear
stiffness $G$, the cross section area $A$, and the second area moments $%
I_{xx},I_{zz}$, and polar moment $J_{x}$. Inserting the constitutive
relation into the Euler-Poincar\'{e} equation (\ref{Strain}) yields the ODE
on $se\left( 3\right) $ for the deformation field $\mathbf{v}$ 
\begin{equation}
\mathbf{v}^{\prime }=\mathbf{v}_{0}^{\prime }+\mathbf{K}^{-1}(%
%
\mathbf{ad}_{\mathbf{v}}^{T}%
%
\mathbf{K}-\mathbf{K}^{\prime })\left( \mathbf{v}-\mathbf{v}^{0}\right) +%
\mathbf{K}^{-1}\mathbf{W}.  \label{ChiBar}
\end{equation}%
This ODE can be solved with given initial/terminal value of $\mathbf{v}$
along with the kinematic equation (\ref{Poisson}). 
%
The boundary deformation is obtained from (\ref{ChiBar}) as $\mathbf{v}_{0}=%
\mathbf{v}^{0}-\mathbf{K}^{-1}\left( 0\right) \bar{\mathbf{W}}\left(
0\right) $ and $\mathbf{v}\left( 1\right) =\mathbf{v}^{0}-\mathbf{K}%
^{-1}\left( 1\right) \bar{\mathbf{W}}\left( 1\right) $, where $\bar{\mathbf{W%
}}\left( 0\right) $ and $\bar{\mathbf{W}}\left( 1\right) $ are concentrated
loads at the start and terminal end, repsectiely. 
%
Notice that the above equations are normalized.\newline
The static equation (\ref{Strain}), respectively (\ref{ChiBar}), is only
used in this example to provide a reference solutions $\boldsymbol{\xi }%
\left( \tau \right) ,\mathbf{v}\left( \tau \right) $. These solutions are
compared with the 3rd- and 4th-order interpolations obtained with $\mathbf{v}%
_{0}$ and $\mathbf{v}_{1}$. The exponential map on $SE(3)$ and its
right-trivialized derivative $\mathrm{dexp}$ in closed form, and the
definition of Lie bracket on $se\left( 3\right) $ are summarized in appendix
A.\newline
For computations, a clamped slender rubber beam of 100\thinspace mm length,
a constant $8\times 8$\thinspace mm$^{2}$ square cross section, and material
parameters $E=10\,$MPa, $G=0.3\,$MPa is considered. Distributed loads $%
\mathbf{W}$ are neglected, thus only the wrench $\bar{\mathbf{W}}\left(
0\right) $ at the start, and $\bar{\mathbf{W}}\left( 1\right) $ at the
terminal end due to the geometric fixation are present. 
%
The ODE (\ref{ChiBar}) and (\ref{Poisson}) are solved with the Munthe-Kaas
method using a RK4 integration scheme \cite%
{MuntheKaas-BIT1998,HairerLubichWanner2006} to obtain a reference solutions $%
\boldsymbol{\xi }\left( \tau \right) $ and $\mathbf{v}\left( \tau \right) $.
The terminal wrench at $\tau =0$ is prescribed as $\bar{\mathbf{W}}\left(
0\right) =\left[ 0,-0.0293,-0.1277,0.0977,0.0665,-0.1950\right] $. From the
latter, the deformation at the start end is obtained as $\mathbf{v}_{0}=%
\left[ 0,0.85797,3.71876,0.99999,-0.00035,0.00102\right] $. With $%
\boldsymbol{\xi }\left( 0\right) =\mathbf{0}$ and $\mathbf{v}\left( 0\right)
=\mathbf{v}_{0}$, the ODE system is solved as initial value problem, which
yields the terminal pose of the cross section, relative to its initial pose
(Fig. \ref{fig2Point_Beam2_4th}), 
\begin{equation*}
g_{1}=%
\begin{bmatrix}
\mathbf{R}_{1} & \mathbf{r}_{1} \\ 
\mathbf{0} & 1%
\end{bmatrix}%
\in SE(3),\ \mathbf{R}_{1}=%
\begin{bmatrix}
-0.51519 & -0.42868 & -0.74217 \\ 
-0.69850 & -0.29163 & 0.65339 \\ 
-0.49653 & 0.85509 & -0.14923%
\end{bmatrix}%
\in SO(3),\ \mathbf{r}_{1}=%
\begin{bmatrix}
-0.14518 \\ 
0.29880 \\ 
-0.42338%
\end{bmatrix}%
\in {\mathbb{R}}^{3}
\end{equation*}%
and $\mathbf{v}_{1}=\left[ 2.92969,-0.50391,1,0.00119,0\right] $ and $%
\mathbf{v}_{0}^{\prime }=\left[ 0,5.71458,1.94894,-0.00006,0.00189,-0.00044%
\right] $.%
%
\newline
The 3rd- and 4th-order approximation are computed with the 2-point
interpolations (\ref{3PPB}) and (\ref{4PPB}), respectively. The deviation
from the reference solution is computed as $\boldsymbol{\varepsilon }^{\left[
k\right] }\left( \tau \right) =\log \left( \exp \left( -\boldsymbol{\xi }%
\left( \tau \right) \right) \exp (\boldsymbol{\xi }^{\left[ k\right] }\left(
\tau \right) )\right) ,k=3,4$, where $\boldsymbol{\varepsilon }^{\left[ k%
\right] }=\left[ \mathbf{x}^{\left[ k\right] },\mathbf{y}^{\left[ k\right] }%
\right] \in \mathbb{R}^{6}$ is the screw coordinate vector. As there is no
bi-invariant metric on $SE(3)$, the deviation is quantified separately for
the $SO(3)$ and $\mathbb{R}^{3}$ part, as $\varepsilon _{\mathrm{r}}^{\left[
k\right] }\left( \tau \right) =\left\Vert \mathbf{x}^{\left[ k\right]
}\left( \tau \right) \right\Vert $ and $\varepsilon _{\mathrm{p}}^{\left[ k%
\right] }\left( \tau \right) =\left\Vert \mathbf{y}^{\left[ k\right] }\left(
\tau \right) \right\Vert $, respectively. Fig. \ref{fig2Point_Beam2_error}a)
shows the results obtained with (\ref{3PPB}) and (\ref{4PPB}). For
completeness, the actual shape is shown in Fig. \ref{fig2Point_Beam2_4th}. 
\begin{figure}[b]
\begin{center}
\includegraphics[draft=false,height=4.cm]{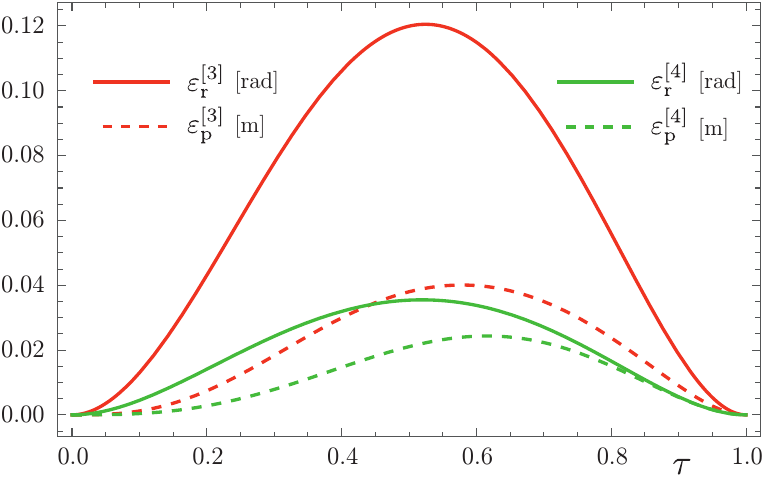} 
\hfil
\includegraphics[draft=false,height=4.cm]{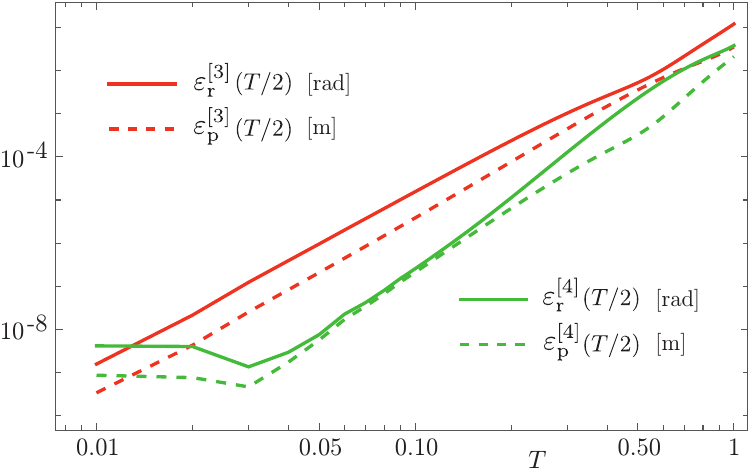}
\end{center}
\caption{a) Error when interpolating a curve in $SE(3)$, representing a
flexible rod, with given initial and terminal strain, using the 3rd- and
4th-order interpolation. b) Interpolation error at the midle of the flexible
rod as function of segment length $T$.\protect\vspace{-3ex}}
\label{fig2Point_Beam2_error}
\end{figure}
\newline
Finally, the length of the segment over which the shape is interpolated is
varied to show the order of interpolation error. To this end, the terminal
point of the interpolation is located at $0<T\leq 1$, i.e. the segment
length is $T$, and the terminal values are computed as $\bar{\boldsymbol{\xi 
}}=\boldsymbol{\xi }\left( T\right) $ and $\mathbf{v}_{1}=\mathbf{v}\left(
T\right) ,\mathbf{v}_{1}^{\prime }=\mathbf{v}^{\prime }\left( T\right) $.
The error measures $\varepsilon _{\mathrm{r}}$ and $\varepsilon _{\mathrm{p}%
} $ are evaluated at the center of the segment. Fig. \ref%
{fig2Point_Beam2_error}b) shows values of $\varepsilon _{\mathrm{r}}\left(
T/2\right) $ and $\varepsilon _{\mathrm{p}}\left( T/2\right) $ for $%
T=0.01,\ldots ,1$. The error decreases with order 3 and 4, respectively,
with the segment length in accordance with Lemma \ref{lemError}. 
\begin{figure}[b]
\begin{center}
a)\includegraphics[draft=false,height=5.cm]{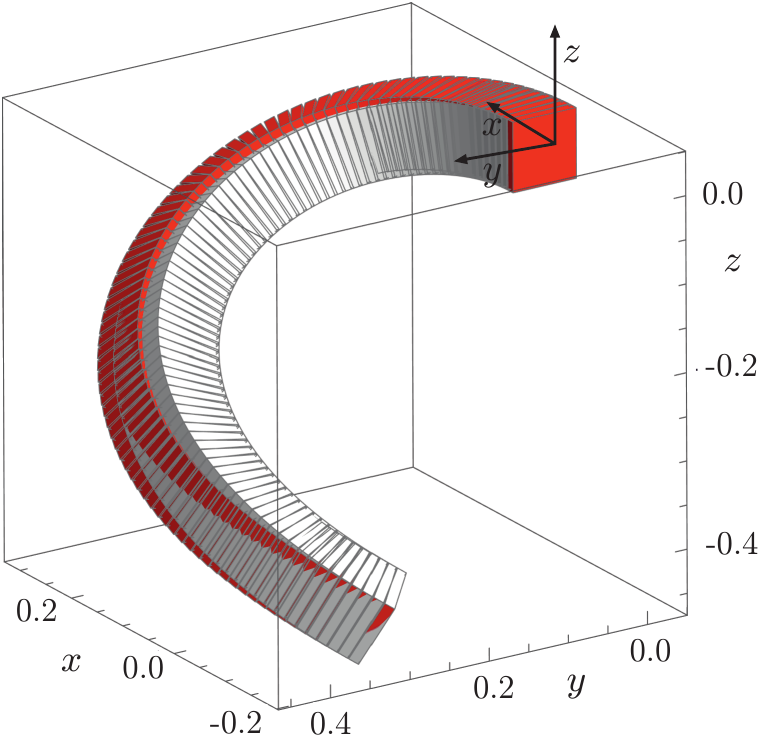} 
\hfil
b)\includegraphics[draft=false,height=5.cm]{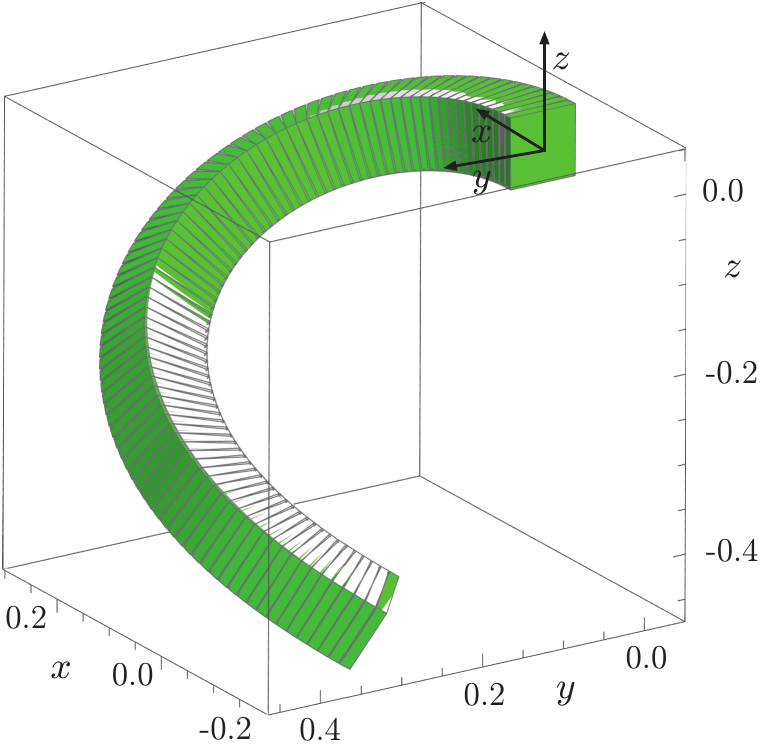}
\end{center}
\caption{Shape of the flexible rod as result of a) the 3rd- and b) the
4th-order interpolation. The exact numerical solution of the static balance
equation (\protect\ref{ChiBar}) is shown in gray color as a reference
(partially covered by the approximate solutions).}
\label{fig2Point_Beam2_4th}
\end{figure}
%
The result when interpolation the displacement field of a rod with
rectangular cross section is shown in Fig. \ref{fig2Point_Beam3}. The error
is similar to that of the rod with square cross section in Fig. \ref%
{fig2Point_Beam2_error}. 
\begin{figure}[b]
\begin{center}
a)\includegraphics[draft=false,height=4.cm]{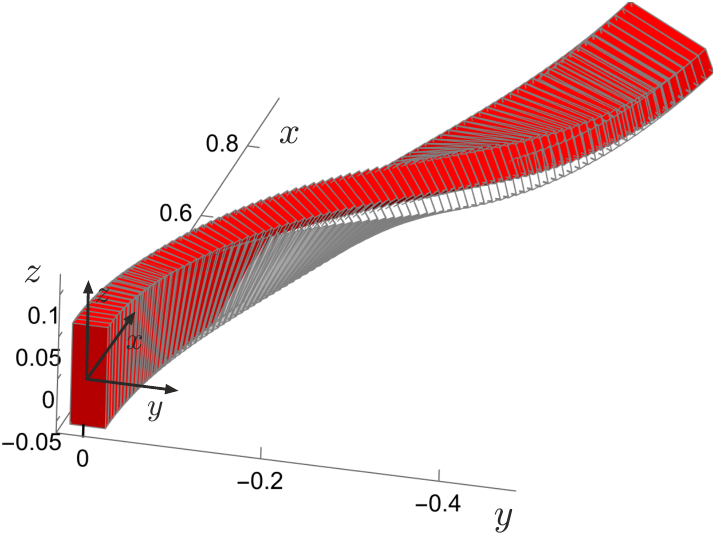} %
\hfil b)%
\includegraphics[draft=false,height=4.cm]{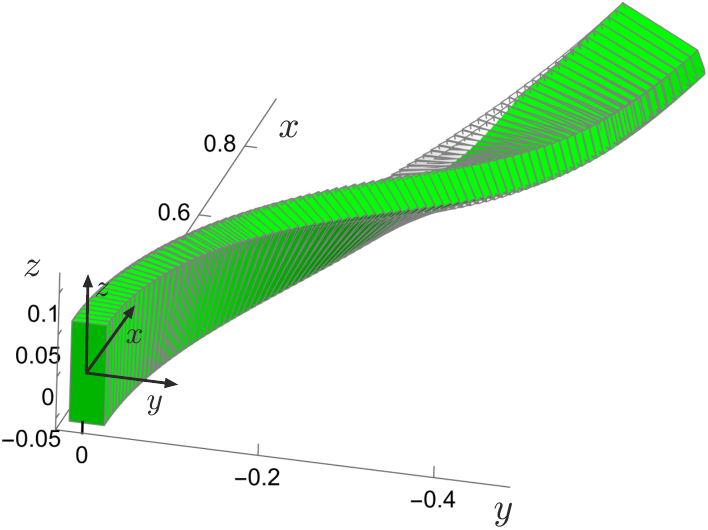}
\end{center}
\caption{%
%
Shape reconstruction of a flexible rod with rectangular cross section. a)
3rd- order (red) and b) 4th-order interpolation (green) are superposed to
the numerically exact solution (gray).}
\label{fig2Point_Beam3}
\end{figure}
%
\end{example}

\vspace{-3ex}

\section{Piecewise Parameterized POE-Splines}

A $k$th-order spline interpolating through $h_{0},h_{1},\ldots ,h_{n}\in G$
is constructed by concatenating 2-point initial value interpolations,
defining $n$ spline segments. The spline curve is described by the global
path parameter $t\in \left[ 0,T\right] $. In segment $i$, a $k$th-order
interpolation (\ref{2PointInt}) between $h_{i-1}$ and $h_{i}$ is introduced.
The range of the path parameter in segment $i$ is $t\in \left[ t_{i-1},t_{i}%
\right] $. A normalized local path parameter for segment $i$ is introduced
as $\tau _{i}\left( t\right) =\left( t-t_{i-1}\right) /\left(
t_{i}-t_{i-1}\right) \in \left[ 0,1\right] $. The local coordinates in
segment $i$ are denoted with $\boldsymbol{\xi }_{i}^{\left[ k\right] }\left(
\tau _{i}\right) $. It is assumed that the local coordinates are zero at the
begin of a segment, i.e. $\boldsymbol{\xi }_{i}^{\left[ k\right] }\left(
0\right) =0$.

\begin{definition}
A $k$th-order $C^{p}$-continuos POE-spline interpolating through $%
h_{0},h_{1},\ldots ,h_{n}\in G$ is a map defined recursively in segment $i$
as%
\begin{equation}
\bar{h}_{i}\left( t\right) =h_{i-1}\exp \boldsymbol{\xi }_{i}^{\left[ k%
\right] }\left( \tau _{i}\left( t\right) \right) ,\ t\in \left[ t_{i-1},t_{i}%
\right]  \label{POESpline}
\end{equation}%
satisfying the matching condition $\bar{h}_{i}\left( t_{i}\right) =h_{i}$,
the continuity conditions $\bar{\mathbf{v}}_{i}^{\left( s\right) }\left(
0\right) =\bar{\mathbf{v}}_{i-1}^{\left( s\right) }\left( 1\right)
,s=0,\ldots ,p-1$, and the initial conditions $\bar{\mathbf{v}}_{1}^{\left(
s\right) }\left( 0\right) =\mathbf{v}_{0}^{\left( s\right) },s=0,\ldots ,p-1$
on the vector fields $\bar{\mathbf{v}}_{i}\left( \tau _{i}\left( t\right)
\right) :=\bar{h}_{i}\left( t\right) ^{-1}\bar{h}_{i}^{\prime }\left(
t\right) $ and its derivatives ($\left( \cdot \right) ^{\prime }$ denotes
derivative w.r.t. $t$).
\end{definition}

\begin{remark}
The map (\ref{POESpline}) can be written as%
\begin{equation}
\bar{h}_{i}\left( t\right) =h_{0}\exp \bar{\boldsymbol{\xi }}_{1}\ldots \exp 
\bar{\boldsymbol{\xi }}_{i-1}\exp \boldsymbol{\xi }_{i}^{\left[ k\right]
}\left( \tau _{i}\left( t\right) \right) ,\ t\in \left[ t_{i-1},t_{i}\right]
\label{POESpline2}
\end{equation}%
with $\bar{\boldsymbol{\xi }}_{i}=\log (\bar{h}_{i-1}^{-1}\bar{h}_{i})$.
This is why (\ref{POESpline}) is referred to as POE-spline. A POE-spline
relies on a \emph{local interpolation} on $G$ around the identity 
%
with local coordinates $\boldsymbol{\xi }_{i}^{\left[ k\right] }$.%
%
\end{remark}

\begin{remark}
In all POE-spline formulations reported in the literature \cite%
{ParkRavani1997,KangPark1998}, it is presumed that $\boldsymbol{\xi }_{i}^{%
\left[ k\right] }\left( 0\right) =0$. That is, $\boldsymbol{\xi }_{i}^{\left[
k\right] }$ are local coordinates on $G$ used to reconstruct the
'incremental' motion within a segment, starting at the identity of $G$. This
concept is also employed in Munthe-Kaas integration schemes on Lie groups 
\cite{MuntheKaas-BIT1998,MuntheKaas1999,IserlesMuntheKaasNrsettZanna2000}.
The consequences of the fact that the spline is constructed piecewise in
terms of local coordinates for the $n$ segments separately will be discussed
in section 5 
.
\end{remark}

\subsection{A $k$th-Order $C^{k-1}$ Continuos POE-Spline\label{secPOESpline}}

The matching conditions are equivalent to $\exp \boldsymbol{\xi }_{i}^{\left[
k\right] }\left( 1\right) =\bar{h}_{i-1}^{-1}\bar{h}_{i}$, which are
satisfied by construction of the 2-point interpolation. It remains to impose
the continuity conditions. The 2-point interpolation in segment $i$ needs
initial values $\bar{\mathbf{v}}_{i}^{\left( s\right) }\left( 0\right)
,s=1,\ldots ,k-1$ (replacing $\mathbf{v}_{0}^{\left( s\right) }$ in (\ref%
{2PointInt})) that match those at the end of segment $i$. The vector field
obtained from the interpolation in segment $i-1$ is 
\begin{equation}
\bar{\mathbf{v}}_{i-1}\left( \tau _{i-1}\right) =\bar{h}_{i-1}\left(
t\right) ^{-1}\bar{h}_{i-1}^{\prime }\left( t\right) =\tfrac{1}{t_{i}-t_{i-1}%
}\mathrm{dexp}_{-\boldsymbol{\xi }_{i-1}^{\left[ k\right] }\left( \tau
_{i-1}\right) }(\boldsymbol{\xi }_{i-1}^{\left[ k\right] \prime }\left( \tau
\right) )  \label{cont1}
\end{equation}%
noting that $\frac{d}{dt}=\frac{1}{t_{i}-t_{i-1}}\frac{d}{d\tau }$. At the
end of segment $i-1$ this is $\bar{\mathbf{v}}_{i-1}\left( 1\right) =\tfrac{1%
}{t_{i}-t_{i-1}}\mathrm{dexp}_{-\bar{\boldsymbol{\xi }}_{i}}(\boldsymbol{\xi 
}_{i-1}^{\left[ k\right] \prime }\left( 1\right) )$, which delivers the
initial value for the 2-point interpolation in segment $i$. The higher-order
continuity conditions are 
\begin{equation}
\bar{\mathbf{v}}_{i}^{\left( s\right) }\left( 0\right) =\tfrac{1}{\left(
t_{i}-t_{i-1}\right) ^{s+1}}\frac{d^{s}}{d\tau ^{s}}\left. \mathrm{dexp}_{-%
\boldsymbol{\xi }_{i-1}^{\left[ k\right] }\left( \tau \right) }(\boldsymbol{%
\xi }_{i-1}^{\left[ k\right] \prime }\left( \tau \right) )\right\vert _{\tau
=1}
\end{equation}%
with $\boldsymbol{\xi }_{i-1}^{\left[ k\right] }\left( 1\right) =\bar{%
\boldsymbol{\xi }}_{i-1}$.

For a 3rd-order spline with 2nd-order continuity, these relations are%
\begin{eqnarray}
\bar{\mathbf{v}}_{i}\left( 0\right) &=&\tfrac{1}{t_{i}-t_{i-1}}\mathrm{dexp}%
_{-\bar{\boldsymbol{\xi }}_{i-1}}(\boldsymbol{\xi }_{i-1}^{\left[ k\right]
\prime }\left( 1\right) ) \\
\bar{\mathbf{v}}_{i}^{\prime }\left( 0\right) &=&\tfrac{1}{\left(
t_{i}-t_{i-1}\right) ^{2}}\left( \mathrm{dexp}_{-\bar{\boldsymbol{\xi }}%
_{i-1}}(\boldsymbol{\xi }_{i-1}^{\left[ k\right] \prime \prime }\left(
1\right) )-(\mathrm{D}_{-\bar{\boldsymbol{\xi }}_{i-1}}\mathrm{dexp})(%
\boldsymbol{\xi }_{i-1}^{\left[ k\right] \prime }\left( 1\right) )(%
\boldsymbol{\xi }_{i-1}^{\left[ k\right] \prime }\left( 1\right) )\right) .
\label{vdot}
\end{eqnarray}%
A 4th-order spline with 3rd-order continuity, must satisfy the additional
condition%
\begin{eqnarray}
\bar{\mathbf{v}}_{i}^{\prime \prime }\left( 0\right) &=&\tfrac{1}{\left(
t_{i}-t_{i-1}\right) ^{2}}%
\Big%
(\mathrm{dexp}_{-\bar{\boldsymbol{\xi }}_{i-1}}(\boldsymbol{\xi }_{i-1}^{%
\left[ k\right] \prime \prime }\left( 1\right) )-2(\mathrm{D}_{-\bar{%
\boldsymbol{\xi }}_{i-1}}\mathrm{dexp})(\boldsymbol{\xi }_{i-1}^{\left[ k%
\right] \prime }\left( 1\right) )(\boldsymbol{\xi }_{i-1}^{\left[ k\right]
\prime \prime }\left( 1\right) )  \notag \\
&&-(\mathrm{D}_{-\bar{\boldsymbol{\xi }}_{i-1}}\mathrm{dexp})(\boldsymbol{%
\xi }_{i-1}^{\left[ k\right] \prime \prime }\left( 1\right) )(\boldsymbol{%
\xi }_{i-1}^{\left[ k\right] \prime }\left( 1\right) )+(\mathrm{D}_{-\bar{%
\boldsymbol{\xi }}_{i-1}}^{2}\mathrm{dexp})(\boldsymbol{\xi }_{i-1}^{\left[ k%
\right] \prime }\left( 1\right) )(\boldsymbol{\xi }_{i-1}^{\left[ k\right]
\prime }\left( 1\right) )(\boldsymbol{\xi }_{i-1}^{\left[ k\right] \prime
}\left( 1\right) )%
\Big%
).
\end{eqnarray}%
The derivatives of $\boldsymbol{\xi }_{i-1}^{\left[ k\right] }\left( \tau
_{i}\right) $ are obtained from (\ref{2PointInt}), particularly from (\ref%
{3PP}) and (\ref{4PP}) for $k=3,4$, respectively. With these continuity
conditions, it is straightforward to derive the 3rd-order $C^{2}$ continuos
spline Algorithm 1, and the 4th-order $C^{3}$ continuos spline Algorithm 2
below. An equivalent 3rd-order algorithm was presented in \cite%
{ParkRavani1997,KangPark1998} for the special case of POE-splines on $%
SO\left( 3\right) $, where 2-point interpolations are derived as minimum
energy curves. That is, the spline algorithm in \cite%
{ParkRavani1997,KangPark1998}, constructed by combining minimum energy
curves, is recovered as the POE-spline constructed from the 2-point
interpolations derived from 3rd-order approximate solutions of the Poisson
equation. 
%
The computation effort of the POE spline algorithm grows linear with the
number of points, i.e. it has linear complexity.%
%

\begin{tabular}[t]{l}
\  \\ 
\textbf{Algorithm 1: 3rd-Order }$C^{2}$\textbf{\ continuos POE-Spline}%
\vspace{1.5ex}
\\ 
\begin{tabular}{ll}
\begin{tabular}[t]{l}
\textbf{1. Input:}%
\vspace{0.5ex}
\\ 
\hspace{2ex}%
\begin{tabular}{ll}
\textbullet
& Knot points $t_{0},t_{1},\ldots ,t_{N}$ \\ 
\textbullet
& Samples $h_{0},h_{1},\ldots ,h_{n}$ \\ 
\textbullet
& Initial values $\mathbf{v}_{0},\dot{\mathbf{v}}_{0}$%
\end{tabular}%
\end{tabular}
& 
\begin{tabular}[t]{l}
\textbf{2. Initialization:}%
\vspace{0.5ex}
\\ 
\hspace{2ex}%
\begin{tabular}{ll}
\textbullet
& $\boldsymbol{\alpha }_{0}^{\ast }:=\mathbf{v}_{0},\ \boldsymbol{\beta }%
_{0}^{\ast }:=\dot{\mathbf{v}}_{0}$ \\ 
\textbullet
& $T_{0}:=1$%
\end{tabular}%
\end{tabular}
\\ 
\begin{tabular}[t]{l}
\textbf{3. Computation of Spline Coefficients:} \\ 
\hspace{2ex}%
\textbf{FOR} $i=1,\ldots ,N$ \textbf{DO}%
\vspace{0.5ex}
\\ 
\hspace{2ex}%
\begin{tabular}[t]{ll}
\textbullet
& $T_{i}:=t_{i}-t_{i-1}$ \\ 
\textbullet
& $\delta _{i}:=T_{i}/T_{i-1}$ \\ 
\textbullet
& $\bar{\boldsymbol{\xi }}_{i}:=\log (h_{i-1}^{-1}h_{i})$ \\ 
\textbullet
& $\boldsymbol{\alpha }_{i-1}:=\delta _{i}\boldsymbol{\alpha }_{i-1}^{\ast }$
\\ 
\textbullet
& $\boldsymbol{\beta }_{i-1}:=\delta _{i}^{2}\boldsymbol{\beta }_{i-1}^{\ast
}$ \\ 
\textbullet
& $\mathbf{a}_{i}:=3\boldsymbol{\xi }_{i}-2\boldsymbol{\alpha }_{i-1}-\frac{1%
}{2}\boldsymbol{\beta }_{i-1}$ \\ 
\textbullet
& $\mathbf{b}_{i}:=6\boldsymbol{\xi }_{i}-6\boldsymbol{\alpha }_{i-1}-2%
\boldsymbol{\beta }_{i-1}$ \\ 
\textbullet
& $\boldsymbol{\alpha }_{i}^{\ast }:=\mathrm{dexp}_{-\bar{\boldsymbol{\xi }}%
_{i}}\mathbf{a}_{i}$ \\ 
\textbullet
& $\boldsymbol{\beta }_{i}^{\ast }:=\mathrm{dexp}_{-\bar{\boldsymbol{\xi }}%
_{i}}\mathbf{b}_{i}-(\mathrm{D}_{-\bar{\boldsymbol{\xi }}_{i}}\mathrm{dexp}%
)\left( \mathbf{a}_{i}\right) \mathbf{a}_{i}$%
\end{tabular}%
\vspace{0.5ex}
\\ 
\ \ \ \ \textbf{END}%
\end{tabular}%
\hspace{-5ex}
& 
\begin{tabular}[t]{l}
\textbf{4. Evaluation} at $t\in \left[ t_{i}-t_{i-1}\right] $, with $\tau
_{i}\left( t\right) =\frac{t-t_{i-1}}{T_{i}}$:%
\vspace{0.5ex}
\\ 
\hspace{2ex}%
$%
\begin{tabular}{rl}
$\boldsymbol{\xi }_{i}^{\left[ 3\right] }{}\left( \tau _{i}\right) $%
\hspace{-2.5ex}
& $=\tau _{i}^{3}\bar{\boldsymbol{\xi }}_{i}+\left( \tau _{i}-\tau
_{i}^{3}\right) \boldsymbol{\alpha }_{i-1}+\frac{1}{2}\left( \tau
_{i}^{2}-\tau _{i}^{3}\right) \boldsymbol{\beta }_{i-1}$ \\ 
$h^{\left[ 3\right] }\left( t\right) $%
\hspace{-2.5ex}
& $=h_{i-1}\exp \boldsymbol{\xi }_{i}^{\left[ 3\right] }{}\left( \tau
_{i}\right) $%
\end{tabular}%
$%
\end{tabular}%
\end{tabular}
\\ 
\ 
\end{tabular}

\begin{tabular}[t]{l}
\textbf{Algorithm 2: 4th-Order }$C^{3}$\textbf{\ continuos POE-Spline}%
\vspace{1.5ex}
\\ 
\begin{tabular}[t]{l}
\textbf{1. Input:}%
\vspace{0.5ex}
\\ 
\hspace{2ex}%
\begin{tabular}{ll}
\textbullet
& Knot points $t_{0},t_{1},\ldots ,t_{N}$ \\ 
\textbullet
& Samples $h_{0},h_{1},\ldots ,h_{n}$ \\ 
\textbullet
& Initial values $\mathbf{v}_{0},\dot{\mathbf{v}}_{0},\ddot{\mathbf{v}}_{0}$%
\end{tabular}%
\end{tabular}%
\hspace{8ex}%
\begin{tabular}[t]{l}
\textbf{2. Initialization:}%
\vspace{0.5ex}
\\ 
\hspace{2ex}%
\begin{tabular}{ll}
\textbullet
& $\boldsymbol{\alpha }_{0}^{\ast }:=\mathbf{v}_{0},\ \boldsymbol{\beta }%
_{0}^{\ast }:=\dot{\mathbf{v}}_{0},\ \boldsymbol{\gamma }_{0}^{\ast }:=\ddot{%
\mathbf{v}}_{0}$ \\ 
\textbullet
& $T_{0}:=1$%
\end{tabular}%
\end{tabular}%
\vspace{1.5ex}
\\ 
\begin{tabular}[t]{l}
\textbf{3. Computation of Spline Coefficients:} \\ 
\hspace{2ex}%
\textbf{FOR} $i=1,\ldots ,N$ \textbf{DO}%
\vspace{0.5ex}
\\ 
\hspace{2ex}%
\begin{tabular}[t]{ll}
\textbullet
& $T_{i}:=t_{i}-t_{i-1}$ \\ 
\textbullet
& $\delta _{i}:=T_{i}/T_{i-1}$ \\ 
\textbullet
& $\bar{\boldsymbol{\xi }}_{i}:=\log (h_{i-1}^{-1}h_{i})$ \\ 
\textbullet
& $\boldsymbol{\alpha }_{i-1}:=\delta _{i}\boldsymbol{\alpha }_{i-1}^{\ast }$
\\ 
\textbullet
& $\boldsymbol{\beta }_{i-1}:=\delta _{i}^{2}\boldsymbol{\beta }_{i-1}^{\ast
}$ \\ 
\textbullet
& $\boldsymbol{\gamma }_{i-1}:=\delta _{i}^{3}\boldsymbol{\gamma }%
_{i-1}^{\ast }$ \\ 
\textbullet
& $\mathbf{a}_{i}:=4\boldsymbol{\xi }_{i}-3\boldsymbol{\alpha }_{i-1}-%
\boldsymbol{\beta }_{i-1}-\frac{1}{12}\left[ \boldsymbol{\alpha }_{i-1},%
\boldsymbol{\beta }_{i-1}\right] -\frac{1}{6}\boldsymbol{\gamma }_{i-1}$ \\ 
\textbullet
& $\mathbf{b}_{i}:=12\boldsymbol{\xi }_{i}-12\boldsymbol{\alpha }_{i-1}-5%
\boldsymbol{\beta }_{i-1}-\frac{1}{2}\left[ \boldsymbol{\alpha }_{i-1},%
\boldsymbol{\beta }_{i-1}\right] -\boldsymbol{\gamma }_{i-1}$ \\ 
\textbullet
& $\mathbf{c}_{i}:=24\boldsymbol{\xi }_{i}-24\boldsymbol{\alpha }_{i-1}-12%
\boldsymbol{\beta }_{i-1}-\frac{3}{2}\left[ \boldsymbol{\alpha }_{i-1},%
\boldsymbol{\beta }_{i-1}\right] -3\boldsymbol{\gamma }_{i-1}$ \\ 
\textbullet
& $\boldsymbol{\alpha }_{i}^{\ast }:=\mathrm{dexp}_{-\bar{\boldsymbol{\xi }}%
_{i}}\mathbf{a}_{i}$ \\ 
\textbullet
& $\boldsymbol{\beta }_{i}^{\ast }:=\mathrm{dexp}_{-\bar{\boldsymbol{\xi }}%
_{i}}\mathbf{b}_{i}-(\mathrm{D}_{-\bar{\boldsymbol{\xi }}_{i}}\mathrm{dexp}%
)\left( \mathbf{a}_{i}\right) \mathbf{a}_{i}$ \\ 
\textbullet
& $\boldsymbol{\gamma }_{i}^{\ast }:=\mathrm{dexp}_{-\bar{\boldsymbol{\xi }}%
_{i}}\mathbf{c}_{i}-2(\mathrm{D}_{-\bar{\boldsymbol{\xi }}_{i}}\mathrm{dexp}%
\left( \mathbf{a}_{i}\right) \mathbf{b}_{i}-(\mathrm{D}_{-\bar{\boldsymbol{%
\xi }}_{i}}\mathrm{dexp})\left( \mathbf{b}_{i}\right) \mathbf{a}_{i}+(%
\mathrm{D}_{-\bar{\boldsymbol{\xi }}_{i}}^{2}\mathrm{dexp})\left( \mathbf{a}%
_{i}\right) \left( \mathbf{a}_{i}\right) \mathbf{a}_{i}$%
\end{tabular}%
\vspace{0.5ex}
\\ 
\ \ \ \ \textbf{END}%
\end{tabular}%
\vspace{1.5ex}
\\ 
\begin{tabular}[t]{l}
\textbf{4. Evaluation} at $t\in \left[ t_{i}-t_{i-1}\right] $, with $\tau
_{i}\left( t\right) =\frac{t-t_{i-1}}{T_{i}}$:%
\vspace{0.5ex}
\\ 
\hspace{2ex}%
\begin{tabular}{rl}
$\boldsymbol{\xi }_{i}^{\left[ 4\right] }{}\left( \tau _{i}\right) $%
\hspace{-2.5ex}
& $=\tau _{i}^{4}\bar{\boldsymbol{\xi }}_{i}+\left( \tau _{i}-\tau
_{i}^{4}\right) \boldsymbol{\alpha }_{i-1}+\frac{1}{2}\left( \tau
_{i}^{2}-\tau _{i}^{4}\right) \boldsymbol{\beta }_{i-1}+\frac{1}{6}\left(
\tau _{i}^{3}-\tau _{i}^{4}\right) \left( \boldsymbol{\gamma }_{i-1}+\frac{1%
}{2}\left[ \boldsymbol{\alpha }_{i-1},\boldsymbol{\beta }_{i-1}\right]
\right) $ \\ 
$h^{\left[ 4\right] }\left( t\right) $%
\hspace{-2.5ex}
& $=h_{i-1}\exp \boldsymbol{\xi }_{i}^{\left[ 4\right] }{}\left( \tau
_{i}\right) $%
\end{tabular}%
\end{tabular}%
\end{tabular}

\newpage

\subsection{A $k$th-Order $C^{k-2}$ Continuos POE-Spline with given
Velocities at Interpolation Points\label{secPOEVel}}

The POE-spline in section \ref{secPOESpline} allows prescribing the initial
velocity and derivatives. It does not admit prescribing terminal values. It
is often desirable, however, to prescribe the terminal velocity, and
moreover to interpolate through given states, i.e. given $h$ and $\mathbf{v}$
at the sample points. This is addressed in the following.

In addition to $h_{0},h_{1},\ldots ,h_{n}\in G$, now the velocities $\mathbf{%
v}_{0},\mathbf{v}_{1},\ldots ,\mathbf{v}_{n}\in \mathfrak{g}$ at the
interpolation points are given. A POE-spline is derived by concatenating
2-point boundary value interpolations as they allow for prescribing $\mathbf{%
v}$ and its derivatives at the start and terminal end of a segment, i.e. at
the knot points. The 3rd- and 4th-order POE-splines are presented here. The
general $k$th-order case becomes rather involved, and seems not to have
practical relevance.

The 3rd-order interpolation (\ref{3PPB}) is used within a segment. In
segment $i$ with $\tau \in \left[ t_{i-1},t_{i}\right] $, the initial
velocity is $\mathbf{v}_{i-1}$ and terminal velocity is $\mathbf{v}_{i}$,
replacing $\mathbf{v}_{0}$ and $\mathbf{v}_{1}$ in (\ref{3PPB}). Therewith,
the $C^{1}$ continuity of the curve obtained by concatenating $n$ such
interpolations is ensured. The corresponding computational scheme is shown
as Algorithm 3. A 4th-order spline is constructed using (\ref{4PPB}) to
interpolate in segment $i$. The initial and terminal values in (\ref{4PPB})
are set to $\mathbf{v}_{i-1}$ and $\mathbf{v}_{i}$ at the knot points of
segment $i$. The derivative $\mathbf{v}_{i-1}^{\prime }$ at the start of
segment $i$ is again calculated from the interpolation in segment $i-1$
using (\ref{cont1}). The so obtained spline curve is $C^{2}$ continuos. This
gives rise to Algorithm 4.

\begin{tabular}[t]{l}
\  \\ 
\textbf{Algorithm 3: 3rd-Order }$C^{1}$\textbf{\ continuos POE Spline with
given Velocities}%
\vspace{1.5ex}
\\ 
\begin{tabular}{ll}
\begin{tabular}[t]{l}
\textbf{1. Input:}%
\vspace{0.5ex}
\\ 
\hspace{2ex}%
\begin{tabular}{ll}
\textbullet
& Knot points $t_{0},t_{1},\ldots ,t_{N}$ \\ 
\textbullet
& Samples $h_{0},h_{1},\ldots ,h_{n}$ \\ 
\textbullet
& Velocities $\mathbf{v}_{0},\mathbf{v}_{1},\ldots ,\mathbf{v}_{n}$%
\end{tabular}%
\end{tabular}
& 
\begin{tabular}[t]{l}
\textbf{2. Initialization:}%
\vspace{0.5ex}
\\ 
\hspace{2ex}%
\begin{tabular}{lll}
\textbullet
& $\boldsymbol{\alpha }_{0}^{\ast }:=\mathbf{v}_{0}$ & 
\textbullet%
$\ \ 
\color[rgb]{0.7,0,0}%
\bar{\boldsymbol{\xi }}_{0}=\mathbf{0}%
\color{black}%
$ \\ 
\textbullet
& $T_{0}:=1$ & 
\end{tabular}%
\end{tabular}
\\ 
\begin{tabular}[t]{l}
\textbf{3. Computation of Spline Coefficients:} \\ 
\hspace{2ex}%
\textbf{FOR} $i=1,\ldots ,N$ \textbf{DO}%
\vspace{0.5ex}
\\ 
\hspace{2ex}%
\begin{tabular}[t]{ll}
\textbullet
& $T_{i}:=t_{i}-t_{i-1}$ \\ 
\textbullet
& $\delta _{i}:=T_{i}/T_{i-1}$ \\ 
\textbullet
& $\bar{\boldsymbol{\xi }}_{i}:=\log (h_{i-1}^{-1}h_{i})$ \\ 
\textbullet
& $\boldsymbol{\alpha }_{i}^{\ast }:=T_{i}%
\color[rgb]{0.7,0,0}%
\mathbf{v}_{i}%
\color{black}%
$ \\ 
\textbullet
& $\boldsymbol{\alpha }_{i-1}:=T_{i}%
\color[rgb]{0.7,0,0}%
\mathbf{v}_{i-1}%
\color{black}%
$ \\ 
\textbullet
& $\mathbf{a}_{i}:=\mathrm{dexp}_{-\bar{\boldsymbol{\xi }}_{i}}^{-1}%
\boldsymbol{\alpha }_{i}^{\ast }$%
\end{tabular}%
\vspace{0.5ex}%
\end{tabular}
& 
\begin{tabular}[t]{l}
\textbf{4. Evaluation} at $t\in \left[ t_{i}-t_{i-1}\right] $, with $\tau
_{i}\left( t\right) =\frac{t-t_{i-1}}{T_{i}}$:%
\vspace{0.5ex}
\\ 
\hspace{2ex}%
$%
\begin{tabular}{rl}
$\boldsymbol{\xi }_{i}^{\left[ 3\right] }{}\left( \tau _{i}\right) $%
\hspace{-2.5ex}
& $=\left( 3\tau _{i}^{2}-2\tau _{i}^{3}\right) 
\color[rgb]{0.7,0,0}%
\bar{\boldsymbol{\xi }}_{i}%
\color{black}%
+\tau _{i}\left( \tau _{i}-1\right) ^{2}%
\color[rgb]{0.7,0,0}%
\boldsymbol{\alpha }_{i}%
\color{black}%
+\left( \tau _{i}^{3}-\tau _{i}^{2}\right) \mathbf{a}_{i}$ \\ 
$h^{\left[ 3\right] }\left( t\right) $%
\hspace{-2.5ex}
& $=h_{i-1}\exp \boldsymbol{\xi }_{i}^{\left[ 3\right] }{}\left( \tau
_{i}\right) $%
\end{tabular}%
$%
\end{tabular}%
\end{tabular}
\\ 
\ 
\end{tabular}

\begin{tabular}[t]{l}
\textbf{Algorithm 4: 4th-Order }$C^{2}$\textbf{\ continuos POE Spline with
given Velocities}%
\vspace{1.5ex}
\\ 
\begin{tabular}{ll}
\begin{tabular}[t]{l}
\textbf{1. Input:}%
\vspace{0.5ex}
\\ 
\hspace{2ex}%
\begin{tabular}{ll}
\textbullet
& Knot points $t_{0},t_{1},\ldots ,t_{N}$ \\ 
\textbullet
& Samples $h_{0},h_{1},\ldots ,h_{n}$ \\ 
\textbullet
& Velocities $\mathbf{v}_{0},\mathbf{v}_{1},\ldots ,\mathbf{v}_{n}$ \\ 
\textbullet
& Initial derivative $\dot{\mathbf{v}}_{0}$%
\end{tabular}%
\end{tabular}
& 
\begin{tabular}[t]{l}
\textbf{2. Initialization:}%
\vspace{0.5ex}
\\ 
\hspace{2ex}%
\begin{tabular}{ll}
\textbullet
& $\boldsymbol{\alpha }_{0}^{\ast }:=\mathbf{v}_{0}$ \\ 
\textbullet
& $\boldsymbol{\beta }_{0}^{\ast }:=\mathbf{v}_{0}^{\prime }$ \\ 
\textbullet
& $T_{0}:=1$%
\end{tabular}%
\end{tabular}
\\ 
\begin{tabular}[t]{l}
\textbf{3. Computation of Spline Coefficients:} \\ 
\textbf{FOR} $i=1,\ldots ,N$ \textbf{DO}%
\vspace{0.5ex}
\\ 
\hspace{2ex}%
\begin{tabular}[t]{ll}
\textbullet
& $T_{i}:=t_{i}-t_{i-1}$ \\ 
\textbullet
& $\delta _{i}:=T_{i}/T_{i-1}$ \\ 
\textbullet
& $\bar{\boldsymbol{\xi }}_{i}:=\log (h_{i-1}^{-1}h_{i})$ \\ 
\textbullet
& $\boldsymbol{\alpha }_{i}^{\ast }:=T_{i}%
\color[rgb]{0.7,0,0}%
\mathbf{v}_{i}%
\color{black}%
$ \\ 
\textbullet
& $\boldsymbol{\alpha }_{i-1}:=T_{i}%
\color[rgb]{0.7,0,0}%
\mathbf{v}_{i-1}%
\color{black}%
$ \\ 
\textbullet
& $\boldsymbol{\beta }_{i-1}:=\delta _{i}^{2}%
\color[rgb]{0.7,0,0}%
\boldsymbol{\beta }_{i}^{\mathbf{\ast }}%
\color{black}%
$ \\ 
\textbullet
& $\mathbf{a}_{i}:=\mathrm{dexp}_{-\bar{\boldsymbol{\xi }}_{i}}^{-1}%
\boldsymbol{\alpha }_{i}^{\ast }$ \\ 
\textbullet
& $\mathbf{b}_{i}:=-12\bar{\boldsymbol{\xi }}_{i}+6\boldsymbol{\alpha }%
_{i-1}+\boldsymbol{\beta }_{i-1}+6\mathbf{a}_{i}$ \\ 
\textbullet
& $\boldsymbol{\beta }_{i}^{\ast }:=%
\color[rgb]{0.7,0,0}%
\mathrm{dexp}_{-\bar{\boldsymbol{\xi }}_{i}}\mathbf{b}_{i}-(\mathrm{D}_{-%
\bar{\boldsymbol{\xi }}_{i}}\mathrm{dexp})\left( \mathbf{a}_{i}\right) 
\mathbf{a}_{i}$%
\end{tabular}%
\vspace{0.5ex}
\\ 
\ \ \ \ \textbf{END}%
\end{tabular}
& 
\begin{tabular}[t]{l}
\textbf{4. Evaluation} at $t\in \left[ t_{i}-t_{i-1}\right] $, with $\tau
_{i}\left( t\right) =\frac{t-t_{i-1}}{T_{i}}$:%
\vspace{0.5ex}
\\ 
\hspace{2ex}%
$%
\begin{tabular}{rl}
$\boldsymbol{\xi }_{i}^{\left[ 4\right] }{}\left( \tau _{i}\right) $%
\hspace{-2.5ex}
& $=\left( 4\tau _{i}^{3}-3\tau _{i}^{4}\right) \bar{\boldsymbol{\xi }}%
_{i}+\left( 1-\tau _{i}\right) ^{2}\left( \tau _{i}+2\tau _{i}^{2}\right) 
\boldsymbol{\alpha }_{i-1}+\frac{1}{2}\tau _{i}^{2}\left( \tau _{i}-1\right)
^{2}\boldsymbol{\beta }_{i-1}+\left( \tau ^{4}-\tau ^{3}\right) \mathbf{a}%
_{i}$ \\ 
$h^{\left[ 4\right] }\left( t\right) $%
\hspace{-2.5ex}
& $=h_{i-1}\exp \boldsymbol{\xi }_{i}^{\left[ 4\right] }{}\left( \tau
_{i}\right) $%
\end{tabular}%
$%
\end{tabular}%
\end{tabular}%
\end{tabular}
\vspace{4ex}

\begin{example}
Interpolation of the spatial rigid body motion through a set of four given
poses $h_{0},\ldots ,h_{3}$ with desired velocities $\mathbf{v}_{0},\ldots ,%
\mathbf{v}_{3}$, within a duration of $T=1$\thinspace s, is considered. The
rigid body poses are represented as elements of the direct product Lie group 
$SO\left( 3\right) \times {\mathbb{R}}^{3}$. This group is compact, in
contrast to $SE\left( 3\right) $, thus its one-parameter subgroups define
geodesics w.r.t. a bi-invariant metric on $SO\left( 3\right) $ and ${\mathbb{%
R}}^{3}$, which are sometimes called double geodesics. This representation
implies that rotations and translations are decoupled. A typical element is $%
h=\left( \mathbf{R},\mathbf{r}\right) $, and the corresponding
left-invariant vector field is $\mathbf{v}=\left( \boldsymbol{\omega },\dot{%
\mathbf{r}}\right) $, where $\boldsymbol{\omega }\in so\left( 3\right) $ is
the body-fixed angular velocity defined via $\tilde{\boldsymbol{\omega }}=%
\mathbf{R}^{T}\dot{\mathbf{R}}\in so\left( 3\right) $, and $\dot{\mathbf{r}}%
\in {\mathbb{R}}^{3}$ is the translational velocity of the body-fixed
reference frame, and $\mathbf{v}$ is called the twist in mixed
representation \cite{MUBOScrew1}. Multiplication on this group is defined as 
$h_{1}h_{2}=\left( \mathbf{R}_{1},\mathbf{r}_{1}\right) \cdot \left( \mathbf{%
R}_{2},\mathbf{r}_{2}\right) =\left( \mathbf{R}_{1}\mathbf{R}_{2},\mathbf{r}%
_{1}+\mathbf{r}_{2}\right) $. The right-trivialized differential on $%
SO\left( 3\right) $ admits the closed form (\ref{dexpSO3}). Pose $h_{i}$ is
represented by a frame $\mathcal{F}_{i}$, as shown in Fig. \ref%
{figSpaceCurve3Vel}. They are defined by position vectors $\mathbf{r}_{i}$
(represented in $\mathcal{F}_{0}$), and the rotation matrices $\mathbf{R}%
_{0i}$ (transforming from $\mathcal{F}_{i}$ to $\mathcal{F}_{0}$) as follows
(all values in SI units)%
\begin{equation*}
\mathbf{r}_{1}=\left[ 
\begin{smallmatrix}
1 \\ 
4 \\ 
1%
\end{smallmatrix}%
\right] ,\mathbf{r}_{2}=\left[ 
\begin{smallmatrix}
4 \\ 
4 \\ 
4%
\end{smallmatrix}%
\right] ,\mathbf{r}_{3}=\left[ 
\begin{smallmatrix}
8 \\ 
4 \\ 
1%
\end{smallmatrix}%
\right] ,\mathbf{R}_{01}=\left[ 
\begin{smallmatrix}
0 & -1 & 0 \\ 
1 & 0 & 0 \\ 
0 & 0 & 1%
\end{smallmatrix}%
\right] ,\mathbf{R}_{02}=\left[ 
\begin{smallmatrix}
0 & 0 & 1 \\ 
1 & 0 & 0 \\ 
0 & 1 & 0%
\end{smallmatrix}%
\right] ,\mathbf{R}_{03}=\mathbf{I}.\vspace{2ex}
\end{equation*}%
The angular velocity is required to be zero at all four points, $\boldsymbol{%
\omega }_{i}=0,i=0,\ldots ,3$. The initial and terminal linear velocity is
zero, $\dot{\mathbf{r}}_{0}=\dot{\mathbf{r}}_{3}=\mathbf{0}$, and at the via
points is prescribed as $\dot{\mathbf{r}}_{1}=\left[ 10,0,0\right] ^{T}$ and 
$\dot{\mathbf{r}}_{2}=\left[ 0,0,10\right] ^{T}$, and the initial velocity
as $\dot{\mathbf{r}}_{0}=\mathbf{0}$. The results obtained with the 3rd- and
4th-order spline are shown in Fig. \ref{figSpaceCurve3Vel}. For the
4th-order spline, the derivative $\mathbf{v}_{0}^{\prime }$ is prescribed
with zero angular acceleration $\dot{\boldsymbol{\omega }}_{0}=\mathbf{0}$
and linear acceleration $\ddot{\mathbf{r}}_{0}=\left[ -10,0,0\right] ^{T}$.
Shown is the moving frame at equidistant samples of $t\in \left[ 0,T\right] $%
. The difference of the 3rd- and 4th-order curve are due to fact that the
4th-order curve respects continuity of the jerk $\mathbf{v}^{\prime \prime }$%
. The special situation when the velocity at all points is set to zero is
shown in Fig. \ref{figSpaceCurve3Vel_Vvec=0}. Then, the 3rd-order spline
interpolation reduces to the concatenation of 3rd-order 2-point
interpolations with matching velocities. When using the 4th-order
interpolations, additionally the accelerations match at the via points,
explaining the apparent difference of the curves. 
%
For comparison, the motion curve constructed with a De Casteljau algorithm
on $SO(3)\times \mathbb{R}^{3}$ is shown in Fig. \ref%
{figSpaceCurveDeCasteljau}. The attitude, i.e. curve in $SO(3)$, is
constructed with algorithm 9 (appendix B), and the position, i.e. curve in
vector space $\mathbb{R}^{3}$, is constructed with the classical De
Casteljau algorithm. It would require a B\'{e}zier-spline algorithm to
construct a motion passing through the two intermediate configurations,
which is still topic of research \cite{Popiel2006,Huber2017}.%
%
\begin{figure}[h]
\centerline{
\hfill
a)\includegraphics[draft=false,height=7.5cm]{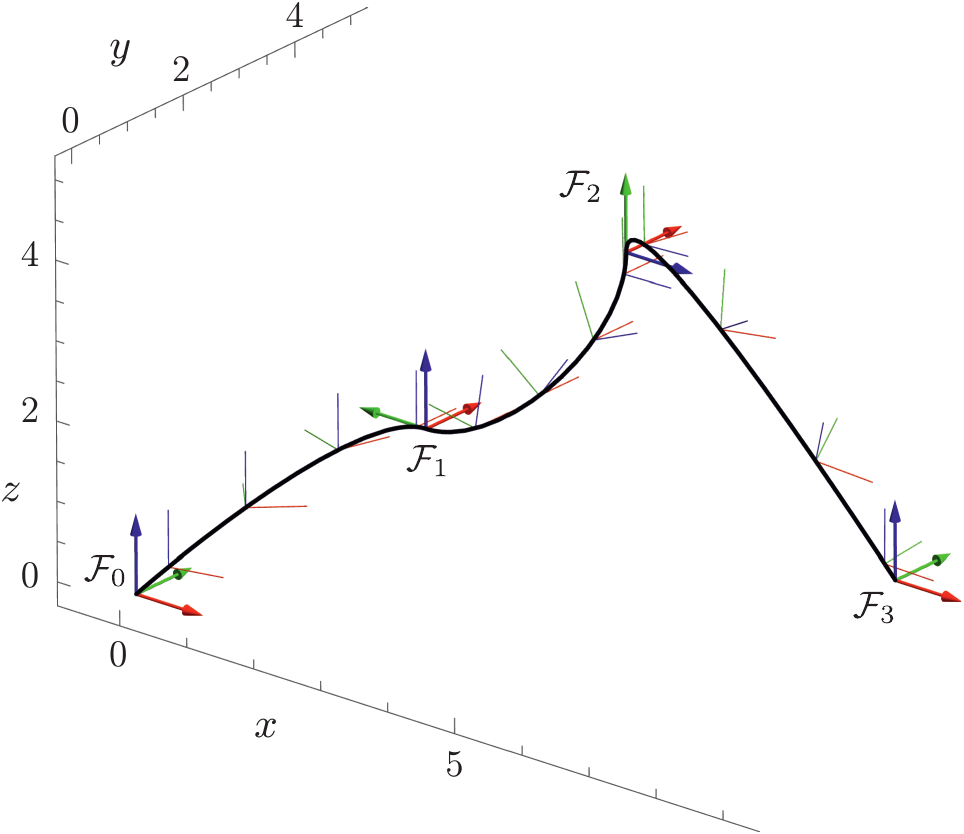}
\hfill
b)\includegraphics[draft=false,height=7.5cm]{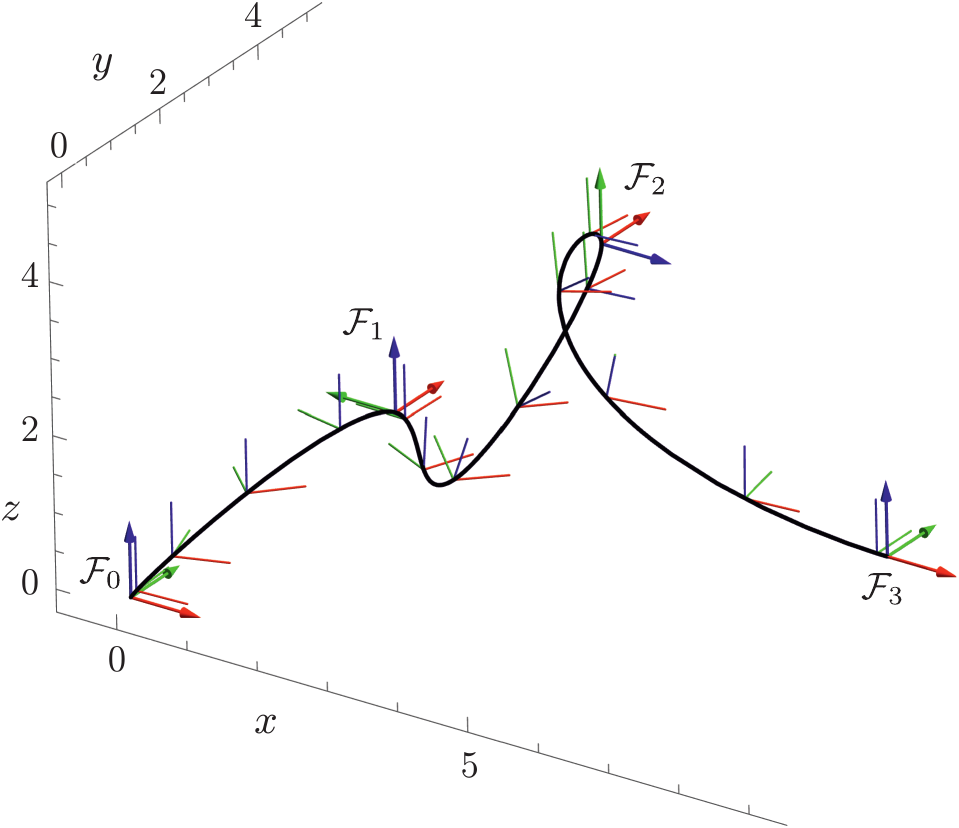}
\hfill\,
}
\caption{%
%
Interpolation of a spatial rigid body motion through four given frame
locations $\mathcal{F}_{0},\ldots ,\mathcal{F}_{3}$, with prescribed
velocities. a) Results when using 3rd-order, and b) when using the 4th-order
spline. $\mathcal{F}_{0}$ is the global reference frame.}
\label{figSpaceCurve3Vel}
\end{figure}
\begin{figure}[h]
\centerline{
\hfill
a)\includegraphics[draft=false,height=5.7cm]{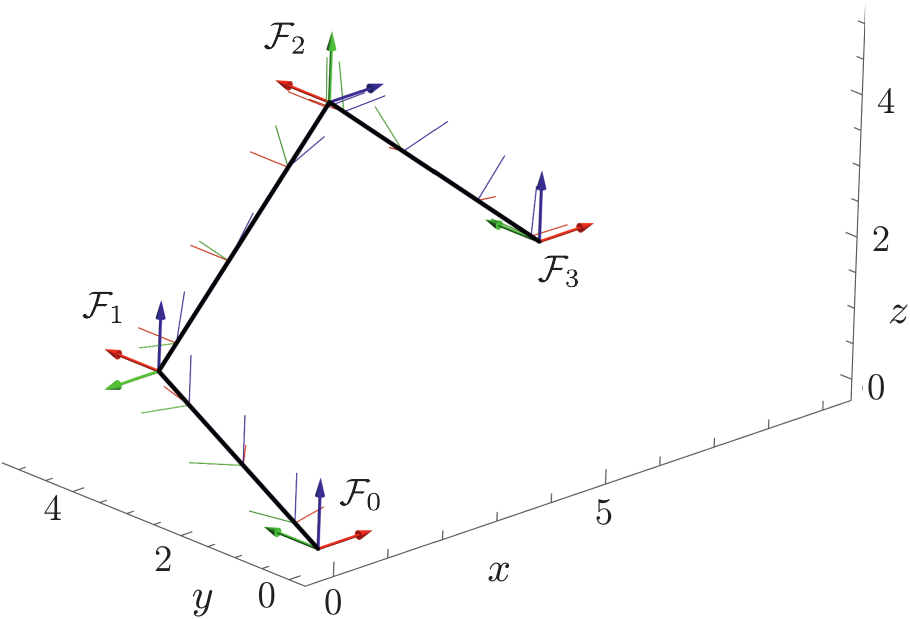}
\hfill
b)\includegraphics[draft=false,height=5.7cm]{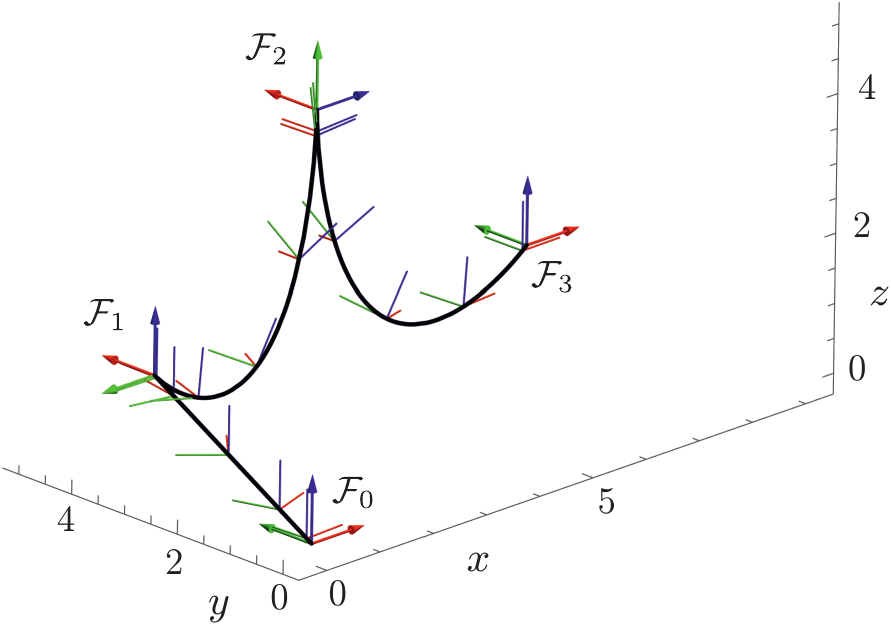}
\hfill\,
}
\caption{%
%
Interpolation of a spatial rigid body motion through four given frame
locations $\mathcal{F}_{0},\ldots ,\mathcal{F}_{3}$ when the velocity is
zero at the given frame locations, using a) 3rd-order, and b) 4th-order
spline. $\mathcal{F}_{0}$ is the global reference frame.}
\label{figSpaceCurve3Vel_Vvec=0}
\end{figure}
\begin{figure}[h]
\centerline{
\includegraphics[draft=false,height=5cm]{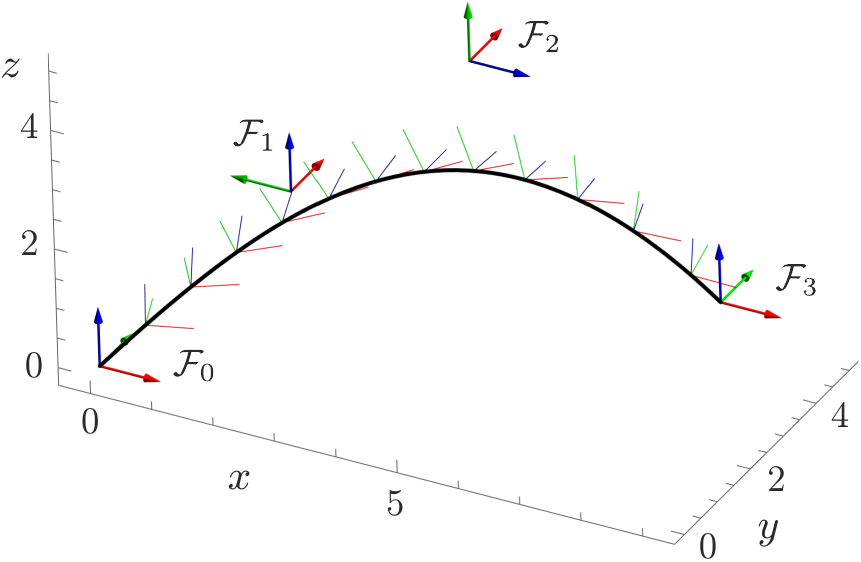}
}
\caption{%
%
B{\'{e}}zier curve interpolating between start and end pose, $\mathcal{F}%
_{0} $ and $\mathcal{F}_{3}$, of the spatial rigid body motion constructed
with the De Casteljau algorithm on $SO(3)\times \mathbb{R}^{3}$ using
intermediate frames $\mathcal{F}_{1}$ and $\mathcal{F}_{2}$ as control
points.}
\label{figSpaceCurveDeCasteljau}
\end{figure}
\end{example}

\newpage

\section{A Limitation of POE-Splines\label{secLimitation}}

A limitation inherent to $k$th-order POE-spline with zero initial local
coordinates $\boldsymbol{\xi }_{i}^{\left[ k\right] }{}\left( 0\right) =%
\mathbf{0}$ is that they cannot exactly reconstruct a $k$th-order curve on $%
G $, i.e. a curve $h\left( t\right) =h_{0}\exp \boldsymbol{\xi }{}\left(
t\right) $ where $\boldsymbol{\xi }{}\left( t\right) $ is a $k$th-order
polynomial in $t$. Assume that the curve is to be interpolated through $%
h_{i}=h\left( t_{i}\right) ,i=0,\ldots ,n$. In interval $t\in \left[
t_{i-1},t_{i}\right] $, the curve can be expressed as $h\left( t\right)
=h_{0}\exp \boldsymbol{\xi }{}\left( t\right) =h_{0}\exp \boldsymbol{\xi }%
{}\left( t_{i-1}\right) \exp \boldsymbol{\eta {}}\left( t\right)
=h_{i-1}\exp \boldsymbol{\eta {}}\left( t\right) $. The coordinates $%
\boldsymbol{\xi }{}\left( t\right) $ and $\boldsymbol{\eta {}}\left(
t\right) $ are related by the BCH formula as%
\begin{eqnarray}
\boldsymbol{\eta {}}\left( t\right) &=&\boldsymbol{\xi }{}\left( t\right) -%
\boldsymbol{\xi }{}\left( t_{i-1}\right) -\frac{1}{2}\left[ \boldsymbol{\xi }%
{}\left( t_{i-1}\right) ,\boldsymbol{\xi }{}\left( t\right) \right]
\label{Limit} \\
&&+\frac{1}{12}\left[ \boldsymbol{\xi }{}\left( t_{i-1}\right) ,\left[ 
\boldsymbol{\xi }{}\left( t_{i-1}\right) ,\boldsymbol{\xi }{}\left( t\right) %
\right] \right] -\frac{1}{12}\left[ \boldsymbol{\xi }{}\left( t\right) ,%
\left[ \boldsymbol{\xi }{}\left( t\right) ,\boldsymbol{\xi }{}\left(
t_{i-1}\right) \right] \right] +\ldots .  \notag
\end{eqnarray}%
If $\boldsymbol{\xi }{}\left( t\right) $ is a $k$th-order polynomial in $t$,
then $\boldsymbol{\eta }\left( t\right) $ is generally not a $k$th-order
polynomial, which requires all nested Lie brackets in (\ref{Limit}) to
vanish. In case of standard spline interpolation on vector spaces, and on
general Abelian groups, this condition is satisfied, of course.

\begin{proposition}
\label{propPOELimit}A $k$th-order POE-spline can only exactly reconstruct a
curve of degree $k$ defined on a 1-parameter subgroup of $G$.
\end{proposition}

This limitation has not been addressed in the literature. It is particularly
relevant when a curve in $G$ is to be interpolated, rather than to be
generated from given via point. Besides this issue, another limitation is
that the initial pose $h_{0}$ must be known beforehand, rather then being
the outcome of the interpolation.

\begin{example}
\label{exaLimit}To demonstrate proposition \ref{propPOELimit}, and thus
lemma \ref{lemExactRec}, the interpolation of a spatial rotation in $%
SO\left( 3\right) $ is considered. The curve to be interpolated in $SO(3)$
is expressed as $h\left( t\right) =\exp (\boldsymbol{\xi }\left( t\right) )$%
, where $h\left( t\right) \in SO\left( 3\right) $ is the rotation matrix,
and $\boldsymbol{\xi }\left( t\right) $ the instantaneous rotation axis
times angle, and the exponential map is the Euler-Rodrigues formula (\ref%
{expSO3}). First consider a motion in the 1-parameter subgroup of rations
about the constant (for simplicity non-unit) axis $\boldsymbol{\xi }_{0}=%
\left[ 0.5,1.5,1\right] \in \mathbb{R}^{3}$. The particular motion to be
interpolated is described by $\boldsymbol{\xi }\left( t\right)
=(t^{3}-t^{2}+3t)\boldsymbol{\xi }_{0}$, which is cubic in $t$. As input to
the spline interpolation, $n=11$ equidistant samples $h_{0},\ldots
,h_{10}\in SO\left( 3\right) $ are computed along the curve 
%
at $\tau _{i},i=0,\ldots ,1$ 
%
giving rise to a 10-segment spline. The 3rd-order cubic POE-spline with
given initial values (Algorithm 1) and with given velocities at the sample
points (Algorithm 3) are applied. For the latter, also the corresponding
angular velocities $\mathbf{v}_{0},\ldots ,\mathbf{v}_{10}\in so\left(
3\right) $ are computed 
%
as $\mathbf{v}_{i}=h^{-1}\left( \tau _{i}\right) \dot{h}\left( \tau
_{i}\right) $.%
%
\newline
The interpolation error is computed as $\varepsilon =\left\Vert \log \left(
\Delta h\right) \right\Vert $, with $\Delta h=h^{-1}\left( t\right) h^{\left[
3\right] }\left( t\right) $, where $h^{\left[ 3\right] }\left( t\right)
=\exp (\boldsymbol{\xi }^{\left[ 3\right] }\left( t\right) )$ is the curve
obtained with the 3rd-order interpolation. Fig. \ref{figOneParam_3rd}a)
shows the result obtained with the two 3rd-order interpolations. Clearly,
this 3rd-order motion can be reconstructed (numerically) exactly by the
3rd-order spline. Now consider the spatial rotation described by $%
\boldsymbol{\xi }\left( t\right) =t\boldsymbol{\xi }_{1}+t^{3}\boldsymbol{%
\xi }_{2}$, with $\boldsymbol{\xi }_{1}=\left[ 0.1,0,0.2\right] \in \mathbb{R%
}^{3}$ and $\boldsymbol{\xi }_{2}=\left[ 0,1.5,0\right] \in \mathbb{R}^{3}$.
It is obvious from Fig. \ref{figOneParam_3rd}b) that the 3rd-order spline
cannot reconstruct this cubic motion. 
%
Also the De Casteljau algorithm 9 (appendix B) is applied for comparison.
For both motions the error magnitude is about 0.06 rad. It is therefore not
shown in the figures.
\end{example}

\begin{figure}[h]
\hfill a)%
\includegraphics[draft=false,height=5cm]{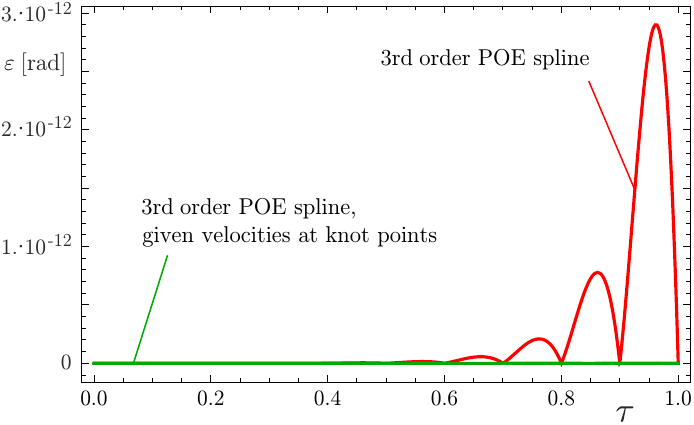}
\hfill b)%
\includegraphics[draft=false,height=5cm]{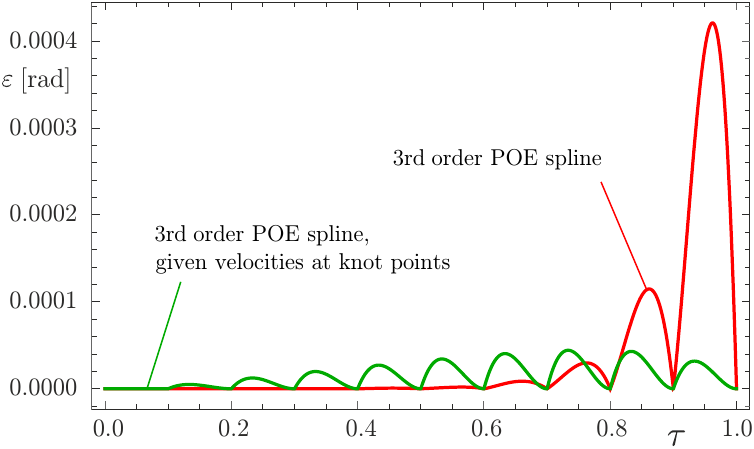}
\hfill\ 
\caption{Error when interpolating a) a cubic 1-parameter motion in $SO(3)$
and b) a general cubic motion in $SO(3)$ with a 3rd-order POE spline.}
\label{figOneParam_3rd}
\end{figure}
%

\section{2-Point Interpolation with Non-Zero Initial Value $\protect\xi %
\left( 0\right) \neq 0$\label{sec2PointNonZero}}

If $\boldsymbol{\xi }{}\left( 0\right) $ is non-zero, the series solution (%
\ref{xi}) does not apply. The following generalized the series solution,
with (\ref{xi}) as special case.

\begin{theorem}
\label{theoApprox}i) The solution of the left-invariant normalized Poisson
equation (\ref{Poisson}) can locally be expressed as $g\left( \tau \right)
=g_{0}\exp \boldsymbol{\xi }\left( \tau \right) $, where $\boldsymbol{\xi }%
\left( \tau \right) $ satisfies the ODE (\ref{locRec}) with $\boldsymbol{\xi 
}\left( 0\right) \neq \mathbf{0}$ small. ii) At $\tau =0$, the solution
admits the series expansion 
\begin{equation}
\boldsymbol{\xi }\left( \tau \right) =\sum_{k\geq 0}\tfrac{\tau ^{k}}{k!}%
\boldsymbol{\Phi }_{k}(\boldsymbol{\xi }\left( 0\right) ,\mathbf{v}\left(
0\right) ,\mathbf{v}^{\prime }\left( 0\right) ,\ldots ,\mathbf{v}^{\left(
k-1\right) }\left( 0\right) )  \label{Xt}
\end{equation}%
with $\boldsymbol{\Phi }_{0}=\boldsymbol{\xi }_{0}$, and $\boldsymbol{\Phi }%
_{k}(\boldsymbol{\xi }\left( \tau \right) ,\mathbf{v}\left( \tau \right) ,%
\mathbf{v}^{\prime }\left( \tau \right) ,\ldots ,\mathbf{v}^{\left(
s-1\right) }\left( \tau \right) )$ in (\ref{Phik}). iii) If the initial
value is zero, $\boldsymbol{\xi }\left( 0\right) =\mathbf{0}$, the
coefficients are given in (\ref{Phik3}), i.e. $\boldsymbol{\Phi }_{k}=%
\boldsymbol{\Phi }_{k}^{0},k\geq 1$ and $\boldsymbol{\Phi }_{k}=\mathbf{0}$.
\end{theorem}

\begin{proof}
i) With $g\left( \tau \right) =g_{0}\exp \boldsymbol{\xi }\left( \tau
\right) $, the left Poisson equation (\ref{Poisson}) becomes $\mathbf{v}%
\left( \tau \right) =\exp (-\boldsymbol{\xi }\left( \tau \right) )\frac{%
\mathrm{d}}{\mathrm{d}\tau }\exp (-\boldsymbol{\xi }\left( \tau \right) )=%
\mathbf{dexp}_{-\boldsymbol{\xi }\left( \tau \right) }\boldsymbol{\xi }%
^{\prime }\left( \tau \right) $. This relation is independent of whether $%
\boldsymbol{\xi }\left( 0\right) =0$. The dexp map exists for $\boldsymbol{%
\xi }\in \mathfrak{g}$ small. ii) Inserting the series expansion of $%
\boldsymbol{\xi }^{\prime }\left( \tau \right) $ and $\mathbf{dexp}_{-%
\boldsymbol{\xi }\left( \tau \right) }^{-1}\mathbf{v}\left( \tau \right) $
at $\tau =0$ in (\ref{locRec}) yields $\sum\limits_{j\geq 0}\tfrac{\tau j}{j!%
}\boldsymbol{\Phi }_{j+1}=\mathbf{dexp}_{-\boldsymbol{\xi }\left( 0\right)
}^{-1}\mathbf{v}\left( 0\right) +\sum\limits_{j\geq 1}\left. \tfrac{\tau ^{j}%
}{j!}\frac{\mathrm{d}^{j}}{\mathrm{d}\tau ^{j}}\left( \mathbf{dexp}_{-%
\boldsymbol{\xi }\left( \tau \right) }^{-1}\mathbf{v}\left( \tau \right)
\right) \right\vert _{\tau =0}$. Application of Leibnitz rule to the last
term, and comparing coefficients for powers of $\tau $ yields (\ref{Phik})
and (\ref{Phik2}). The $\mathbf{dexp}$ map being analytic in a neighborhood
of the the identity ensures convergence. Convergence proofs can be found in 
\cite{Magnus1954,Varadarajan1984}. iii) This was proven in \cite{ZAMM2010}.
\end{proof}

\begin{remark}
The condition that $\boldsymbol{\xi }\left( 0\right) \neq \mathbf{0}$ is
small ensures that $\exp \boldsymbol{\xi }\left( \tau \right) $ remains in a
neighborhood of the identity in $G$. In practical applications this
requirement is lifted, and parameter singularities are tackled using the
limit values of the respective exp map.
\end{remark}

\begin{remark}
It is instructive to consider case iii) as the limit of $\boldsymbol{\Phi }%
_{k}$ in (\ref{Phik}) for $\boldsymbol{\xi }\rightarrow \mathbf{0}$. To this
end, the series expansion (\ref{dexpInv}) is used. The series expansions of
the first and second derivatives are, for instance,%
\begin{align}
(\mathrm{D}_{\boldsymbol{\xi }}\mathrm{dexp}^{-1})(\boldsymbol{\eta })&
=\sum_{i\geq 1}\frac{B_{i}}{i!}\sum_{\substack{ j,l\geq 0  \\ j+l+1=i}}%
\mathrm{ad}_{\boldsymbol{\xi }}^{j}\mathrm{ad}_{\boldsymbol{\eta }}\mathrm{ad%
}_{\boldsymbol{\xi }}^{l}  \label{derSeries1} \\
(\mathrm{D}_{\boldsymbol{\xi }}^{2}\mathrm{dexp}^{-1})(\boldsymbol{\eta })(%
\boldsymbol{\omega })& =\sum_{i\geq 2}\frac{B_{i}}{i!}\sum_{\substack{ %
j,l\geq 0  \\ j+l+2=i}}^{i-1}\left( \sum_{\substack{ r,s\geq 0  \\ r+s=j}}%
^{i-1}\mathrm{ad}_{\boldsymbol{\xi }}^{r}\mathrm{ad}_{\boldsymbol{\omega }}%
\mathrm{ad}_{\boldsymbol{\xi }}^{s}\mathrm{ad}_{\boldsymbol{\eta }}\mathrm{ad%
}_{\boldsymbol{\xi }}^{l}+\sum_{\substack{ r,s\geq 0  \\ r+s=l}}^{i-1}%
\mathrm{ad}_{\boldsymbol{\xi }}^{j}\mathrm{ad}_{\boldsymbol{\eta }}\mathrm{ad%
}_{\boldsymbol{\xi }}^{r}\mathrm{ad}_{\boldsymbol{\omega }}\mathrm{ad}_{%
\boldsymbol{\xi }}^{s}\right)  \label{derSeries2}
\end{align}%
with $\boldsymbol{\xi },\boldsymbol{\eta },\boldsymbol{\omega }\in \mathfrak{%
g}$. The third derivative vanishes identically since $B_{3}=0$. Therewith,
the limits at $\boldsymbol{\xi }\left( 0\right) =0$ are found as%
\begin{equation}
(\mathrm{D}_{0}\mathrm{dexp}^{-1})(\boldsymbol{\eta })=-\frac{1}{2}\mathrm{ad%
}_{\boldsymbol{\eta }},\ (\mathrm{D}_{0}^{2}\mathrm{dexp}^{-1})(\boldsymbol{%
\eta })(\boldsymbol{\omega })=\frac{1}{12}(\mathrm{ad}_{\boldsymbol{\omega }}%
\mathrm{ad}_{\boldsymbol{\eta }}+\mathrm{ad}_{\boldsymbol{\eta }}\mathrm{ad}%
_{\boldsymbol{\omega }}).  \label{coeff}
\end{equation}%
The coefficients (\ref{PhiVal}) are then obtained from (\ref{coeff}).
\end{remark}

Theorem \ref{theoApprox} gives rise to $k$th-order 2-point interpolation
formulae with $\boldsymbol{\xi }{}\left( 0\right) \neq \mathbf{0}$. The
derivation is analogous to that in section 3
. Denote $\Delta \boldsymbol{\xi }:=\bar{\boldsymbol{\xi }}-\boldsymbol{\xi }%
_{0}$, the initial value interpolation (\ref{2PointInt}) becomes%
\begin{equation}
\boldsymbol{\xi }^{\left[ k\right] }\left( \tau \right) =\;\boldsymbol{\xi }%
_{0}+\tau ^{k}\Delta \boldsymbol{\xi }-\sum_{j=1}^{k-1}\frac{1}{j!}(\tau
^{j}-\tau ^{k})\mathbf{v}_{0}^{\left( j-1\right) }+\sum_{r=3}^{k}(\tau
^{r}-\tau ^{k})\mathbf{a}_{r}
\end{equation}%
and in particular, the 3rd- and 4th-order interpolation formulae are%
\begin{eqnarray}
\boldsymbol{\xi }^{\left[ 3\right] }\left( \tau \right) &=&\boldsymbol{\xi }%
_{0}+\tau ^{3}\Delta \boldsymbol{\xi }+\left( \tau -\tau ^{3}\right) \mathbf{%
v}_{0}+\frac{1}{2}\left( \tau ^{2}-\tau ^{3}\right) \mathbf{v}_{0}^{\prime }
\\
\boldsymbol{\xi }^{\left[ 4\right] }\left( \tau \right) &=&\boldsymbol{\xi }%
_{0}+\tau ^{4}\Delta \boldsymbol{\xi }+\left( \tau -\tau ^{4}\right) \mathbf{%
v}_{0}+\frac{1}{2}\left( \tau ^{2}-\tau ^{4}\right) \mathbf{v}_{0}^{\prime }+%
\frac{1}{6}\left( \tau ^{3}-\tau ^{4}\right) (\mathbf{v}_{0}^{\prime \prime
}+\frac{1}{2}\left[ \mathbf{v}_{0}^{\prime },\mathbf{v}_{0}\right] ).  \notag
\end{eqnarray}%
The 3rd- and 4th-order boundary value interpolation, (\ref{3PPB}) and (\ref%
{4PPB}) respectively, become%
\begin{eqnarray}
%
\boldsymbol{\xi }^{\left[ 3\right] }\left( \tau \right) &=&%
%
\boldsymbol{\xi }_{0}+\left( 3\tau ^{2}-2\tau ^{3}\right) \Delta \boldsymbol{%
\xi }+\tau \left( \tau -1\right) ^{2}\mathbf{v}_{0}+\left( \tau ^{3}-\tau
^{2}\right) \mathrm{dexp}_{-\bar{\boldsymbol{\xi }}}^{-1}\mathbf{v}_{1} \\
%
\boldsymbol{\xi }^{\left[ 4\right] }\left( \tau \right) &=&%
%
\boldsymbol{\xi }_{0}+\left( 4\tau ^{3}-3\tau ^{4}\right) \Delta \boldsymbol{%
\xi }+\tau \left( \tau -\tau ^{2}\right) \left( 1+2\tau \right) \mathbf{v}%
_{0}+\left( \tau ^{4}-\tau ^{3}\right) \mathrm{dexp}_{-\bar{\boldsymbol{\xi }%
}}^{-1}\mathbf{v}_{1}+\frac{1}{2}\tau ^{2}\left( 1-\tau \right) ^{2}\mathbf{v%
}_{0}^{\prime }.
\end{eqnarray}

\section{Globally Parameterized $k$th-Order Continuos Splines}

Instead of concatenating local interpolations by group multiplication, which
leads to the POE-spline, the interpolation is defined by a single
exponential of a curve in $\mathfrak{g}$, represented by $\boldsymbol{\xi }%
{}\left( t\right) $ defined on the entire range $t\in \left[ 0,T\right] $ of
the global path parameter.

\begin{definition}
A $k$th-order $C^{p}$-continuos global spline interpolating through $%
h_{0},h_{1},\ldots ,h_{n}\in G$ is a map%
\begin{equation}
\bar{h}\left( t\right) =\bar{h}_{0}\exp \boldsymbol{\xi }\left( t\right) ,\
t\in \left[ 0,T\right]  \label{hGlob2}
\end{equation}%
satisfying the matching condition $h\left( t_{i}\right) =h_{i}$ and the
initial conditions $\bar{\mathbf{v}}^{\left( s\right) }\left( 0\right) =%
\mathbf{v}_{0}^{\left( s\right) },s=0,\ldots ,p-1$, where the coordinate
function $\boldsymbol{\xi }\left( t\right) $ and $\bar{\mathbf{v}}^{\left(
s\right) }\left( t\right) ,s=0,\ldots ,p-1$ are continuos, with
left-invariant vector field defined as $\bar{\mathbf{v}}\left( t\right) :=%
\bar{h}\left( t\right) ^{-1}\bar{h}^{\prime }\left( t\right) $.
\end{definition}

This yields a \emph{global interpolation} in terms of local coordinates $%
\boldsymbol{\xi }{}\left( t\right) $ on $G$. The canonical coordinates will
still be interpolated separately for each spline segment, but the
propagation through the segment is described as $\boldsymbol{\xi }{}\left(
t\right) =\boldsymbol{\xi }{}\left( t_{i-1}\right) +\boldsymbol{\xi }{}_{i}^{%
\left[ k\right] }\left( \tau _{i}\right) ,t\in \left[ t_{i-1},t_{i}\right] $%
, which does not suffer from the issue discussed in section 5 
as the concatenation is done on the Lie algebra (a vector space). Notice
that global refers to fact that local coordinates $\boldsymbol{\xi }$ are
used to parameterize the complete curve in $G$. Such global interpolations
were addressed e.g. in \cite{HanBauchau2018} for rigid body motions. Since $%
\boldsymbol{\xi }{}\left( 0\right) $ is no longer required to be zero, in
general $\bar{h}_{0}\neq h_{0}$. Nevertheless, $\boldsymbol{\xi }{}\left(
0\right) =\mathbf{0}$ is an obvious choice.

\subsection{A $k$th-Order $C^{k-1}$ Continuos Spline with given Initial
Values}

The key to this spline formulation is the 2-point interpolation in section 6 
. The POE-formulation is readily amended leading to Algorithms 5 and 6. The
canonical coordinates at the end of segment $i$ are now obtained from $h_{0}$
and $h_{i}$ as $\bar{\boldsymbol{\xi }}_{i}:=\log (h_{0}^{-1}h_{i})$, rather
than from the relative local increment $h_{0}^{-1}h_{i}$ as in case of POE
splines.

\begin{remark}
The uniqueness of the $\log $ map on the Lie group $G$ must be taken into
account in numerical computations when the 'distance' of the interpolation
points $h_{i}$ becomes large. In case of $SO(3)$, and thus $SE(3)$, the $%
\log $ map only has a (removable) singularity at the origin, but it is only
unique for $\left\Vert \mathbf{x}\right\Vert \leq \pi $. This non-uniqueness
issue when extracting rotation axis and angle from a given rotation matrix
(or unit quaternions) is well-known \cite%
{BottemaRoth1979,AltmannBook1986,Wittenburg2016}, and can be treated if
continuity of the curve $h\left( t\right) $ is ensured.
\end{remark}

\begin{tabular}[t]{l}
\  \\ 
\textbf{Algorithm 5: 3rd-Order }$C^{2}$\textbf{\ Continuos Global Spline}%
\vspace{1.5ex}
\\ 
\begin{tabular}[t]{ll}
\begin{tabular}[t]{l}
\textbf{1. Input:}%
\vspace{0.5ex}
\\ 
\hspace{2ex}%
\begin{tabular}{ll}
\textbullet
& Knot points $t_{0},t_{1},\ldots ,t_{N}$ \\ 
\textbullet
& Samples $h_{0},h_{1},\ldots ,h_{n}$ \\ 
\textbullet
& Initial values $\mathbf{v}_{0},\mathbf{v}_{0}^{\prime }$ \\ 
\textbullet
& Reference $\bar{h}_{0}$%
\end{tabular}%
\end{tabular}
& 
\begin{tabular}[t]{l}
\textbf{2. Initialization:}%
\vspace{0.5ex}
\\ 
\hspace{2ex}%
\begin{tabular}{ll}
\textbullet
& $\boldsymbol{\alpha }_{0}^{\ast }:=\mathbf{v}_{0},\ \boldsymbol{\beta }%
_{0}^{\ast }:=\mathbf{v}_{0}^{\prime },\boldsymbol{\gamma }_{0}^{\ast }:=%
\mathbf{v}_{0}^{\prime \prime }$ \\ 
\textbullet
& $\boldsymbol{\beta }_{0}^{\ast }:=\mathbf{v}_{0}^{\prime }$ \\ 
\textbullet
& $T_{0}:=1$%
\end{tabular}%
\end{tabular}
\\ 
\begin{tabular}[t]{l}
\textbf{3. Computation of Spline Coefficients:} \\ 
\hspace{2ex}%
\textbf{FOR} $i=1,\ldots ,N$ \textbf{DO}%
\vspace{0.5ex}
\\ 
\hspace{2ex}%
\begin{tabular}[t]{ll}
\textbullet
& $T_{i}:=t_{i}-t_{i-1}$ \\ 
\textbullet
& $\delta _{i}:=T_{i}/T_{i-1}$ \\ 
\textbullet
& $\boldsymbol{\xi }_{i}:=\log (h_{0}^{-1}h_{i})$ \\ 
\textbullet
& $\Delta \boldsymbol{\xi }_{i}:=\boldsymbol{\xi }_{i}-\boldsymbol{\xi }%
_{i-1}$ \\ 
\textbullet
& $\boldsymbol{\alpha }_{i-1}:=\delta _{i}\boldsymbol{\alpha }_{i-1}^{\ast }$
\\ 
\textbullet
& $\boldsymbol{\beta }_{i-1}:=\delta _{i}^{2}\boldsymbol{\beta }_{i-1}^{\ast
}$ \\ 
\textbullet
& $\mathbf{d}_{i}:=\mathrm{dexp}_{-\boldsymbol{\xi }_{i-1}}^{-1}\boldsymbol{%
\alpha }_{i-1}$ \\ 
\textbullet
& $\mathbf{e}_{i}:=\mathrm{dexp}_{-\boldsymbol{\xi }_{i-1}}^{-1}\boldsymbol{%
\beta }_{i-1}-(\mathrm{D}_{-\boldsymbol{\xi }_{i-1}}\mathrm{dexp}%
^{-1})\left( \mathbf{d}_{i}\right) \boldsymbol{\alpha }_{i-1}$ \\ 
\textbullet
& $\mathbf{a}_{i}:=3\boldsymbol{\xi }_{i}-2\boldsymbol{\alpha }_{i-1}-\frac{1%
}{2}\boldsymbol{\beta }_{i-1}$ \\ 
\textbullet
& $\mathbf{b}_{i}:=6\boldsymbol{\xi }_{i}-6\boldsymbol{\alpha }_{i-1}-2%
\boldsymbol{\beta }_{i-1}$ \\ 
\textbullet
& $\boldsymbol{\alpha }_{i}^{\ast }:=\mathrm{dexp}_{-\boldsymbol{\xi }_{i}}%
\mathbf{a}_{i}$ \\ 
\textbullet
& $\boldsymbol{\beta }_{i}^{\ast }:=\mathrm{dexp}_{-\boldsymbol{\xi }_{i}}%
\mathbf{b}_{i}-(\mathrm{D}_{-\boldsymbol{\xi }_{i}}\mathrm{dexp})\left( 
\mathbf{a}_{i}\right) \mathbf{a}_{i}$%
\end{tabular}%
\vspace{0.5ex}
\\ 
\ \ \ \ \textbf{END}%
\end{tabular}%
\hspace{-5ex}
& 
\begin{tabular}[t]{l}
\textbf{4. Evaluation} at $t\in \left[ t_{i-1},t_{i}\right] $, with $\tau
_{i}\left( t\right) =\frac{t-t_{i-1}}{T_{i}}$:%
\vspace{0.5ex}
\\ 
\hspace{2ex}%
\begin{tabular}{rl}
$\boldsymbol{\xi }_{i}^{\left[ 3\right] }{}\left( \tau _{i}\right) $%
\hspace{-2.5ex}
& $=\boldsymbol{\xi }_{i-1}+\tau _{i}^{3}\Delta \boldsymbol{\xi }_{i}+\left(
\tau _{i}-\tau _{i}^{3}\right) \mathbf{d}_{i}+\frac{1}{2}\left( \tau
_{i}^{2}-\tau _{i}^{3}\right) \mathbf{e}_{i}$ \\ 
$h^{\left[ 3\right] }\left( t\right) $%
\hspace{-2.5ex}
& $=\bar{h}_{0}\exp \boldsymbol{\xi }_{i}^{\left[ 3\right] }{}\left( \tau
_{i}\left( t\right) \right) $%
\end{tabular}%
\end{tabular}%
\end{tabular}
\\ 
\ 
\end{tabular}

\begin{tabular}[t]{l}
\  \\ 
\textbf{Algorithm 6: 4th-Order }$C^{3}$\textbf{\ Continuos Global Spline}%
\vspace{1.5ex}
\\ 
\begin{tabular}[t]{l}
\textbf{1. Input:}%
\vspace{0.5ex}
\\ 
\hspace{2ex}%
\begin{tabular}{ll}
\textbullet
& Knot points $t_{0},t_{1},\ldots ,t_{N}$ \\ 
\textbullet
& Samples $h_{0},h_{1},\ldots ,h_{n}$ \\ 
\textbullet
& Initial values $\mathbf{v}_{0},\mathbf{v}_{0}^{\prime }\mathbf{,v}%
_{0}^{\prime \prime }$ \\ 
\textbullet
& Reference $\bar{h}_{0}$%
\end{tabular}%
\end{tabular}%
\hspace{8ex}%
\begin{tabular}[t]{l}
\textbf{2. Initialization:}%
\vspace{0.5ex}
\\ 
\hspace{2ex}%
\begin{tabular}{ll}
\textbullet
& $\boldsymbol{\alpha }_{0}^{\ast }:=\mathbf{v}_{0},\ \boldsymbol{\beta }%
_{0}^{\ast }:=\mathbf{v}_{0}^{\prime },\boldsymbol{\gamma }_{0}^{\ast }:=%
\mathbf{v}_{0}^{\prime \prime }$ \\ 
\textbullet
& $\boldsymbol{\beta }_{0}^{\ast }:=\mathbf{v}_{0}^{\prime }$ \\ 
\textbullet
& $\boldsymbol{\gamma }_{0}^{\ast }:=\mathbf{v}_{0}^{\prime \prime }$ \\ 
\textbullet
& $T_{0}:=1$%
\end{tabular}%
\end{tabular}%
\vspace{1.5ex}
\\ 
\begin{tabular}[t]{l}
\textbf{3. Computation of Spline Coefficients:} \\ 
\hspace{2ex}%
\textbf{FOR} $i=1,\ldots ,N$ \textbf{DO}%
\vspace{0.5ex}
\\ 
\hspace{2ex}%
\begin{tabular}[t]{ll}
\textbullet
& $T_{i}:=t_{i}-t_{i-1}$ \\ 
\textbullet
& $\delta _{i}:=T_{i}/T_{i-1}$ \\ 
\textbullet
& $\boldsymbol{\xi }_{i}:=\log (h_{0}^{-1}h_{i})$ \\ 
\textbullet
& $\Delta \boldsymbol{\xi }_{i}:=\boldsymbol{\xi }_{i}-\boldsymbol{\xi }%
_{i-1}$ \\ 
\textbullet
& $\boldsymbol{\alpha }_{i-1}:=\delta _{i}\boldsymbol{\alpha }_{i-1}^{\ast }$
\\ 
\textbullet
& $\boldsymbol{\beta }_{i-1}:=\delta _{i}^{2}\boldsymbol{\beta }_{i-1}^{\ast
}$ \\ 
\textbullet
& $\boldsymbol{\gamma }_{i-1}:=\delta _{i}^{3}\boldsymbol{\gamma }%
_{i-1}^{\ast }$ \\ 
\textbullet
& $\mathbf{d}_{i}:=\mathrm{dexp}_{-\boldsymbol{\xi }_{i-1}}^{-1}\boldsymbol{%
\alpha }_{i-1}$ \\ 
\textbullet
& $\mathbf{e}_{i}:=\mathrm{dexp}_{-\boldsymbol{\xi }_{i-1}}^{-1}\boldsymbol{%
\beta }_{i-1}-(\mathrm{D}_{-\boldsymbol{\xi }_{i-1}}\mathrm{dexp}%
^{-1})\left( \mathbf{d}_{i}\right) \boldsymbol{\alpha }_{i-1}$ \\ 
\textbullet
& $\mathbf{f}_{i}:=\mathrm{dexp}_{-\boldsymbol{\xi }_{i-1}}^{-1}\boldsymbol{%
\gamma }_{i-1}-2(\mathrm{D}_{-\boldsymbol{\xi }_{i-1}}\mathrm{dexp}%
^{-1})\left( \mathbf{d}_{i}\right) \boldsymbol{\beta }_{i-1}-(\mathrm{D}_{-%
\boldsymbol{\xi }_{i-1}}\mathrm{dexp}^{-1})\left( \mathbf{e}_{i}\right) 
\boldsymbol{\alpha }_{i-1}+(\mathrm{D}_{-\boldsymbol{\xi }_{i-1}}^{2}\mathrm{%
dexp})\left( \mathbf{d}_{i},\mathbf{d}_{i}\right) \boldsymbol{\alpha }_{i-1}$
\\ 
\textbullet
& $\mathbf{a}_{i}:=4\Delta \boldsymbol{\xi }_{i}-3\mathbf{d}_{i}-\mathbf{e}%
_{i}-\frac{1}{6}\mathbf{f}_{i}$ \\ 
\textbullet
& $\mathbf{b}_{i}:=12\Delta \boldsymbol{\xi }_{i}-12\mathbf{d}_{i}-5\mathbf{e%
}_{i}-\mathbf{f}_{i}$ \\ 
\textbullet
& $\mathbf{c}_{i}:=24\Delta \boldsymbol{\xi }_{i}-24\mathbf{d}_{i}-12\mathbf{%
e}_{i}-3\mathbf{f}_{i}$ \\ 
\textbullet
& $\boldsymbol{\alpha }_{i}^{\ast }:=\mathrm{dexp}_{-\boldsymbol{\xi }_{i}}%
\mathbf{a}_{i}$ \\ 
\textbullet
& $\boldsymbol{\beta }_{i}^{\ast }:=\mathrm{dexp}_{-\boldsymbol{\xi }_{i}}%
\mathbf{b}_{i}-(\mathrm{D}_{-\boldsymbol{\xi }_{i}}\mathrm{dexp})\left( 
\mathbf{a}_{i}\right) \mathbf{a}_{i}$ \\ 
\textbullet
& $\boldsymbol{\gamma }_{i}^{\ast }:=\mathrm{dexp}_{-\boldsymbol{\xi }_{i}}%
\mathbf{c}_{i}-2(\mathrm{D}_{-\boldsymbol{\xi }_{i}}\mathrm{dexp})\left( 
\mathbf{a}_{i}\right) \mathbf{b}_{i}-(\mathrm{D}_{-\boldsymbol{\xi }_{i}}%
\mathrm{dexp})\left( \mathbf{b}_{i}\right) \mathbf{a}_{i}+(\mathrm{D}_{-%
\boldsymbol{\xi }_{i}}^{2}\mathrm{dexp})\left( \mathbf{a}_{i},\mathbf{a}%
_{i}\right) \mathbf{a}_{i}$%
\end{tabular}%
\vspace{0.5ex}
\\ 
\ \ \ \ \textbf{END}%
\end{tabular}%
\vspace{1.5ex}
\\ 
\begin{tabular}[t]{l}
\textbf{4. Evaluation} at $t\in \left[ t_{i-1},t_{i}\right] $, with $\tau
_{i}\left( t\right) =\frac{t-t_{i-1}}{T_{i}}$:%
\vspace{0.5ex}
\\ 
\hspace{2ex}%
\begin{tabular}{rl}
$\boldsymbol{\xi }_{i}^{\left[ 4\right] }{}\left( \tau _{i}\right) $%
\hspace{-2.5ex}
& $=\boldsymbol{\xi }_{i-1}+\tau _{i}^{4}\Delta \boldsymbol{\xi }_{i}+\left(
\tau _{i}-\tau _{i}^{4}\right) \mathbf{d}_{i}+\frac{1}{2}\left( \tau
_{i}^{2}-\tau _{i}^{4}\right) \mathbf{e}_{i}+\frac{1}{6}\left( \tau
_{i}^{3}-\tau _{i}^{4}\right) \mathbf{f}_{i}$ \\ 
$h^{\left[ 4\right] }\left( t\right) $%
\hspace{-2.5ex}
& $=\bar{h}_{0}\exp \boldsymbol{\xi }_{i}^{\left[ 4\right] }{}\left( \tau
_{i}\left( t\right) \right) $%
\end{tabular}%
\end{tabular}
\\ 
\ 
\end{tabular}

\subsection{A $k$th-Order $C^{k-2}$ Continuos Spline with given Velocities
at Interpolation Points}

As in section \ref{secPOEVel}, it is now assumed that velocities $\mathbf{v}%
_{0},\mathbf{v}_{1},\ldots ,\mathbf{v}_{n}$ are prescribed at the knot
points. Algorithms 3 and 4 are readily adopted to the case of non-zero
initial value $\boldsymbol{\xi }_{i}\left( 0\right) \neq \mathbf{0}$. This
leads to Algorithms 7 and 8. The splines are constructed by concatenating
the 2-point boundary interpolation formulae (\ref{3PPB}) and (\ref{4PPB}).
In the 4th-order algorithm, the derivative $\mathbf{v}_{i}^{\prime }\left(
t_{1-i}\right) $ at begin of segment $i$ is computed with (\ref{vdot}).

\begin{tabular}[t]{l}
\textbf{Algorithm 7: 3rd-Order }$C^{1}$\textbf{\ Continuos Global Spline
with Given Velocities}%
\vspace{1.5ex}
\\ 
\begin{tabular}[t]{l}
\textbf{1. Input:}%
\vspace{0.5ex}
\\ 
\hspace{2ex}%
\begin{tabular}{ll}
\textbullet
& Knot points $t_{0},t_{1},\ldots ,t_{N}$ \\ 
\textbullet
& Samples $h_{0},h_{1},\ldots ,h_{n}$ \\ 
\textbullet
& Velocities $\mathbf{v}_{0},\mathbf{v}_{1},\ldots ,\mathbf{v}_{n}$ \\ 
\textbullet
& Reference $\bar{h}_{0}$%
\end{tabular}%
\end{tabular}%
\hspace{8ex}%
\begin{tabular}[t]{l}
\textbf{2. Initialization:}%
\vspace{0.5ex}
\\ 
\hspace{2ex}%
\begin{tabular}{ll}
\textbullet
& $\boldsymbol{\alpha }_{0}^{\ast }:=\mathbf{v}_{0}$ \\ 
\textbullet
& $T_{0}:=1$%
\end{tabular}%
\end{tabular}%
\vspace{1.5ex}
\\ 
\begin{tabular}[t]{l}
\textbf{3. Computation of Spline Coefficients:} \\ 
\hspace{2ex}%
\textbf{FOR} $i=1,\ldots ,N$ \textbf{DO}%
\vspace{0.5ex}
\\ 
\hspace{2ex}%
\begin{tabular}[t]{ll}
\textbullet
& $T_{i}:=t_{i}-t_{i-1}$ \\ 
\textbullet
& $\delta _{i}:=T_{i}/T_{i-1}$ \\ 
\textbullet
& $\boldsymbol{\xi }_{i}:=\log (h_{0}^{-1}h_{i})$ \\ 
\textbullet
& $\Delta \boldsymbol{\xi }_{i}=\log (h_{i-1}^{-1}h_{i})$ \\ 
\textbullet
& $\boldsymbol{\alpha }_{i}^{\ast }:=T_{i}%
\color[rgb]{0.7,0,0}%
\mathbf{v}_{i}%
\color{black}%
$ \\ 
\textbullet
& $\boldsymbol{\alpha }_{i-1}:=T_{i}%
\color[rgb]{0.7,0,0}%
\mathbf{v}_{i-1}%
\color{black}%
$ \\ 
\textbullet
& $\mathbf{a}_{i}:=\mathrm{dexp}_{-\boldsymbol{\xi }_{i}}^{-1}\boldsymbol{%
\alpha }_{i}^{\ast }$ \\ 
\textbullet
& $\mathbf{d}_{i}:=\mathrm{dexp}_{-\boldsymbol{\xi }_{i-1}}^{-1}\boldsymbol{%
\alpha }_{i-1}$%
\end{tabular}%
\vspace{0.5ex}
\\ 
\ \ \ \ \textbf{END}%
\end{tabular}%
\vspace{1.5ex}
\\ 
\begin{tabular}[t]{l}
\textbf{4. Evaluation} at $t\in \left[ t_{i-1},t_{i}\right] $, with $\tau
_{i}\left( t\right) =\frac{t-t_{i-1}}{T_{i}}$:%
\vspace{0.5ex}
\\ 
\hspace{2ex}%
$%
\begin{tabular}{rl}
$\boldsymbol{\xi }_{i}^{\left[ 3\right] }{}\left( \tau _{i}\right) $%
\hspace{-2.5ex}
& $=\bar{\boldsymbol{\xi }}_{i}+\tau _{i}^{3}\Delta \boldsymbol{\xi }%
_{i}+\left( \tau _{i}-\tau _{i}^{3}\right) \boldsymbol{\alpha }_{i-1}+\frac{1%
}{2}\left( \tau _{i}^{2}-\tau _{i}^{3}\right) \boldsymbol{\beta }_{i-1}$ \\ 
$h^{\left[ 3\right] }\left( t\right) $%
\hspace{-2.5ex}
& $=\bar{h}_{0}\exp \boldsymbol{\xi }_{i}^{\left[ 3\right] }{}\left( \tau
_{i}\left( t\right) \right) $%
\end{tabular}%
$%
\end{tabular}%
\end{tabular}

\begin{tabular}[t]{l}
\textbf{Algorithm 8: 4th-Order }$C^{2}$\textbf{\ Continuos Global Spline
with given Velocities}%
\vspace{1.5ex}
\\ 
\begin{tabular}[t]{l}
\textbf{1. Input:}%
\vspace{0.5ex}
\\ 
\hspace{2ex}%
\begin{tabular}{ll}
\textbullet
& Knot points $t_{0},t_{1},\ldots ,t_{N}$ \\ 
\textbullet
& Samples $h_{0},h_{1},\ldots ,h_{n}$ \\ 
\textbullet
& Velocities $\mathbf{v}_{0},\mathbf{v}_{1},\ldots ,\mathbf{v}_{n}$ \\ 
\textbullet
& Initial derivative $\mathbf{v}_{0}^{\prime }$ \\ 
\textbullet
& Reference $\bar{h}_{0}$%
\end{tabular}%
\end{tabular}%
\hspace{8ex}%
\begin{tabular}[t]{l}
\textbf{2. Initialization:}%
\vspace{0.5ex}
\\ 
\hspace{2ex}%
\begin{tabular}{ll}
\textbullet
& $\boldsymbol{\alpha }_{0}^{\ast }:=\mathbf{v}_{0}$ \\ 
\textbullet
& $\boldsymbol{\beta }_{0}^{\ast }:=\mathbf{v}_{0}^{\prime }$ \\ 
\textbullet
& $T_{0}:=1$%
\end{tabular}%
\end{tabular}%
\vspace{1.5ex}
\\ 
\begin{tabular}[t]{l}
\textbf{3. Computation of Spline Coefficients:} \\ 
\hspace{2ex}%
\textbf{FOR} $i=1,\ldots ,N$ \textbf{DO}%
\vspace{0.5ex}
\\ 
\hspace{2ex}%
\begin{tabular}[t]{ll}
\textbullet
& $T_{i}:=t_{i}-t_{i-1}$ \\ 
\textbullet
& $\delta _{i}:=T_{i}/T_{i-1}$ \\ 
\textbullet
& $\bar{\boldsymbol{\xi }}_{i}:=\log (h_{0}^{-1}h_{i})$ \\ 
\textbullet
& $\Delta \boldsymbol{\xi }_{i}=\log (h_{i-1}^{-1}h_{i})$ \\ 
\textbullet
& $\boldsymbol{\alpha }_{i}^{\ast }:=T_{i}%
\color[rgb]{0.7,0,0}%
\mathbf{v}_{i}%
\color{black}%
$ \\ 
\textbullet
& $\boldsymbol{\alpha }_{i-1}:=T_{i}%
\color[rgb]{0.7,0,0}%
\mathbf{v}_{i-1}%
\color{black}%
$ \\ 
\textbullet
& $\mathbf{a}_{i}:=\mathrm{dexp}_{-\boldsymbol{\xi }_{i}}^{-1}\boldsymbol{%
\alpha }_{i}^{\ast }$ \\ 
\textbullet
& $\mathbf{b}_{i}:=-12\boldsymbol{\xi }_{i}+6\boldsymbol{\alpha }_{i-1}+%
\boldsymbol{\beta }_{i-1}+6\mathbf{a}_{i}$ \\ 
\textbullet
& $\boldsymbol{\beta }_{i}^{\ast }:=\mathrm{dexp}_{-\boldsymbol{\xi }_{i}}%
\mathbf{b}_{i}-(\mathrm{D}_{-\boldsymbol{\xi }_{i}}\mathrm{dexp})\left( 
\mathbf{a}_{i}\right) \mathbf{a}_{i}$%
\end{tabular}%
\vspace{0.5ex}
\\ 
\ \ \ \ \textbf{END}%
\end{tabular}%
\vspace{1.5ex}
\\ 
\begin{tabular}[t]{l}
\textbf{4. Evaluation} at $t\in \left[ t_{i-1},t_{i}\right] $, with $\tau
_{i}\left( t\right) =\frac{t-t_{i-1}}{T_{i}}$:%
\vspace{0.5ex}
\\ 
\hspace{2ex}%
$%
\begin{tabular}{rl}
$\boldsymbol{\xi }_{i}^{\left[ 4\right] }{}\left( \tau _{i}\right) $%
\hspace{-2.5ex}
& $=\boldsymbol{\xi }_{i-1}+\left( 4\tau _{i}^{3}-3\tau _{i}^{4}\right)
\Delta \boldsymbol{\xi }_{i}+\left( 1-\tau _{i}\right) ^{2}\left( \tau
_{i}+2\tau _{i}^{2}\right) \mathbf{d}_{i}+\frac{1}{2}\tau _{i}^{2}\left(
\tau _{i}-1\right) ^{2}\mathbf{e}_{i}+\left( \tau _{i}^{4}-\tau
_{i}^{3}\right) \mathbf{a}_{i}$ \\ 
$h^{\left[ 4\right] }\left( t\right) $%
\hspace{-2.5ex}
& $=\bar{h}_{0}\exp \boldsymbol{\xi }_{i}^{\left[ 4\right] }{}\left( \tau
_{i}\left( t\right) \right) $%
\end{tabular}%
$%
\end{tabular}%
\end{tabular}

\begin{example}
The interpolation of a cubic spatial rotation in $SO\left( 3\right) $,
described by $\boldsymbol{\xi }\left( t\right) =t\boldsymbol{\xi }_{1}+t^{3}%
\boldsymbol{\xi }_{2}$, was investigated in example \ref{exaLimit},
revealing the limitation of POE-splines. The same input data, $h_{0},\ldots
,h_{10}$ and $\mathbf{v}_{0},\ldots ,\mathbf{v}_{10}$, are now used for a
global 3rd-order spline interpolation. The results obtained with the
3rd-order global spline for given initial values (Algorithm 5) and when
velocities samples are provided (Algorithm 7) are shown in Fig. \ref%
{figCubicSO3_3rd_GlobalVel}. The global interpolation can clearly
reconstruct the cubic curve, in contrast to the POE-spline (see Fig. \ref%
{figOneParam_3rd}). The same is true when applying the 4th-order
interpolation.
\end{example}

\begin{figure}[h]
\begin{center}
\includegraphics[draft=false,height=5cm]{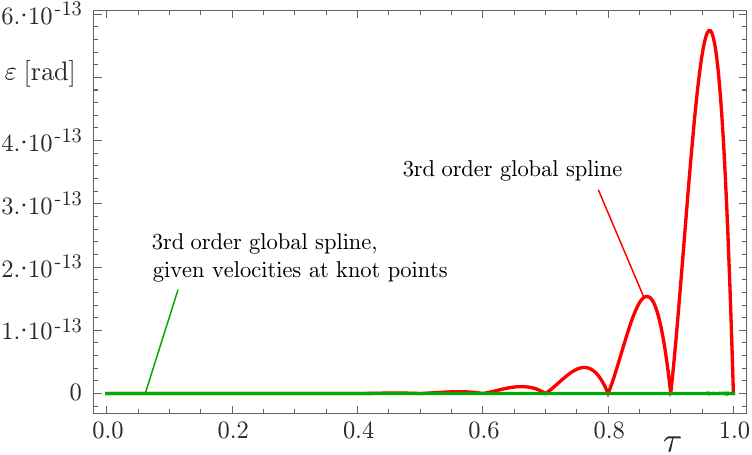}
\end{center}
\caption{Error when interpolating a general cubic motion in $SO(3)$ with a
global 3rd-order spline.}
\label{figCubicSO3_3rd_GlobalVel}
\end{figure}

\begin{figure}[b]
\begin{center}
\includegraphics[draft=false,height=5.5cm]{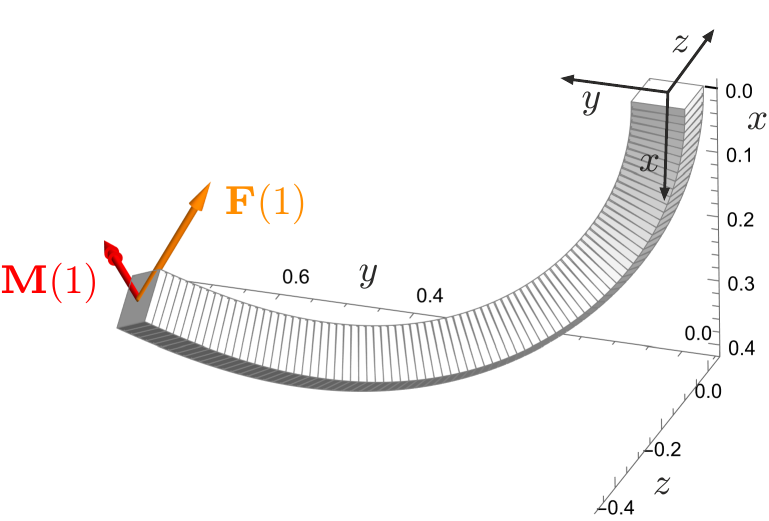}
\end{center}
\caption{Static deformation of the slender rod when subjected to the wrench $%
\mathbf{W}\left( 1\right) =\left[ \mathbf{M}\left( 1\right) ,\mathbf{F}%
\left( 1\right) \right] $ at the terminal end.}
\label{figBeamForce_3DView}
\end{figure}
\begin{figure}[h]
\hfill a)%
\includegraphics[draft=false,height=4.5cm]{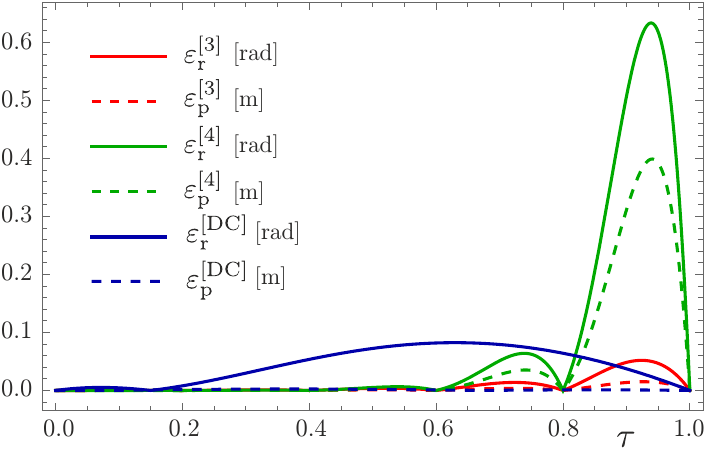}%
\hfill b)%
\includegraphics[draft=false,height=4.5cm]{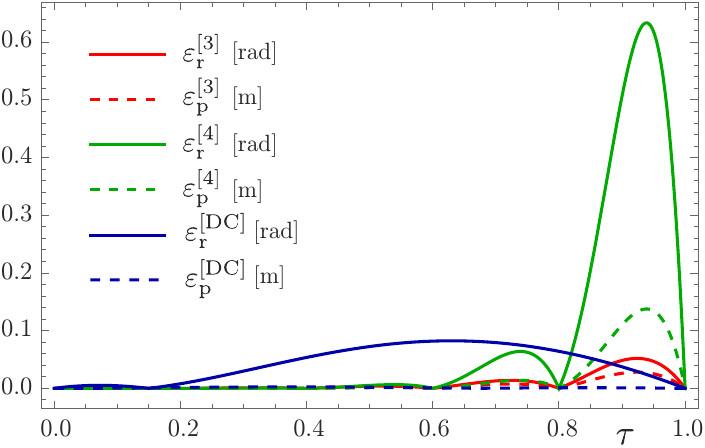}
\hfill \,
\caption{\color[rgb]{0.5,0,0}Interpolation error $\boldsymbol{\protect%
\varepsilon }_{\mathrm{r}}$ and $\boldsymbol{\protect\varepsilon }_{\mathrm{p%
}}$ for orientation and position when using a) POE-spline, and b) global
spline. Additionally shown in blue is the error for a B\'{e}zier curve
interpolating through the 5 points using De Casteljau algorithm 9. }
\label{figBeamForce_Rot+Pos_Loc_3+4}
\end{figure}

\begin{example}
\label{examRodSpline}The shape reconstruction of a slender rod undergoing
spatial deflections is addressed. In example \ref{exaBeam2Point}, the shape
was approximated by a 2-point interpolation method. Now a spline on $%
SE\left( 3\right) $ is used to represent the rod, and parameterized by the
normalized global path parameter $t\in \left[ 0,1\right] $. 
%
To this end, the cross section displacement is sampled at $n=5$ equidistant
points $t_{1},\ldots ,t_{n}$ along the rod yielding $h_{0},\ldots ,h_{n}\in
SE\left( 3\right) $. 
%
The poses are obtained from the reference solution $\boldsymbol{\xi }\left(
t\right) $ as $h_{i}=\exp \boldsymbol{\xi }\left( t_{i}\right) $ (see
example \ref{exaBeam2Point}). For this example, the deformation to be
interpolated is the result of a 
%
concentrated load $\bar{\mathbf{W}}=\left[ \mathbf{M},\mathbf{F}\right] $ 
%
in (\ref{ChiBar}), applied at the terminal end of the rod ($\tau =1$). 
%
Deformations due to a pure moment about a principal cross section axis are 
%
exactly recovered by any of the splines since the deformation is described
by a constant screw axis (constant curvature). In the following, the applied
wrench is computed with the force $\mathbf{F}\left( 1\right) =\left[
0,0.12,-0.08\right] $ and moment $\mathbf{M}\left( 1\right) =\left[
0,-0.008,-0.012\right] $. The resulting shape is shown in Fig. \ref%
{figBeamForce_3DView}.

The initial an terminal deformation measure, $\mathbf{v}_{0}$ and $\mathbf{v}%
_{1}$, and the derivative $\mathbf{v}_{0}^{\prime }$, needed for the spline
interpolation, are computed from the reference solution. Fig. \ref%
{figBeamForce_Rot+Pos_Loc_3+4}a) shows the interpolation errors when using
the 3rd- and 4th-order POE spline algorithms 1 and 2, respectively. The
global spline, constructed with algorithms 5 and 6, yields the results in
Fig. \ref{figBeamForce_Rot+Pos_Loc_3+4}b). Clearly, for this general
deformation the POE- and the global spline yield almost the same results.
The visible larger error of the 4th-order spline can be explained by its
sensitivity to the derivative $\mathbf{v}_{0}^{\prime }$, which is obtained
numerically. 
%
For comparison, the result of a B\'{e}zier curve interpolation constructed
with the De Casteljau algorithm 9 (appendix B) is shown in blue. The 5
samples yield a 4th-order B\'{e}zier curve. Apparently, the curve does not
pass through the intermediate configurations, which serve as control points,
but interpolates between start and end configuration (see also appendix B).
To make the curve pass through the points needs computation of control
points in the Lie group, which is an open problem and not the topic of this
paper. For constructing B-splines on $SE\left( 3\right) $, a De Boor
algorithm was presented in \cite{SprottRavani2002,Ding2002}. 
%
Next the POE spline algorithms 3 and 4 are applied. To this end, the
deformation $\mathbf{v}_{i},i=0,\ldots ,n$ is computed at the $n=5$
interpolation points. This is found from the equilibrium conditions as $%
\mathbf{v}_{i}=\mathbf{K}^{-1}\mathbf{Ad}_{h_{i}}^{T}\mathbf{K}\left( 
\mathbf{v}_{0}-\mathbf{v}^{0}\right) +\mathbf{v}^{0}$, where $\mathbf{Ad}%
_{h_{i}}$ is the adjoined operator on $SE\left( 3\right) $ \cite%
{MurrayBook,SeligBook}. The improvement in accuracy when providing the
deformations $\mathbf{v}_{i}$ at the interpolation points is apparent from
Fig. \ref{figBeamForce_Rot+Pos_LocVel_3+4}a). This alleviates the
sensitivity of the 4th-order spline, which better approximates the shape.
The results obtained with the global spline algorithms 5 and 6, with given
deformations $\mathbf{v}_{i},i=0,\ldots ,n$, are shown in Fig. \ref%
{figBeamForce_Rot+Pos_LocVel_3+4}b), which are again similar to those of the
POE-spline. Finally, Fig. \ref{figBeamForce_Rot+Pos_LocVel_3+4_10Samples}
shows the interpolation error when $n=10$ interpolation points are used,
i.e. the number of segments is doubled.\newline
While for this example the $h_{i}$ and $\mathbf{v}_{i}$ are computed from a
reference solution, it should be remarked that it depends on the application
how the $h_{i}$ are obtained.
\end{example}

\begin{figure}[h]
\hfill a)%
\includegraphics[draft=false,height=4.5cm]{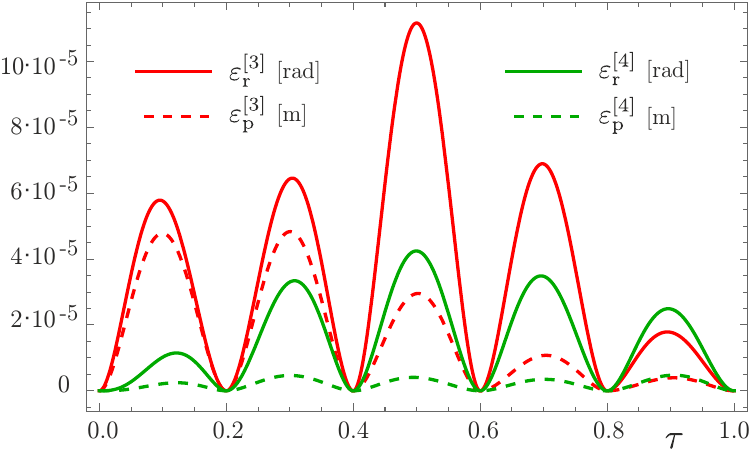}%
\hfill b) %
\includegraphics[draft=false,height=4.5cm]{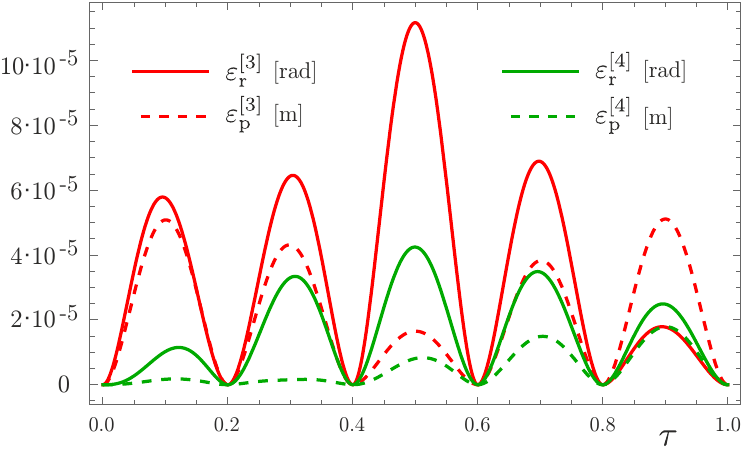}
\hfill \,
\caption{Interpolation error $\boldsymbol{\protect\varepsilon }_{\mathrm{r}}$
and $\boldsymbol{\protect\varepsilon }_{\mathrm{p}}$ for orientation and
position obtained with a) POE-spline, and b) global spline, when
deformations are prescribed at the interpolation points.}
\label{figBeamForce_Rot+Pos_LocVel_3+4}
\end{figure}
\begin{figure}[h]
\hfill a)%
\includegraphics[draft=false,height=4.5cm]{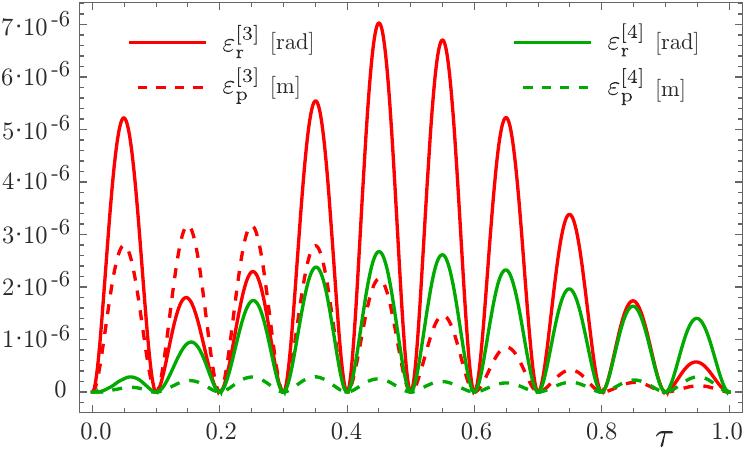}%
\hfill b) %
\includegraphics[draft=false,height=4.5cm]{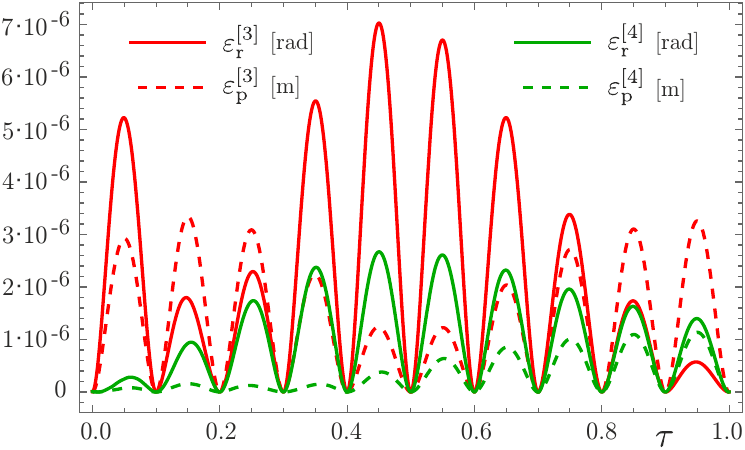}
\hfill
\caption{Interpolation error $\boldsymbol{\protect\varepsilon }_{\mathrm{r}}$
and $\boldsymbol{\protect\varepsilon }_{\mathrm{p}}$ for orientation and
position obtained with a) POE-spline, and b) global spline with 20 segments,
when deformations are prescribed at the interpolation points.}
\label{figBeamForce_Rot+Pos_LocVel_3+4_10Samples}
\end{figure}

\newpage

\begin{example}
%
The elastic rod from example \ref{examRodSpline} is considered, but with a
rectangular cross section $4\times 12$\thinspace mm$^{2}$. The rod is
subjected to a force $\mathbf{F}\left( 0\right) =\left[ 0,-0.03,0.01\right] $
and moment $\mathbf{M}\left( 0\right) =\left[ 0,0.002,0.002\right] $ applied
at the initial end at $\tau =0$ that is in static equilibrium with the force 
$\mathbf{F}\left( 1\right) =\left[ -0.0021,0.0192,-0.0181\right] $ and
moment $\mathbf{M}\left( 1\right) =\left[ 0.0296,0.0082,0.0075\right] $ at
the terminal with $\tau =1$. The numerically exact displacement obtained as
solution of (\ref{Poisson}) and (\ref{ChiBar}) is shown in Fig. \ref%
{figPOE_FlatRod_Example1} in gray color. The shape reconstructed with the
3rd-order interpolation is shown in Fig. \ref{figPOE_FlatRod_Example1}a),
and in Fig. \ref{figPOE_FlatRod_Example1}b) when additionally the
deformation is prescribed at the interpolation points. The interpolation
error is shown in Fig. \ref{figPOE_FlatRod_Example1_err}. Results are shown
in Fig. \ref{figPOE_FlatRod_Example2} when the rod exhibits larger
variations in the deformation field. In all examples it is apparent that
providing the deformations at the interpolation points reduces the error
amplification with increasing via points, which is well-known for spline
interpolations in general. For completeness, the results are shown in Fig. %
\ref{figPOE_FlatRod_Example1_err} for the 4th-order B\'{e}zier curve
constructed with the 5 equidistant samples along the rod used as control
points using the De Casteljau algorithm 9.%

\begin{figure}[h]
\begin{center}
a) %
\includegraphics[draft=false,height=4.cm]{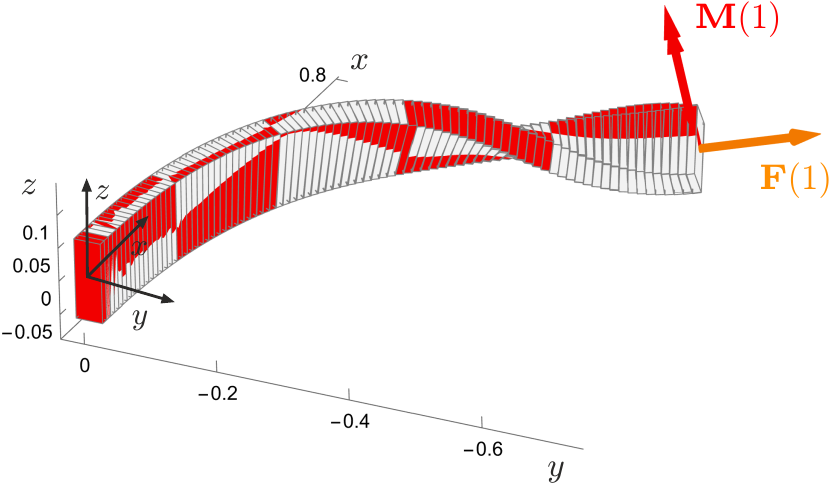}%
\hfil b)%
\includegraphics[draft=false,height=4.cm]{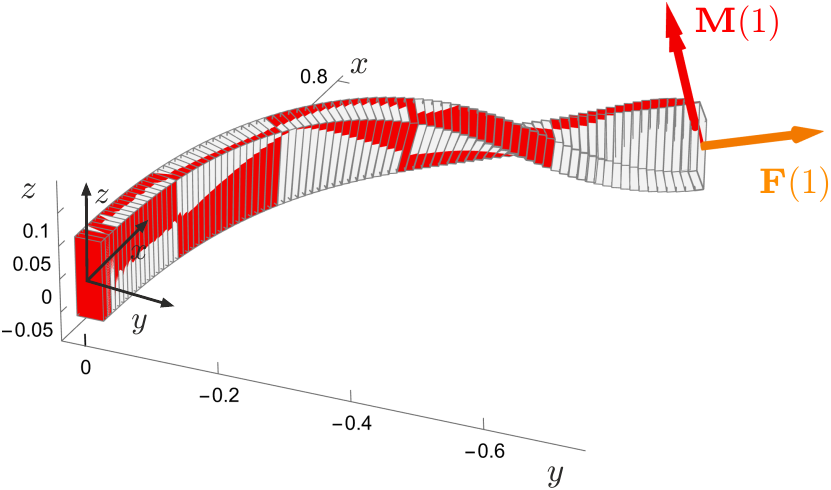}
\end{center}
\caption{%
%
Shape of flexible rod resulting from a) the 3rd-order interpolation and b)
the 3rd-order interpolation with deformation prescribed at the interpolation
pints (red). The exact numerical solution is shown in gray color (partially
covered). }
\label{figPOE_FlatRod_Example1}
\end{figure}
\begin{figure}[h]
\hfill a)%
\includegraphics[draft=false,height=4.5cm]{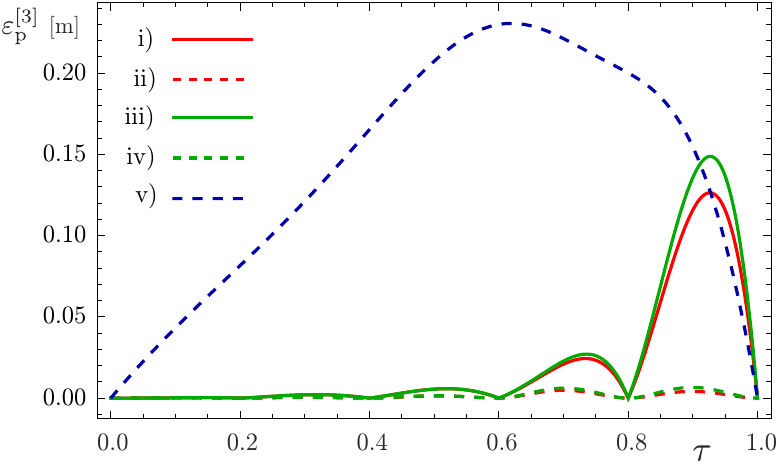}
\hfill b) %
\includegraphics[draft=false,height=4.5cm]{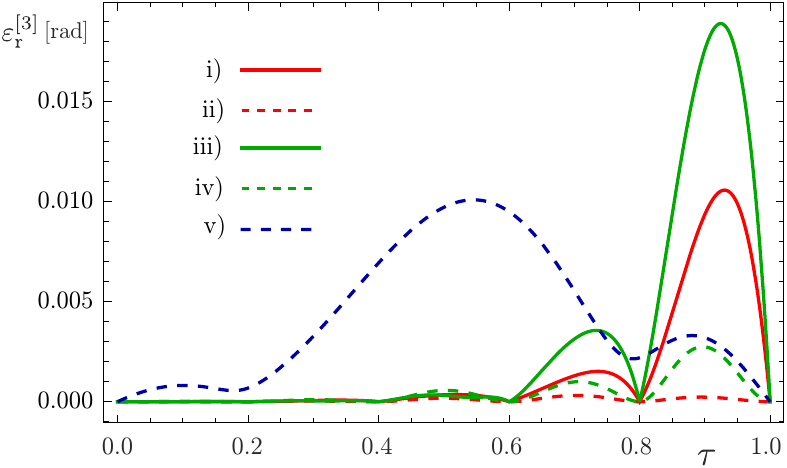}
\hfill \,
\caption{a) Position error $\boldsymbol{\protect\varepsilon }_{\mathrm{p}}$
and b) orientation error $\boldsymbol{\protect\varepsilon }_{\mathrm{p}}$
obtained with 3r-order piecewise POE-spline (i), global POE spline (ii),
piecewise POE-spline (iii) and global POE-spline (iv) with prescribed
deformation at interpolation points, and the B\'{e}zier curve of order 4
constructed with the De Casteljau algotithm 9 (v). }
\label{figPOE_FlatRod_Example1_err}
\end{figure}

\begin{figure}[h]
\begin{center}
a)%
\includegraphics[draft=false,height=5.5cm]{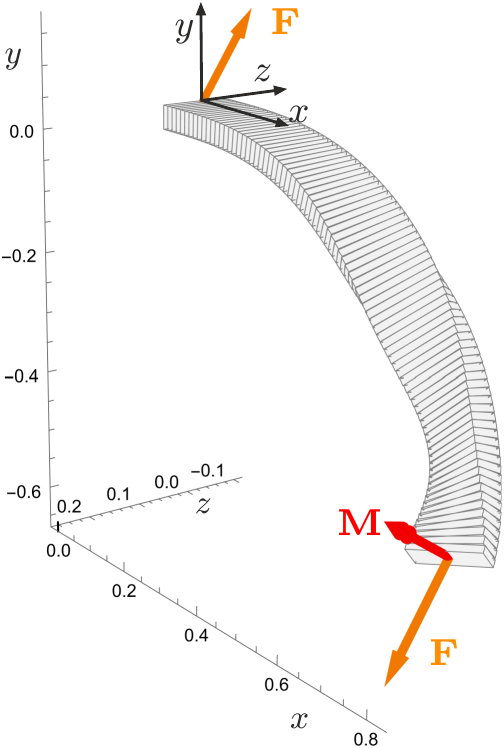}
\hfil b)%
\includegraphics[draft=false,height=5.5cm]{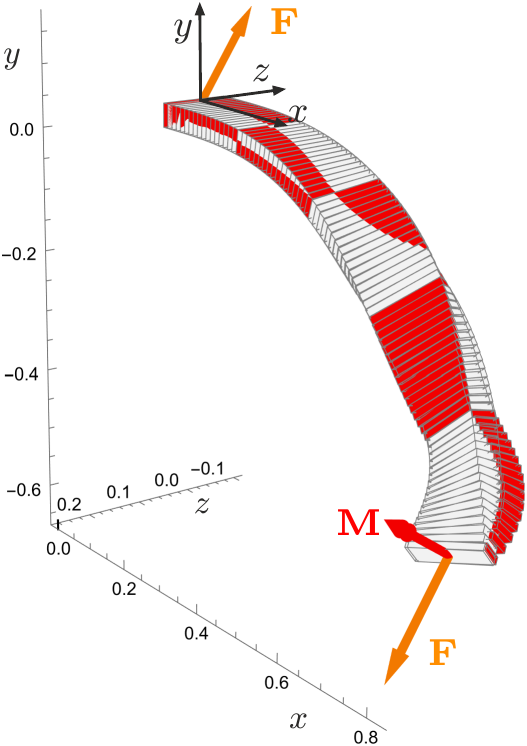}%
\hfil c)%
\includegraphics[draft=false,height=5.5cm]{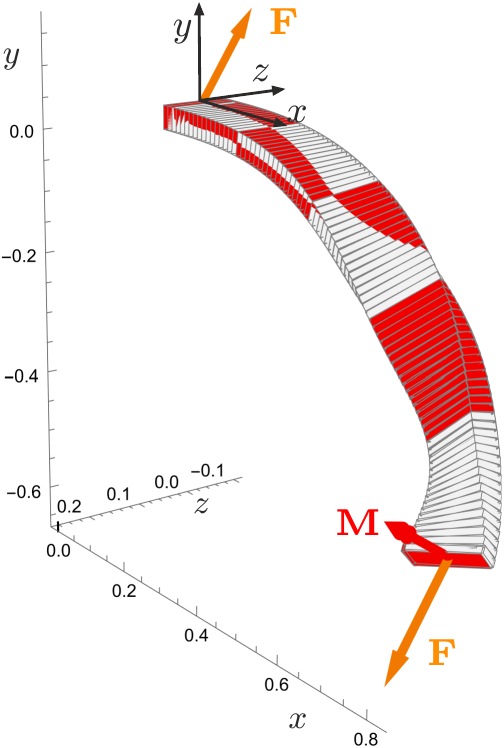}
\end{center}
\caption{%
%
Shape of flexible rod with larger variations of the deformation. a)
Numerically exact solution, b) 3rd-order piecewise interpolation, c)
3rd-order piecewise interpolation with prescribed deformation at the five
interpolation points.}
\label{figPOE_FlatRod_Example2}
\end{figure}
\end{example}

\newpage

\section{Conclusion and Outlook}

This paper addressed the interpolation problem on connected Lie groups, with
particular focus on interpolating rigid body motions, i.e. curves in $%
SO\left( 3\right) $ and $SE\left( 3\right) $. A $k$th-order 2-point
interpolation was derived from a truncated series solution of the Poisson
equation. The minimum acceleration and minimum energy curves w.r.t. a left
invariant metric are recovered as spacial cases. Thereupon, $k$th-order $%
C^{k-1}$ continuos POE spline interpolation schemes can be constructed. A
3rd- and 4th-order interpolation scheme was presented. The 3rd-order spline
algorithm is equivalent to that proposed in \cite%
{ParkRavani1997,KangPark1998}. 
%
An important contribution of this paper is the extension of these algorithms
to additionally allow prescribing a vector field, in addition to the points,
the interpolation curve must match. In case of rigid body motions, this
admits prescribing the velocity at the interpolation points, thus allows
trajectory planning through a given set of states rather than just poses.
These formulations also reduce the sensitivity of all such spline
formulations w.r.t. the initial conditions. Interpolation of frame curves in
Euclidean space has various further applications. One such applications is
the shape reconstruction of geometric non-linear flexible rods modeled using
Cosserat theory. This leads to an interpolation problem on $SE\left(
3\right) $. The main benefit of the proposed formulation is that it does not
explicitly rely on geodesics. The reported results confirm the potential for
such applications.

%
It was pointed out that a $k$th-order POE spline cannot exactly represent a $%
k$th-order motion on $G$, i.e. a curve described by canonical coordinates
that are a $k$th-order polynomials in the path parameter. To overcome this
problem, globally parameterized splines were introduced. This rests on a
2-point interpolation with non-zero initial values of the coordinates, in
contrast to existing interpolation schemes. The global parameterization is
particularly relevant when a polynomial curve in $G$ is to be interpolated,
which is a rather special situation, however. The global splines are also
important when the initial value of the exponential coordinates is non-zero.
This happens in case of continuum elements in soft robots, for instance,
where the initial $h\left( 0\right) \neq h_{0}$ and terminal pose $h\left(
1\right) $ are computed from an elastic equilibrium.

Future work will investigate the interpolation on subspaces of $G$ that are
not subgroups. An important special case are symmetric spaces \cite%
{WuCarricato2020,CarricatoMMT2025}. To this end, the higher-order solutions
must be derived taking into account particular generators of the symmetric
space.

\bmsection*{Acknowledgments} This work was supported by the LCM-K2 Center
within the framework of the Austrian COMET-K2 program.

\bmsection*{Conflict of interest}

The authors declare no potential conflict of interests.

\appendix

\section{\label{secAppendix}Closed Form Expressions for the Lie Groups $%
SO\left( 3\right) $ and $SE\left( 3\right) $}

In the following, the closed form relations are summarized, which can be
found in \cite{RSPA2021}, and partially in \cite{MurrayBook,SeligBook}. To
be consistent with the notation of this paper, Lie group elements are
denoted with $g$ and elements of the Lie algebra with $\boldsymbol{\xi }$.

The Lie group $SO\left( 3\right) $ is the group of orthogonal $3\times 3$
matrices with determinate equal to 1, referred to as rotation matrices. A
rotation matrix $g\in SO\left( 3\right) $ is expressed in terms of a unit
vector $\mathbf{n}\in {\mathbb{R}}^{3}$ along the rotation axis and rotation
angle $\varphi $ via the exponential map%
\begin{equation}
\exp \tilde{\boldsymbol{\xi }}=\mathbf{I}+\alpha \tilde{\boldsymbol{\xi }}+%
\tfrac{1}{2}\beta \tilde{\boldsymbol{\xi }}^{2}  \label{expSO3}
\end{equation}%
where $\boldsymbol{\xi }:=\varphi \mathbf{n},\alpha :=\mathrm{sinc}\varphi
,\beta :=\mathrm{sinc}^{2}\mathrm{\,}\frac{\varphi }{2}$, and $\tilde{%
\boldsymbol{\xi }}\in so\left( 3\right) $ is the skew symmetric matrix
associated with $\boldsymbol{\xi }\in {\mathbb{R}}^{3}$. The Lie algebra $%
so\left( 3\right) $ is the vector space of skew symmetric $3\times 3$
matrices. With the unique correspondence of $\boldsymbol{\xi }\in {\mathbb{R}%
}^{3}$ and $\tilde{\boldsymbol{\xi }}\in so\left( 3\right) $, the notation $%
\boldsymbol{\xi }$ and $\tilde{\boldsymbol{\xi }}$ are used interchangeably.
The right-trivialized differential of (\ref{expSO3}), and its inverse admit
the closed form expressions%
\begin{eqnarray}
\mathbf{dexp}_{\boldsymbol{\xi }} &=&\mathbf{I}+\tfrac{\beta }{2}\tilde{%
\boldsymbol{\xi }}+\tfrac{1}{\left\Vert \mathbf{x}\right\Vert ^{2}}\left(
1-\alpha \right) \tilde{\boldsymbol{\xi }}^{2}=\mathbf{I}+\tfrac{\beta }{2}%
\tilde{\boldsymbol{\xi }}+\left( 1-\alpha \right) \tilde{\mathbf{n}}^{2}
\label{dexpSO3} \\
\mathbf{dexp}_{\boldsymbol{\xi }}^{-1} &=&\mathbf{I}-\tfrac{1}{2}\tilde{%
\boldsymbol{\xi }}+\tfrac{1}{\left\Vert \mathbf{x}\right\Vert ^{2}}\left(
1-\gamma \right) \tilde{\boldsymbol{\xi }}^{2}=\mathbf{I}-\tfrac{1}{2}\tilde{%
\boldsymbol{\xi }}+\left( 1-\gamma \right) \tilde{\mathbf{n}}^{2}
\label{SO3dexpInv1}
\end{eqnarray}%
where $\gamma :=\frac{\alpha }{\beta }$, $\delta :=\frac{1-\alpha }{\varphi
^{2}}$. The Lie bracket on $so\left( 3\right) $ is given by the cross
product, $\left[ \boldsymbol{\xi }_{1},\boldsymbol{\xi }_{2}\right] =\mathbf{%
ad}_{\boldsymbol{\xi }_{1}}\boldsymbol{\xi }_{2}=\tilde{\boldsymbol{\xi }}%
_{1}\boldsymbol{\xi }_{2}=\boldsymbol{\xi }_{1}\times \boldsymbol{\xi }_{2}$%
, for $\boldsymbol{\xi }_{1},\boldsymbol{\xi }_{2}\in {\mathbb{R}}^{3}\cong
so\left( 3\right) $.

The Lie group $SE\left( 3\right) $ is the group of isometric and orientation
preserving transformations in $E^{3}$ \cite{SeligBook}. In matrix
representation, a typical element is%
\begin{equation}
g=%
\begin{bmatrix}
\mathbf{R} & \mathbf{r} \\ 
\mathbf{0} & 1%
\end{bmatrix}%
\in SE\left( 3\right)  \label{C}
\end{equation}%
where $\mathbf{R}\in SO\left( 3\right) $, $\mathbf{r}\in {\mathbb{R}}^{3}$.
The Lie algebra $se\left( 3\right) $ is the vector space of matrices%
\begin{equation}
\hat{\boldsymbol{\xi }}=%
\begin{bmatrix}
\tilde{\mathbf{x}} & \mathbf{y} \\ 
\mathbf{0} & 0%
\end{bmatrix}%
\in se\left( 3\right)  \label{se3}
\end{equation}%
where $\tilde{\mathbf{x}}\in so\left( 3\right) $, $\mathbf{y}\in {\mathbb{R}}%
^{3}$. The exp map can be expressed in closed form as, with $\mathbf{dexp}_{%
\mathbf{x}}$ in (\ref{dexpSO3}),%
\begin{equation}
\exp (\hat{\boldsymbol{\xi }})=%
\begin{bmatrix}
\exp \tilde{\mathbf{x}} & \mathbf{dexp}_{\mathbf{x}}\mathbf{y} \\ 
\mathbf{0} & 0%
\end{bmatrix}%
.  \label{expSE3}
\end{equation}%
The vector space ${\mathbb{R}}^{6}$ is isomorphic to $se\left( 3\right) $ by
identifying $\hat{\boldsymbol{\xi }}\in se\left( 3\right) $ with $%
\boldsymbol{\xi }=\left( \mathbf{x},\mathbf{y}\right) \in {\mathbb{R}}^{6}$.
Therefore, the exponential map will likewise be written as $\exp (%
\boldsymbol{\xi })$ with $\boldsymbol{\xi }\in {\mathbb{R}}^{6}$. The vector 
$\boldsymbol{\xi }$ is the screw coordinate vector that describes a screw
motion given by $\exp (\boldsymbol{\xi })\in SE\left( 3\right) $. The Lie
bracket on $se\left( 3\right) $ is the screw product, $\left[ \boldsymbol{%
\xi }_{1},\boldsymbol{\xi }_{2}\right] =\mathbf{ad}_{\boldsymbol{\xi }_{1}}%
\boldsymbol{\xi }_{2}=\left( \mathbf{x}_{1}\times \mathbf{x}_{2},\mathbf{x}%
_{1}\times \mathbf{y}_{2}+\mathbf{y}_{1}\times \mathbf{x}_{2}\right) $, for $%
\boldsymbol{\xi }_{i}=\left( \mathbf{x}_{i},\mathbf{y}_{i}\right) \in {%
\mathbb{R}}^{6}\cong se\left( 3\right) $. The right-trivialized differential
of (\ref{expSE3}), and its inverse, in terms of canonical coordinates $%
\boldsymbol{\xi }=\left( \mathbf{x},\mathbf{y}\right) \in {\mathbb{R}}^{6}$,
are in closed form%
\begin{eqnarray}
\mathbf{dexp}_{\boldsymbol{\xi }} &=&%
\begin{bmatrix}
\mathbf{dexp}_{\mathbf{x}} & \mathbf{0} \\ 
\left( \mathrm{D}_{\mathbf{x}}\mathbf{dexp}\right) 
\hspace{-0.5ex}%
\left( \mathbf{y}\right) & \mathbf{dexp}_{\mathbf{x}}%
\end{bmatrix}
\label{dexpSE31} \\
\mathbf{dexp}_{\boldsymbol{\xi }}^{-1} &=&%
\begin{bmatrix}
\mathbf{dexp}_{\mathbf{x}}^{-1} & \mathbf{0} \\ 
(\mathrm{D}_{\mathbf{x}}\mathbf{dexp}^{-1})%
\hspace{-0.6ex}%
\left( \mathbf{y}\right) & \mathbf{dexp}_{\mathbf{x}}^{-1}%
\end{bmatrix}
\label{DexpInvSE3}
\end{eqnarray}%
where $\mathbf{dexp}_{\mathbf{x}}$ and $\mathbf{dexp}_{\mathbf{x}}^{-1}$ are
as in (\ref{dexpSO3}) and (\ref{SO3dexpInv1}), and%
\begin{eqnarray}
\left( \mathrm{D}_{\mathbf{x}}\mathbf{dexp}\right) 
\hspace{-0.5ex}%
\left( \mathbf{y}\right) &=&\tfrac{\beta }{2}\tilde{\mathbf{y}}+\delta
\left( \tilde{\mathbf{x}}\tilde{\mathbf{y}}+\tilde{\mathbf{y}}\tilde{\mathbf{%
x}}\right) +\tfrac{\mathbf{x}^{T}\mathbf{y}}{\left\Vert \mathbf{x}%
\right\Vert ^{2}}\left( \alpha -\beta \right) \tilde{\mathbf{x}}+\tfrac{%
\mathbf{x}^{T}\mathbf{y}}{\left\Vert \mathbf{x}\right\Vert ^{2}}\left( 
\tfrac{\beta }{2}-3\delta \right) \tilde{\mathbf{x}}^{2}
\label{diffDexpSO33} \\
(\mathrm{D}_{\mathbf{x}}\mathbf{dexp}^{-1})%
\hspace{-0.5ex}%
\left( \mathbf{y}\right) &=&-\tfrac{1}{2}\tilde{\mathbf{y}}+\frac{1}{%
\left\Vert \mathbf{x}\right\Vert ^{2}}\left( 1-\gamma \right) \left( \tilde{%
\mathbf{x}}\tilde{\mathbf{y}}+\tilde{\mathbf{y}}\tilde{\mathbf{x}}\right) +%
\tfrac{\mathbf{x}^{T}\mathbf{y}}{\left\Vert \mathbf{x}\right\Vert ^{4}}%
\left( \tfrac{1}{\beta }+\gamma -2\right) \tilde{\mathbf{x}}^{2}.
\label{diffDexpInvSO3}
\end{eqnarray}%
The directional derivatives of (\ref{dexpSE31}) and (\ref{DexpInvSE3}) are%
\begin{eqnarray}
(\mathrm{D}_{\boldsymbol{\xi }}\mathbf{dexp})%
\hspace{-0.5ex}%
\left( \boldsymbol{\eta }\right) &=&%
\begin{bmatrix}
(\mathrm{D}_{\mathbf{x}}\mathbf{dexp})%
\hspace{-0.5ex}%
\left( \mathbf{u}\right) & \mathbf{0} \\ 
(\mathrm{D}_{\boldsymbol{\xi }}\mathbf{Ddexp})%
\hspace{-0.5ex}%
\left( \boldsymbol{\eta }\right) & (\mathrm{D}_{\mathbf{x}}\mathbf{dexp})%
\hspace{-0.5ex}%
\left( \mathbf{u}\right)%
\end{bmatrix}
\label{diffDexpSE3} \\
(\mathrm{D}_{\boldsymbol{\xi }}\mathbf{dexp}^{-1})%
\hspace{-0.5ex}%
\left( \boldsymbol{\eta }\right) &=&%
\begin{bmatrix}
(\mathrm{D}_{\mathbf{x}}\mathbf{dexp}^{-1})%
\hspace{-0.5ex}%
\left( \mathbf{u}\right) & \mathbf{0} \\ 
(\mathrm{D}_{\boldsymbol{\xi }}\mathbf{Ddexp}^{-1})%
\hspace{-0.5ex}%
\left( \boldsymbol{\eta }\right) & (\mathrm{D}_{\mathbf{x}}\mathbf{dexp}%
^{-1})%
\hspace{-0.5ex}%
\left( \mathbf{u}\right)%
\end{bmatrix}
\label{diffDexpInvSE3}
\end{eqnarray}%
where $\boldsymbol{\eta }=\left[ \mathbf{u},\mathbf{v}\right] $. Therein the
directional derivatives of matrices $\mathbf{Ddexp}$ and $\mathbf{Ddexp}%
^{-1} $ possesses the explicitly form%
\begin{align}
(\mathrm{D}_{\boldsymbol{\xi }}\mathbf{Ddexp})%
\hspace{-0.5ex}%
\left( \boldsymbol{\eta }\right) =~& \tfrac{\beta }{2}\tilde{\mathbf{v}}+%
\tfrac{\alpha -\beta }{\left\Vert \mathbf{x}\right\Vert ^{2}}(\mathbf{x}^{T}%
\mathbf{u})\tilde{\mathbf{y}}+\tfrac{\alpha -\beta }{\left\Vert \mathbf{x}%
\right\Vert ^{2}}(\mathbf{x}^{T}\mathbf{y})\tilde{\mathbf{u}}+\tfrac{\beta
/2-3\delta }{\left\Vert \mathbf{x}\right\Vert ^{2}}\left( (\mathbf{x}^{T}%
\mathbf{u})(\tilde{\mathbf{x}}\tilde{\mathbf{y}}+\tilde{\mathbf{y}}\tilde{%
\mathbf{x}})+(\mathbf{x}^{T}\mathbf{y})(\tilde{\mathbf{x}}\tilde{\mathbf{u}}+%
\tilde{\mathbf{u}}\tilde{\mathbf{x}})\right) +\delta \left( \tilde{\mathbf{x}%
}\tilde{\mathbf{v}}+\tilde{\mathbf{v}}\tilde{\mathbf{x}}+\tilde{\mathbf{y}}%
\tilde{\mathbf{u}}+\tilde{\mathbf{u}}\tilde{\mathbf{y}}\right)
\label{diff2dexpSO31} \\
& +\tfrac{1}{\left\Vert \mathbf{x}\right\Vert ^{2}}\left( \left( \alpha
-\beta \right) (\mathbf{x}^{T}\mathbf{v+y}^{T}\mathbf{u})-\tfrac{1}{2}\beta (%
\mathbf{x}^{T}\mathbf{u})(\mathbf{x}^{T}\mathbf{y})+\tfrac{1-5\alpha +4\beta 
}{\left\Vert \mathbf{x}\right\Vert ^{2}}(\mathbf{x}^{T}\mathbf{y})(\mathbf{x}%
^{T}\mathbf{u})\right) \tilde{\mathbf{x}}  \notag \\
& +\tfrac{1}{\left\Vert \mathbf{x}\right\Vert ^{2}}\left( (\tfrac{1}{2}\beta
-3\delta )(\mathbf{x}^{T}\mathbf{v+y}^{T}\mathbf{u})+\tfrac{1}{\left\Vert 
\mathbf{x}\right\Vert ^{2}}(\alpha -\tfrac{7}{2}\beta +15\delta )(\mathbf{x}%
^{T}\mathbf{y})(\mathbf{x}^{T}\mathbf{u})\right) \tilde{\mathbf{x}}^{2}. 
\notag \\
(\mathrm{D}_{\boldsymbol{\xi }}\mathbf{Ddexp}^{-1})%
\hspace{-0.5ex}%
\left( \mathbf{U}\right) =& -\tfrac{1}{2}\tilde{\mathbf{v}}+\tfrac{1}{%
\left\Vert \mathbf{x}\right\Vert ^{2}}\left( \tfrac{1}{4}(\mathbf{x}^{T}%
\mathbf{u})(\tilde{\mathbf{x}}\tilde{\mathbf{y}}+\tilde{\mathbf{y}}\tilde{%
\mathbf{x}})+\left( 1-\gamma \right) \left( \tilde{\mathbf{x}}\tilde{\mathbf{%
v}}+\tilde{\mathbf{v}}\tilde{\mathbf{x}}+\tilde{\mathbf{y}}\tilde{\mathbf{u}}%
+\tilde{\mathbf{u}}\tilde{\mathbf{y}}\right) \right)
\label{diff2dexpInvSO31} \\
& +\tfrac{1}{\left\Vert \mathbf{x}\right\Vert ^{4}}%
\Big%
((\mathbf{x}^{T}\mathbf{u})\left( \gamma \left( 1+\gamma \right) -2\right) (%
\tilde{\mathbf{x}}\tilde{\mathbf{y}}+\tilde{\mathbf{y}}\tilde{\mathbf{x}})+(%
\tfrac{1}{\beta }+\gamma -2)(\mathbf{x}^{T}\mathbf{y})(\tilde{\mathbf{x}}%
\tilde{\mathbf{u}}+\tilde{\mathbf{u}}\tilde{\mathbf{x}})  \notag \\
& +%
\big%
((\tfrac{1}{\beta }+\gamma -2)(\mathbf{x}^{T}\mathbf{v}+\mathbf{y}^{T}%
\mathbf{u})-\tfrac{1}{4}(\mathbf{x}^{T}\mathbf{y})(\mathbf{x}^{T}\mathbf{u})%
\big%
)\tilde{\mathbf{x}}^{2}%
\Big%
)  \notag \\
& +\tfrac{1}{\left\Vert \mathbf{x}\right\Vert ^{6}}(\mathbf{x}^{T}\mathbf{y}%
)(\mathbf{x}^{T}\mathbf{u})\left( 8-\gamma \left( 3+\gamma \right) -\tfrac{2%
}{\beta }\left( 1+\gamma \right) \right) \tilde{\mathbf{x}}^{2}.  \notag
\end{align}%
The log map on $SE(3)$ is and $SO(3)$%
\begin{equation}
h=%
\begin{bmatrix}
\mathbf{R} & \mathbf{r} \\ 
\mathbf{0} & 1%
\end{bmatrix}%
\in SE\left( 3\right) \mapsto \hat{\boldsymbol{\xi }}=\log \left( h\right) =%
\begin{bmatrix}
\tilde{\mathbf{x}} & \mathbf{dexp}_{\tilde{\mathbf{x}}}^{-1}\mathbf{r} \\ 
\mathbf{0} & 1%
\end{bmatrix}%
\in se\left( 3\right)  \label{logSO3}
\end{equation}%
or in vector notation $\boldsymbol{\xi }=[\mathbf{x},\mathbf{dexp}_{\mathbf{x%
}}^{-1}\mathbf{r}]\in {\mathbb{R}}^{6}$, where $\mathbf{x}\in {\mathbb{R}}%
^{3}$ is obtained with the log on $SO\left( 3\right) $ 
\begin{equation}
\mathbf{R}\in SO\left( 3\right) \mapsto \tilde{\mathbf{x}}=\log \mathbf{R}=%
\frac{1}{2\sin \varphi }\left( \mathbf{R}-\mathbf{R}^{T}\right) \in so\left(
3\right)  \label{logSE3}
\end{equation}%
with angle $\varphi =\frac{1}{2}\arccos \left( \mathrm{tr}\left( \mathbf{R}%
\right) -1\right) $.%
%

\section{\label{secDeCast}De Casteljau Algorithm on Lie Groups}

An easy and flexible way to represent 3-dimensional curves in $\mathbb{R}%
^{3} $ (and generally on any metric vector space) in terms of control
polygons are B\'{e}zier curves \cite{Beizier1986}. The classical De
Casteljau algorithm \cite{deCasteljau1959,KrautterParizot1971} is a
computational way to generate the control polygon from a given set of
control points \cite{Farin1997}. On curved metric spaces, particularly on
Riemannian manifolds, including symmetric spaces, and Lie groups, the
definition of 'control polygons' must respect the metric structure of the
space. Ideally, B\'{e}zier curves should be geodesics w.r.t. a certain
metric with certain invariance. This complicates the derivation of De
Casteljau type algorithms. There is a body of literature on polynomial
curves, B\'{e}zier curves, and B-splines on (symmetric) Riemannian manifolds
and Lie groups \cite%
{CrouchLeite1995,CrouchKunLeite1999,Camarinha2001,Popiel2006,BogfjellmoModinVerdier2018,ZhangNoakes2019,WuCarricato2020,SprottRavani2002,Ding2002}%
. The crucial differences when designing a De Casteljau algorithm on curved
spaces is that the subdivision properties and smoothness results of the B%
\'{e}zier curves are not generally inherited from the algorithm on vector
spaces \cite{Popiel2006}.

For the special case of compact Lie groups, and particularly for $SO(3)$ and
the quaternion group, De Casteljau algorithms where introduced in \cite%
{ParkRavani-JMD1995}. This was extended in \cite{GeRavani-JMD1994} to
spatial motions described by dual quaternions, i.e. the double cover of $%
SE\left( 3\right) $, that relies on geodesics w.r.t. a left-invariant
metric. The algorithm 9 presented below is an adapted version of \cite%
{ParkRavani-JMD1995}. It rests on the assumption that, at the different
steps of the algorithm, curves between points are expressed as motion in a
1-parameter subgroup. For the compact rotation group $SO\left( 3\right) $
these are geodesics, whereas for the non-compact rigid body motion group $%
SE\left( 3\right) $ these are screw motions (not geodesics). Possible
refinements can be found in \cite{Popiel2006,ZhangNoakes2019}. Algorithm 9
constructs a B\'{e}zier curve from $n+1$ given points points $h_{0},\ldots
,h_{n}$ in a Lie group $G$, on which an exp map is defined. The curve
interpolates between terminal points $h_{0}$ and $h_{n}$, while the
remaining points serve as 'control points'. Intermediate points $h_{i}^{k}$
are computed recursively. The B\'{e}zier curve of order $n$ is then obtained
as $h_{0}^{n}\left( \tau \right) $. At every $\tau $, its evaluation
requires $n$ nested evaluations of the log and exp map, which dictates the
computational effort. The algorithm complexity is quadratic in the number $n$
of points. Each $h_{i}^{k}$ depends on $\tau $, and symbolic construction of
the B\'{e}zier curve is generally not possible, even when closed form
expressions for exp and log are available. The algorithm is used to evaluate
the curve at given $\tau $. The algorithm assumes zero initial and terminal
velocities. The case for non-zero boundary velocities is easily
incorporated. Details can be found in \cite{CrouchKunLeite1999}.

\begin{tabular}[t]{l}
\  \\ 
\textbf{Algorithm 9: De Casteljau Algorithm on Lie Group }$G$%
\vspace{1.5ex}
\\ 
\begin{tabular}{l}
\begin{tabular}[t]{l}
\textbf{1. Input:}%
\vspace{0.5ex}
\\ 
\hspace{2ex}%
\begin{tabular}{ll}
\textbullet
& Samples $h_{0},h_{1},\ldots ,h_{n}\in G$ \\ 
\textbullet
& Parameter value $\tau \in \left[ 0,1\right] $%
\end{tabular}%
\end{tabular}
\\ 
\begin{tabular}[t]{l}
\textbf{2. Recursive Computation of 'Control Points':} \\ 
\hspace{2ex}%
\textbf{FOR} $k=1,\ldots ,n$ \textbf{DO}%
\vspace{0.5ex}
\\ 
\hspace{2ex}%
\begin{tabular}[t]{l}
\textbf{FOR} $i=0,\ldots ,n-k$ \textbf{DO}%
\vspace{0.5ex}
\\ 
\begin{tabular}[t]{ll}
\textbullet
& $\boldsymbol{\xi }_{i}^{k}\left( \tau \right) :=\log \left(
(h_{i}^{k-1})^{-1}h_{i+1}^{k-1}\right) $ \\ 
\textbullet
& $h_{i}^{k}\left( \tau \right) :=h_{i}^{k-1}\left( \tau \right) \exp (\tau 
\boldsymbol{\xi }_{i}^{k}\left( \tau \right) )$%
\end{tabular}
\\ 
\textbf{END}%
\end{tabular}%
\vspace{0.5ex}
\\ 
\ \ \ \ \textbf{END}%
\end{tabular}%
\hspace{-5ex}
\\ 
\begin{tabular}[t]{l}
\textbf{3. Value of B\'{e}zier curve at} $\tau \in \left[ 0,1\right] $: $%
h_{0}^{n}\left( \tau \right) $%
\vspace{0.5ex}
\\ 
\end{tabular}%
\end{tabular}
\\ 
\ 
\end{tabular}

%
\bibliographystyle{spmpsci}
\bibliography{MotionInterpolation2}

\begin{thebibliography}{10}
\providecommand{\url}[1]{{#1}}
\providecommand{\urlprefix}{URL }
\expandafter\ifx\csname urlstyle\endcsname\relax
  \providecommand{\doi}[1]{DOI~\discretionary{}{}{}#1}\else
  \providecommand{\doi}{DOI~\discretionary{}{}{}\begingroup
  \urlstyle{rm}\Url}\fi

\bibitem{AlmaghoutCherubiniKlimchik2024}
Almaghout, K., Cherubini, A., Klimchik, A.: Robotic co-manipulation of
  deformable linear objects for large deformation tasks.
\newblock Robotics and Autonomous Systems  (2024).
\newblock \doi{10.1016/j.robot.2024.104652}

\bibitem{AlmaghoutKlimchik2024}
Almaghout, K., Klimchik, A.: Manipulation planning for cable shape control.
\newblock Robtics \textbf{13}(18), 20 (2024)

\bibitem{AltmannBook1986}
Altmann, S.: {Rotations, Quaternions, and Double Groups}.
\newblock Clarendon Press (1989)

\bibitem{BartovnJuettlerWang2010}
Barto{\v{n}}, M., J{\"u}ttler, B., Wang, W.: Construction of rational curves
  with rational rotation-minimizing frames via m{\"o}bius transformations.
\newblock In: Mathematical Methods for Curves and Surfaces: 7th International
  Conference, MMCS 2008, T{\o}nsberg, Norway, June 26-July 1, 2008, Revised
  Selected Papers 7, pp. 15--25. Springer (2010)

\bibitem{BauchauEppleHeo2008}
Bauchau, O., Epple, A., Heo, S.: Interpolation of finite rotations in flexible
  multi-body dynamics simulations.
\newblock Proceedings of the Institution of Mechanical Engineers, Part K:
  Journal of Multi-body Dynamics \textbf{222}(4), 353--366 (2008)

\bibitem{Beizier1986}
B\'{e}zier, P.: {The Mathematical Basis of UNISURF CAD System}.
\newblock Butterworths, London (1986)

\bibitem{BlackTillRucker2018}
Black, C., Till J.and~Rucker, D.: {Parallel Continuum Robots}.
\newblock IEEE Transactions on Robotics \textbf{34}(1), 29--47 (2018)

\bibitem{BlochBook2003}
Bloch, A.M.: Nonholonomic mechanics and control.
\newblock Springer (2003)

\bibitem{BogfjellmoModinVerdier2018}
Bogfjellmo, G., Modin, K., Verdier, O.: {A Numerical Algorithm for
  $C^2$-Splines on Symmetric Spaces}.
\newblock SIAM Journal on Numerical Analysis \textbf{56}(4), 2623--2647 (2018)

\bibitem{BottemaRoth1979}
Bottema, O., Roth, B.: {Theoretical Kinematics}.
\newblock North-Holland (1979)

\bibitem{BriotBoyer2022}
Briot, S., Boyer, F.: A geometrically exact assumed strain modes approach for
  the geometrico-and kinemato-static modelings of continuum parallel robots.
\newblock IEEE Transactions on Robotics \textbf{39}(2), 1527--1543 (2022)

\bibitem{Camarinha2001}
Camarinha, M., Leite, F.S., Crouch, P.: {On the geometry of Riemannian cubic
  polynomials}.
\newblock Differential Geometry and its Applications \textbf{15}(2), 107--135
  (2001)

\bibitem{CarricatoMMT2025}
Carricato, M.: Classification of the three-dimensional persistent poe manifolds
  of se (3).
\newblock Mechanism and Machine Theory \textbf{206}, 105926 (2025)

\bibitem{deCasteljau1959}
de~Casteljau, P.: Enveloppe soleau 40.040.
\newblock Document no P \textbf{2108} (1959)

\bibitem{ConduracheAASAIAA2017}
Condurache, D.: {Poisson-Darboux problems's extended in dual Lie algebra}.
\newblock In: AAS/AIAA Astrodynamics Specialist Conference, Stevenson, WA, USA
  (2017)

\bibitem{CrouchKunLeite1999}
Crouch, P., Kun, G., Leite, F.S.: {The De Casteljau algorithm on Lie groups and
  spheres}.
\newblock Journal of Dynamical and Control Systems \textbf{5}(3), 397--429
  (1999)

\bibitem{CrouchLeite1991}
Crouch, P., Leite, F.S.: Geometry and the dynamic interpolation problem.
\newblock In: 1991 American Control Conference, pp. 1131--1136. IEEE (1991)

\bibitem{CrouchLeite1995}
Crouch, P., Leite, F.S.: {The dynamic interpolation problem: on Riemannian
  manifolds, Lie groups, and symmetric spaces}.
\newblock Journal of Dynamical and control systems \textbf{1}(2), 177--202
  (1995)

\bibitem{Darboux1887}
Darboux, G.: {Le\c{c}ons sur la th\'{e}orie g\'{e}n\'{e}rale des surfaces et
  les applications g\'{e}om\'{e}triques du calcul infinitesimal}.
\newblock Gautiers-Villars, Paris \textbf{4} (1887)

\bibitem{Dhullipalla2019}
Dhullipalla, M.H., Hamrah, R., Warier, R.R., Sanyal, A.K.: {Trajectory
  generation on SE(3) for an underactuated vehicle with pointing direction
  constraints}.
\newblock In: 2019 American Control Conference (ACC), pp. 1930--1935. IEEE
  (2019)

\bibitem{Ding2002}
Ding, R.: {Drawing ruled surfaces using the dual De Boor algorithm}.
\newblock Electronic Notes in Theoretical Computer Science \textbf{61},
  178--190 (2002)

\bibitem{DuffyARK1991}
Duffy, J., Griffis, M., Swinson, M.: The fallacy of modern hybrid control
  theory for the simultaneous control of force and motion.
\newblock In: S.~Stifter, J.~Lenar{\v{c}}i{\v{c}} (eds.) Advances in Robot
  Kinematics, pp. 248--258. Springer Vienna, Vienna (1991)

\bibitem{Farin1997}
Farin, G.E.: {Curves and surfaces for computer-aided design: A practical
  guide}.
\newblock Academic Press (1997)

\bibitem{GeRavani-JMD1994-2}
Ge, Q., Ravani, B.: Computer aided geometric design of motion interpolants.
\newblock Journal of Mechanical Design \textbf{116}(3), 756--762 (1994)

\bibitem{GeRavani-JMD1994}
Ge, Q., Ravani, B.: {Geometric Construction of B{\'{e}}zier Motions}.
\newblock Journal of Mechanical Design \textbf{116}(3), 749--755 (1994)

\bibitem{HairerLubichWanner2006}
Hairer, E., Lubich, C., Wanner, G.: {Geometric Numerical Integration}.
\newblock Springer (2006)

\bibitem{HanBauchau2018}
Han, S., Bauchau, O.A.: On the global interpolation of motion.
\newblock Computer Methods in Applied Mechanics and Engineering \textbf{337},
  352--386 (2018)

\bibitem{Hausdorff1906}
Hausdorff, F.: {Die symbolische {E}xponentialformel in der {G}ruppentheorie}.
\newblock Berichte der K{\"o}niglich-S{\"a}chsischen Geselschaft der
  Wissenschaften zu Leipzig, Mathematisch-Physische Klasse \textbf{58}, 19--48
  (1906)

\bibitem{Huang2020}
Huang, Y., Abu-Dakka, F.J., Silv{\'e}rio, J., Caldwell, D.G.: Toward
  orientation learning and adaptation in cartesian space.
\newblock IEEE Transactions on Robotics \textbf{37}(1), 82--98 (2020)

\bibitem{Huber2017}
Huber, P., Perl, R., Rumpf, M.: Smooth interpolation of key frames in a
  riemannian shell space.
\newblock Computer Aided Geometric Design \textbf{52}, 313--328 (2017)

\bibitem{Hussain2021}
Hussain, I., Malvezzi, M., Gan, D., Iqbal, Z., Seneviratne, L., Prattichizzo,
  D., Renda, F.: Compliant gripper design, prototyping, and modeling using
  screw theory formulation.
\newblock The International Journal of Robotics Research \textbf{40}(1), 55--71
  (2021)

\bibitem{HustySchroecker2010}
Husty, M.L., Schr\"ocker, H.P.: Algebraic geometry and kinematics.
\newblock In: I.Z. Emiris, F.~Sottile, T.~Theobald (eds.) Nonlinear
  Computational Geometry, \emph{The IMA Volumes in Mathematics and its
  Applications}, vol. 151, chap. Algebraic Geometry and Kinematics, pp.
  85--107. Springer (2010)

\bibitem{Iserles1984}
Iserles, A.: Solving linear ordinary differential equations by exponentials of
  iterated commutators.
\newblock Numerische Mathematik \textbf{45}(2), 183--199 (1984)

\bibitem{IserlesMuntheKaasNrsettZanna2000}
Iserles, A., Munthe-Kaas, H.Z., N{\o}rsett, S.P., Zanna, A.: Lie-group methods.
\newblock Acta Numerica \textbf{9}, 215–365 (2000).
\newblock \doi{10.1017/S0962492900002154}

\bibitem{KangPark1998}
Kang, I., Park, F.: Cubic spline algorithms for orientation interpolation.
\newblock International journal for numerical methods in engineering
  \textbf{46}(1), 45--64 (1999)

\bibitem{Kavan2006}
Kavan, L., Collins, S., O’Sullivan, C., Zara, J.: Dual quaternions for rigid
  transformation blending.
\newblock Trinity College Dublin \textbf{5} (2006)

\bibitem{KimNam1995}
Kim, M.S., Nam, K.W.: Interpolating solid orientations with circular blending
  quaternion curves.
\newblock Computer-Aided Design \textbf{27}(5), 385--398 (1995)

\bibitem{KrautterParizot1971}
Krautter, J., Parizot, S.: Systeme d’aidea la d{\'e}finition eta l’usinage
  des surf{\^a}ces de carosserie.
\newblock Journal de la SIA \textbf{44}, 581--586 (1971)

\bibitem{KryslEndres2005}
Krysl, P., Endres, L.: Explicit newmark/verlet algorithm for time integration
  of the rotational dynamics of rigid bodies.
\newblock Int. J. Numer. Methods Eng. \textbf{62}, 2154--2177 (2005)

\bibitem{Legnani2021}
Legnani, G., Fassi, I., Tasora, A., Fusai, D.: A practical algorithm for smooth
  interpolation between different angular positions.
\newblock Mechanism and Machine Theory \textbf{162}, 104341 (2021)

\bibitem{LiSchichoSchroecker2018}
Li, Z., Schicho, J., Schr{\"o}cker, H.P.: Kempe’s universality theorem for
  rational space curves.
\newblock Foundations of Computational Mathematics \textbf{18}, 509--536 (2018)

\bibitem{LiYe2024}
Li, Z., Ye, K.: Rational curves on real classical groups.
\newblock arXiv preprint arXiv:2408.04453  (2024)

\bibitem{Lovegrove2013}
Lovegrove, S., Patron-Perez, A., Sibley, G.: Spline fusion: A continuous-time
  representation for visual-inertial fusion with application to rolling shutter
  cameras.
\newblock In: BMVC, vol.~2, p.~8 (2013)

\bibitem{Magnus1954}
Magnus, W.: On the exponential solution of differential equations for a linear
  operator.
\newblock Communications on pure and applied mathematics \textbf{7}(4),
  649--673 (1954)

\bibitem{MarsdenBook1995}
Marsden, J.E., Ratiu, T.S.: Introduction to mechanics and symmetry.
\newblock Physics Today \textbf{48}(12), 65 (1995)

\bibitem{Marthinsen1999}
Marthinsen, A.: Interpolation in lie groups.
\newblock SIAM Journal on Numerical Analysis \textbf{37}(1), 269--285 (1999)

\bibitem{ZAMM2010}
M{\"u}ller, A.: Approximation of finite rigid body motions from velocity
  fields.
\newblock ZAMM-Journal of Applied Mathematics and Mechanics/Zeitschrift f{\"u}r
  Angewandte Mathematik und Mechanik: Applied Mathematics and Mechanics
  \textbf{90}(6), 514--521 (2010)

\bibitem{MUBOScrew1}
M{\"{u}}ller, A.: {Screw and Lie group theory in multibody dynamics --Motion
  representation and recursive kinematics of tree-topology systems}.
\newblock Multibody System Dynamics \textbf{43}(1), 1--34 (2018).
\newblock \doi{10.1007/s11044-017-9582-7}

\bibitem{RSPA2021}
M{\"u}ller, A.: {Review of the exponential and Cayley map on SE (3) as relevant
  for Lie group integration of the generalized Poisson equation and flexible
  multibody systems}.
\newblock Proceedings of the Royal Society A \textbf{477}(2253) (2021)

\bibitem{JMR2024}
M\"{u}ller, A.: {Fourth-Order Accurate Strain-Parameterized Shape
  Representation of Beam Elements for Modeling Continuum Robots and Robotic
  Manipulation of Slender Objects}.
\newblock ASME Journal of Mechanisms and Robotics \textbf{17} (2024)

\bibitem{MuntheKaas-BIT1998}
Munthe-Kaas, H.: {Runge-Kutta methods on Lie groups}.
\newblock BIT Numerical Mathematics \textbf{38}(1), 92--111 (1998)

\bibitem{MuntheKaas1999}
Munthe-Kaas, H.: {High order Runge-Kutta methods on manifolds}.
\newblock Applied Numerical Mathematics \textbf{29}(1), 115--127 (1999)

\bibitem{MurrayBook}
Murray, R., Li, Z., Sastry, S.: {A Mathematical Introduction to Robotic
  Manipulation}.
\newblock CRC Press (1994)

\bibitem{Noakes1989}
Noakes, L., Heinzinger, G., Paden, B.: Cubic splines on curved spaces.
\newblock IMA Journal of Mathematical Control and Information \textbf{6}(4),
  465--473 (1989)

\bibitem{ParkRavani-JMD1995}
Park, F., Ravani, B.: {B{\'{e}}zier curves on Riemannian manifolds and Lie
  groups with kinematics applications}.
\newblock Journal of Mechanical Design \textbf{117}(1), 36--40 (1995)

\bibitem{ParkRavani1997}
Park, F.C., Ravani, B.: Smooth invariant interpolation of rotations.
\newblock ACM Transactions on Graphics (TOG) \textbf{16}(3), 277--295 (1997)

\bibitem{Popiel2006}
Popiel, T.: On parametric smoothness of generalised b-spline curves.
\newblock Computer aided geometric design \textbf{23}(8), 655--668 (2006)

\bibitem{RavaniMeghdari2004}
Ravani, R., Meghdari, A.: Spatial rational motions based on rational
  frenet-serret curves.
\newblock In: 2004 IEEE International Conference on Systems, Man and
  Cybernetics (IEEE Cat. No. 04CH37583), vol.~5, pp. 4456--4461. IEEE (2004)

\bibitem{Ren2025}
Ren, K., Yuan, T., Liu, J.: Geometrically exact beam finite element with
  generalized b-spline interpolation on the special euclidean group se (3).
\newblock Computer Methods in Applied Mechanics and Engineering \textbf{441},
  117979 (2025)

\bibitem{RuckerJonesWebster2010}
Rucker, D.C., Jones, B.A., Webster~III, R.J.: A geometrically exact model for
  externally loaded concentric-tube continuum robots.
\newblock IEEE transactions on robotics \textbf{26}(5), 769--780 (2010)

\bibitem{Sarker2020}
Sarker, A., Sinha, A., Chakraborty, N.: On screw linear interpolation for
  point-to-point path planning.
\newblock In: 2020 IEEE/RSJ International Conference on Intelligent Robots and
  Systems (IROS), pp. 9480--9487. IEEE (2020)

\bibitem{SeligBook}
Selig, J.: {Geometric Fundamentals of Robotics}.
\newblock Springer (2005)

\bibitem{Selig-IMA2007}
Selig, J.: {Curves of stationary acceleration in SE (3)}.
\newblock IMA Journal of Mathematical Control and Information \textbf{24}(1),
  95--113 (2007)

\bibitem{Selig2010}
Selig, J.: Rational interpolation of rigid-body motions.
\newblock In: Advances in the Theory of Control, Signals and Systems with
  Physical Modeling, pp. 213--224. Springer (2010)

\bibitem{Shoemake1985}
Shoemake, K.: Animating rotation with quaternion curves.
\newblock In: Proceedings of the 12th annual conference on Computer graphics
  and interactive techniques, pp. 245--254 (1985)

\bibitem{Shoemake1987}
Shoemake, K.: Quaternion calculus and fast animation, computer animation: 3-d
  motion specification and control.
\newblock Siggraph (1987)

\bibitem{SprottRavani2002}
Sprott, K., Ravani, B.: Kinematic generation of ruled surfaces.
\newblock Advances in Computational Mathematics \textbf{17}, 115--133 (2002)

\bibitem{Tagliavini2023}
Tagliavini, A., Bianco, C.G.L.: A smooth orientation planner for trajectories
  in the cartesian space.
\newblock IEEE Robotics and Automation Letters \textbf{8}(5), 2606--2613 (2023)

\bibitem{TanXingFanHong2018}
Tan, J., Xing, Y., Fan, W., Hong, P.: Smooth orientation interpolation using
  parametric quintic-polynomial-based quaternion spline curve.
\newblock Journal of Computational and Applied Mathematics \textbf{329},
  256--267 (2018)

\bibitem{Varadarajan1984}
Varadarajan, V.S.: {Lie groups, Lie algebras, and their representations}.
\newblock Springer Science \& Business Media (2013)

\bibitem{Wang2018}
Wang, H., Totaro, M., Beccai, L.: Toward perceptive soft robots: Progress and
  challenges.
\newblock Advanced Science \textbf{5}(9), 1800541 (2018)

\bibitem{Wittenburg2016}
Wittenburg, J.: Kinematics.
\newblock Berlin, Heidelberg: Springer Berlin Heidelberg \textbf{10}, 978--3
  (2016)

\bibitem{WuCarricato2020}
Wu, Y., Carricato, M.: Persistent manifolds of the special euclidean group se
  (3): A review.
\newblock Computer Aided Geometric Design \textbf{79}, 101872 (2020)

\bibitem{Wu-ARK2016}
Wu, Y., M{\"u}ller, A., Carricato, M.: {The 2D orientation interpolation
  problem: a symmetric space approach}.
\newblock In: Advances in Robot Kinematics 2016, pp. 293--302. Springer (2017)

\bibitem{XiaoJMR2023}
Xiao, Q., Musa, M., Godage, I.S., Su, H., Chen, Y.: Kinematics and stiffness
  modeling of soft robot with a concentric backbone.
\newblock Journal of Mechanisms and Robotics \textbf{15}(5), 051011 (2023)

\bibitem{ZefranKumar1996}
{\v{Z}}efran, M., Kumar, V.: Planning of smooth motions on se (3).
\newblock In: Proceedings of IEEE International Conference on Robotics and
  Automation, vol.~1, pp. 121--126. IEEE (1996)

\bibitem{ZefranKumar1998}
{\v{Z}}efran, M., Kumar, V.: Interpolation schemes for rigid body motions.
\newblock Computer-Aided Design \textbf{30}(3), 179--189 (1998)

\bibitem{ZefranKumarCroke-IJRR1999}
Zefran, M., Kumar, V., Croke, C.: Metrics and connections for rigid-body
  kinematics.
\newblock International Journal of Robotics Research \textbf{18}(2), 243--258
  (1999)

\bibitem{ZefranKumarCroke-TRO1998}
{\v{Z}}efran, M., Kumar, V., Croke, C.B.: On the generation of smooth
  three-dimensional rigid body motions.
\newblock IEEE Transactions on Robotics and Automation \textbf{14}(4), 576--589
  (1998)

\bibitem{ZhangNoakes2019}
Zhang, E., Noakes, L.: {The cubic de Casteljau construction and Riemannian
  cubics}.
\newblock Computer Aided Geometric Design \textbf{75}, 101789 (2019)

\end{thebibliography}

\end{document}